\documentclass[11pt,oneside,leqno]{amsart}
\usepackage{amsxtra}
\usepackage{amsopn}
\usepackage{color}
\usepackage{amsmath,amsthm,amssymb}
\usepackage{amscd}
\usepackage{amsfonts}
\usepackage{latexsym}
\usepackage{verbatim}
\usepackage{graphicx}
\usepackage{graphics}
\usepackage{multirow}
\usepackage{lscape}
\usepackage{nicefrac}
\usepackage{mathabx}
\usepackage[all]{xy}
\usepackage{float}
\usepackage[table]{xcolor}
\usepackage{colortbl}
\usepackage{hhline}

\usepackage[hypertexnames=false,
 backref=UseNone,
    pdftex,
    pdfpagemode=UseNone,
    breaklinks=true,
    extension=pdf,
    colorlinks=true,
    linkcolor=blue,
    citecolor=blue,
    urlcolor=blue,
]
{hyperref}

\theoremstyle{plain}
\newtheorem{theorem}{Theorem}[section]
\newtheorem{definition}[theorem]{Definition}
\newtheorem{lemma}[theorem]{Lemma}
\newtheorem{proposition}[theorem]{Proposition}
\newtheorem{corollary}[theorem]{Corollary}
\newtheorem{remark}[theorem]{Remark}

\newtheorem{example}[theorem]{Example}
\newtheorem{question}[theorem]{Question}
\newtheorem{remark-question}[theorem]{Remark-Question}

\newcommand\C{{\mathbb C}}

\newcommand\Q{{\mathbb Q}}
\newcommand\R{{\mathbb R}}

\newcommand\fra{{\mathfrak a}} 
\newcommand\frb{{\mathfrak b}}
\newcommand\frg{{\mathfrak g}}
\newcommand\frh{{\mathfrak h}}
\newcommand\frk{{\mathfrak K}}

\newcommand\frt{{\mathfrak t}}

\newcommand\Real{{\mathfrak R}{\frak e}\,} 
\newcommand\Imag{{\mathfrak I}{\frak m}\,}
\newcommand\nilm{\Gamma\backslash G}

\newcommand\db{{\bar{\partial}}}

\begin{document}

\title[]{\small{On astheno-K\"ahler nilmanifolds with balanced metrics}}


\author{Adela Latorre}
\address[A. Latorre]{Departamento de Matem\'atica Aplicada,
Universidad Polit\'ecnica de Madrid,
Avda. Juan de Herrera 4,
28040 Madrid, Spain}
\email{adela.latorre@upm.es}

\author{Luis Ugarte}
\address[L. Ugarte]{Departamento de Matem\'aticas\,-\,I.U.M.A.\\
Universidad de Zaragoza\\
Campus Plaza San Francisco\\
50009 Zaragoza, Spain}
\email{ugarte@unizar.es}


\begin{abstract}
In this paper we study the structure of complex nilmanifolds $X$ admitting some 
special classes of Hermitian metrics, namely, astheno-K\"ahler, strongly Gauduchon and balanced metrics.
We prove that, in complex dimension 4, the existence of a (non necessarily invariant) astheno-K\"ahler metric on $X$ 
implies that the nilmanifold is at most $2$-step and it has first Betti number $\geq 6$. Moreover, the complex structure has a very specific form, sometimes called ``\emph{of special type}'' in the literature.
We also study the interplay between the existence of astheno-K\"ahler metrics and 
that of strongly Gauduchon or balanced metrics. 
A key result is the use of 
some obstructions that are preserved by what we call $\mathfrak{b}$-extensions. This allows us to study the existence of these metrics on important classes of complex nilmanifolds, such as almost abelian, those having maximal nilpotent complex structures, and 8-dimensional nilmanifolds with non-nilpotent complex structures.
We also construct, in every complex dimension $n\geq 4$, complex nilmanifolds admitting both an astheno-K\"ahler metric (possibly also being strongly Gauduchon) and another metric that is balanced. 
As an application, astheno-K\"ahler nilmanifolds with balanced metrics and 
with Fr\"olicher spectral sequence not degenerating at the second or third pages
are found. 
To our knowledge, these are the first compact astheno-K\"ahler manifolds with such properties. 
\end{abstract}

\maketitle

\setcounter{tocdepth}{2} \tableofcontents

\section{Introduction}\label{intro}

\noindent
Compact complex non-K\"ahler manifolds admitting certain types of Hermitian metrics play an important role in geometry and physics. In fact, some results and theories in these areas require that different types of Hermitian metrics coexist on a given compact complex manifold $X$. 
Our main goal in this paper is to study the existence and the interplay between the classes of astheno-K\"ahler and balanced (more generally, strongly Gauduchon) metrics on 
complex nilmanifolds $X=(M,J)$, where the nilmanifold $M$ is endowed with an invariant complex structure $J$. 

Recall that a Hermitian metric $F$ on a complex manifold $X$ is said to be \emph{astheno-Kähler} if $\partial\db F^{n-2}=0$, where $n$ is the complex dimension of $X$. 
In \cite{JY}, Jost and Yau used Hermitian harmonic maps to extend Siu’s rigidity theorem to the case where the domain manifold is astheno-Kähler. 
For $n = 3$, the notion of astheno-Kähler metric and the notion of \emph{SKT} (or \emph{pluriclosed}) metric, defined by the condition $\partial\db F=0$, coincide. In \cite{FT}, Fino and Tomassini proved that 
these classes of Hermitian metrics are no longer the same
in higher dimensions.

The Hermitian metrics satisfying that $F^{n-1}$ is $\partial \db$-closed are 
called \emph{Gauduchon} (or \emph{standard}). By~\cite{Gau}, there exists a Gauduchon metric in the conformal class of
any given Hermitian metric on a compact complex manifold $X$.
An interesting class of Gauduchon metrics is 
that of \emph{balanced} 
Hermitian metrics,
defined by the condition $dF^{n-1}=0$. Important properties of these metrics were first investigated by Michelshon in~\cite{Mi}, where an obstruction to their existence by using currents is given. 

In~\cite{Popovici-invent}, Popovici introduced and studied the class of \emph{strongly Gauduchon} metrics, defined by the condition
$\partial F^{n-1}=\db\alpha$, for some complex form $\alpha$ of bidegree $(n,n-2)$ on $X$. Popovici proved an intrinsic
characterization  
in terms of the non-existence of certain currents, and used such metrics to study the holomorphic deformation limits of projective manifolds (see ~\cite{Popovici-invent} for more details).  
It is clear  from the definitions that 
any balanced metric is strongly Gauduchon, and that any strongly Gauduchon metric is standard.

\vskip.1cm

We recall below some results where the coexistence of an astheno-K\"ahler metric and another balanced, or strongly Gauduchon, Hermitian metric on the manifold $X$ becomes relevant. 

\vskip.1cm

On the one hand, 
Tosatti and Weinkove proved in \cite{TW-Crelle} Calabi-Yau theorems for Gauduchon and strongly Gauduchon metrics on the class of compact astheno-K\"ahler manifolds.
In \cite{STW}, Sz\'ekelyhidi, Tosatti and Weinkove 
proved 
that the Calabi-Yau type equation introduced in \cite{FWW2010} 
can be solved  
on astheno-K\"ahler manifolds with a balanced metric. 
Moreover, 
they observed that in this case 
the conjectures in 
\cite[Conjectures 4.1 and 4.2]{Tosatti} would hold,  
providing in particular
a Calabi-Yau theorem for balanced metrics in the Chern-Ricci flat setting on the class of compact astheno-K\"ahler manifolds. 
This led them to pose the question
about the existence of examples of non-K\"ahler compact complex manifolds admitting both balanced and astheno-K\"ahler metrics.

 A conjecture in \cite{Fei}, formulated as part of a folklore conjecture, stated that 
a compact balanced manifold cannot admit any astheno-K\"ahler metric, unless it is K\"ahler. However, 
in \cite[Proposition 4.8 and Corollary 4.10]{FU} a counterexample in real dimension $8$ can be found. 
Other counterexamples are constructed in \cite{FGV} in the simply connected $22$-dimensional case, and in \cite{LU-cr} on nilmanifolds 
of every real dimension $2n\geq 8$.  
Note that for $n=3$ this is an open question; indeed, the Fino-Vezzoni conjecture \cite{FV1,FV2} states that any compact complex manifold $X$ that admits both an SKT metric and a balanced metric must be a K\"ahler manifold.

\vskip.1cm

On the other hand, Matsuo and Takahashi proved in \cite{MT} that if a given Hermitian metric $F$ satisfies the astheno-K\"ahler condition and the balanced condition, then  
$dF=0$, i.e. it is a K\"ahler metric. However, there are compact complex non-K\"ahler manifolds of complex dimension 3 with Hermitian metrics that are both SKT (astheno-K\"ahler) and strongly Gauduchon. 
In fact, the invariant complex structures on 6-dimensional nilmanifolds 
that admit metrics of this type are classified in \cite[Proposition 2.1]{OUV}. To our knowledge, no general existence or classification results on nilmanifolds with strongly Gauduchon astheno-K\"ahler metrics 
have been obtained in higher dimensions.

Astheno-K\"ahler and strongly Gauduchon 
metrics play an important role in \cite{PZZ}, where 
C. Pan, C. Zhang and X. Zhang study 
the non-abelian Hodge correspondence on certain non-K\"ahler manifolds. More concretely, in their extension of the classical Corlette-Simpson correspondence to a non-K\"ahler setting, they consider compact complex manifolds $X$ with $\dim_{\C} X=n$ endowed with a Hermitian metric $F$ 
satisfying 
the following three conditions: $F$ is Gauduchon, $F$ is 
astheno-K\"ahler (i.e. $\partial\db F^{n-1}=0=\partial\db F^{n-2}$), and 
$\displaystyle\int_X \partial[\eta]\wedge F^{n-1} =0$,  
for any Dolbeault cohomology class $[\eta]\in H^{0,1}_{\db}(X)$ (see \cite[Theorem 1.5]{PZZ}). Note that any left-invariant Hermitian metric on a complex nilmanifold which is at the same time astheno-Kähler and strongly Gauduchon satisfies the three conditions above.

\vskip.15cm

As we already mentioned, the aim of this paper is to study 
the existence and the interplay between the classes of astheno-K\"ahler and balanced, or more generally strongly Gauduchon, metrics on 
complex nilmanifolds. 
By a \emph{complex nilmanifold} $X=(M,J)$ we mean a nilmanifold $M=\nilm$ endowed with an \emph{invariant} 
complex structure $J$, i.e. a complex structure coming from a left-invariant one on the Lie group~$G$. 
In what follows, AK stands for ``astheno-K\"ahler'' and sG for ``strongly Gauduchon''.

\vskip.12cm

For our study of AK nilmanifolds, we first consider an obstruction to the existence of AK metrics given by Sferruzza and Tomassini 
in \cite[Lemma 3.5]{ST-MathZ}, which is expressed in terms of  the $(2,2)$-component of the $dd^c$ operator acting on certain real 2-forms on the manifold. We shall refer to it as the \emph{AK obstruction}. After revisiting in Proposition~\ref{classif-SKT} the complex 3-dimensional case, 
in Section~\ref{sec:estructurageneral} we first prove 
that the AK obstruction is preserved by $\frb$-extensions. Here $\frb$ is any proper $J$-invariant ideal in the Lie algebra $\frg$ underlying the complex nilmanifold $X=(M,J)$. 
This allows us 
to describe the structure of AK nilmanifolds of complex dimension~4 in Theorem~\ref{general-structure-of-aK-dim8}. 
In particular, the existence of a generic (i.e. non necessarily left-invariant) AK metric on a complex  nilmanifold $X=(M,J)$ with $\dim_{\C} X=4$, implies that $M$ has first Betti number $b_1(M)\geq 6$ and its underlying Lie algebra $\frg$ is at most $2$-step. 
Furthermore, the complex structure $J$ must be \emph{of special type}, in the terminology of \cite{ST-MathZ}. 
As an application, we show in 
Propositions~\ref{Extension-ST} and~\ref{Extension-ST-bis} that the AK condition plays a key role in the geometric-Bott-Chern-formality \cite{CT-arxiv,ST-JGA}  of complex nilmanifolds of complex dimension~4.

In Section~\ref{sec:no-sG}, we focus on the interplay between the existence of AK metrics and that of sG or balanced metrics, and we 
study it on several important classes of complex nilmanifolds.
In Proposition~\ref{condition-for-non-sG} we consider 
an obstruction to the existence of sG metrics in terms of 
certain complex $(0,1)$-forms 
on the manifold that 
we call the \emph{sG obstruction}. It turns out that the sG obstruction is also preserved by $\frb$-extensions. This allows us to find 
complex $n$-dimensional nilmanifolds, for every $n\geq 3$, that are AK but not sG (see Proposition~\ref{abel-AK-noSG-dim-n}), as well as 
complex nilmanifolds that are neither AK nor sG (see Proposition~\ref{max-nilpotent}). In particular, we prove that the nilmanifolds endowed with \emph{maximal nilpotent} complex structures
introduced by Gao, Zhao and Zheng 
in \cite{GaoZZ} 
do not admit any AK or sG (or balanced) metric. 

As another application of the preservation by $\frb$-extensions of the AK and sG obstructions, we study in Proposition~\ref{almost-ab-prop} the class of 
\emph{almost abelian} complex nilmanifolds $X$. 
The complex structure equations of any such $X$ are characterized by 
Andrada, Arroyo, Barberis, Rollenske and Wehle 
in~\cite{AABRW} in terms of a tuple 
$(k_i)_{i=0}^r$ of integer numbers. We prove that 
$X$ has an sG (or balanced) metric if and only if $k_0=0$, whereas 
$X$ is AK if and only if $k_i=1$ for every $0\leq i\leq r$ (i.e. $X$ is the product of the Kodaira-Thurston manifold and a torus).
On the other hand, in Subsection~\ref{subsec-non-nilp} we identify the complex structures of \emph{non-nilpotent type} in real dimension 8 that do not admit any sG metric. 
As a consequence, we provide a restriction on the first Betti number of sG nilmanifolds in
Corollary~\ref{Betti-for-sG-dim8}.

In Section~\ref{sec:Finvariante}, 
we study the properties of those Hermitian metrics on complex nilmanifolds $X=(M,J)$ 
that come from left-invariant ones on the corresponding Lie group. 
This allows us to restrict our attention to the Lie algebra level $(\frg,J)$. 
In Proposition~\ref{reducc-cond-met-prop}, we characterize the AK and the balanced conditions on central extensions, thus covering the class of quasi-nilpotent complex structures $J$ on $\frg$. 
Combining this result with others, 
we provide in Theorem~\ref{main-theorem-2} a complete study of the existence of balanced and sG metrics on nilmanifolds of real dimension 8 endowed with non-nilpotent complex structures. 
In addition, several structure results are obtained in Theorems~\ref{classif-aK-dim8-invariante-theorem}\! -- \!\ref{classif-aK-dim8-invariante-balanced} 
for 4-dimensional complex nilmanifolds admitting \emph{invariant} AK metrics and 
a balanced or sG metric. 
We then extend the previous results to the construction, in every complex dimension $n\geq 4$, of complex nilmanifolds admitting both an AK metric (possibly sG) and another metric that is balanced
(see Proposition~\ref{existenciaAK-balanced-n-dim}).

In Section~\ref{FSS} we study the Fr\"olicher spectral sequence (FSS for short) $E_r(X)$ of some 
compact complex manifolds $X$ with AK and balanced metrics. 
It is well known that compact K\"ahler manifolds (or $\partial\db$-manifolds) have FSS degenerating at the first page, whereas the Iwasawa manifold satisfies $E_1\ne E_2=E_\infty$. The first examples of compact complex manifolds with $E_2\ne E_\infty$ were constructed by Cordero, Fern\'andez and Gray in \cite{CFG-CR} by means of complex nilmanifolds.  As far as we know, all the examples in the literature of complex nilmanifolds with $E_2\ne E_\infty$ are those given in 
\cite{BR,CFG-CR,CFG-illinois,CFGU97} 
and the more recent examples in \cite{LUV-Fro} and \cite{KS}. 
We prove in Propositions~\ref{Prop-CFG}\! --\! \ref{Prop-FSS-Kasuya-Stelzig} that all these examples do not support any AK metric. 

In Theorem~\ref{FSS-AK-balanced-5dim} we construct 5-dimensional AK nilmanifolds with balanced metrics  and with $E_3\ne E_\infty$. To our knowledge, these are the first compact complex manifolds with such properties (see Remark~\ref{sec6-first-examples}). 
Other examples of balanced AK nilmanifolds in complex dimension 4 with $E_2\ne E_\infty$ are given in Proposition~\ref{Prop-FSS-dim4}. 

\section{An obstruction to the existence of astheno-K\"ahler metrics}\label{AKobstruction-section}

\noindent
We begin this section recalling some essential notions about complex manifolds that will serve us to fix the notation used throughout this paper. We pay special attention to Hermitian metrics and the relations among different
classes of them. Later, we focus on nilmanifolds endowed with invariant complex structures and describe
some key facts about them, both concerning their structure and the existence of Hermitian metrics. 

\vskip.1cm

Let $M$ be a $2n$-dimensional differentiable manifold.
An \emph{almost complex structure} $J$ on $M$ is an endomorphism
of the space $\mathfrak X(M)$ of smooth vector fields on $M$ satisfying $J^2=-Id$. Then, $(M,J)$ 
is called an almost complex manifold.
The almost complex structure $J$ induces a natural bigraduation on the space of complex $k$-forms 
$$\Omega^k_{\mathbb C}(M)=\bigoplus_{p+q=k}\Omega^{p,q}(M).$$
We will denote the projection of each $\alpha\in\Omega^k_{\mathbb C}(M)$
onto its $(p,q)$ part by $\alpha^{p,q}$, namely, $\alpha^{p,q}:=\pi^{p,q}(\alpha)$ with 
$\pi^{p,q}:\Omega^{k}_{\mathbb C}(M)\to\Omega^{p,q}(M)$ and $k=p+q$.

An almost complex structure $J$ on $M$ is \emph{integrable} when its Nijenhuis tensor
\begin{equation}\label{nijenhuis}
N_J(X,Y):=[X,Y]+J[JX,Y]+J[X,JY]-[JX,JY],\quad X,Y\in\mathfrak X(M),
\end{equation}
vanishes identically. Then, $J$ is said to be a \emph{complex structure} on $M$ as, 
by the celebrated Newlander-Nirenberg
Theorem~\cite{NN}, the almost complex manifold $(M,J)$ turns to be a complex manifold.
Indeed, the aforementioned result characterizes every $n$-dimensional complex manifold $X$ as a pair $(M,J)$ 
where $M$ is a $2n$-dimensional differentiable manifold and $J$ is a complex structure on it. Moreover, the
integrability of the almost complex structure $J$ is equivalent to the condition
$$d\alpha\in\Omega^{2,0}(M)\oplus\Omega^{1,1}(M),$$
for every $\alpha\in\Omega^1_{\mathbb C}(M)$, where $d$ denotes the exterior differential. In particular,
on every complex manifold $(M,J)$ one can decompose $d=\partial+\bar{\partial}$, with
$$\partial:\Omega^{p,q}(M)\to\Omega^{p+1,q}(M)
\quad\text{and}\quad
\bar{\partial}:\Omega^{p,q}(M)\to\Omega^{p,q+1}(M).$$

A \emph{Hermitian metric} $g$ on an $n$-dimensional complex manifold $X=(M,J)$ is a Riemannian metric on $M$
that satisfies $g(J\cdot,J\cdot)=g(\cdot,\cdot)$. 
Every Hermitian metric $g$ on $X$ has an associated $2$-form given
by $F(\cdot,\cdot):=g(J\cdot,\cdot)$, which is a real and positive (1,1)-form on $X$. 
Since $g$ and $F$ are in one-to-one correspondence, we will indistinctly refer to the Hermitian metric 
as $g$ or $F$.

Probably, one of the best known classes of Hermitian metrics is that of \emph{K\"ahler} metrics. These are 
characterized by the condition $dF=0$, which makes the fundamental form $F$ be a symplectic form. However,
there are complex manifolds that are not K\"ahler. In contrast, it was proved in~\cite{Gau} that every compact complex
manifold admits a \emph{Gauduchon metric}, namely, a Hermitian metric that
satisfies $\partial\bar{\partial}F^{n-1}=0$. We next present other
special types of Hermitian metrics that will appear in this paper.

\begin{definition}\label{def-metricas}
Let $X$ be a complex manifold with $\dim_{\C}X=n$. A Hermitian metric $F$ is said to be
\begin{itemize}
\item[\textit{(i)}] \emph{balanced}, if $dF^{n-1}=0$;
\item[\textit{(ii)}] \emph{strongly Gauduchon} \emph{(sG)}, if there exists an $(n,n-2)$-form $\alpha$ such that $\partial F^{n-1}=\bar{\partial}\alpha$;
\item[\textit{(iii)}] \emph{pluriclosed} or \emph{strong K\"ahler with torsion} \emph{(SKT)}, if $\partial\bar{\partial}F=0$;
\item[\textit{(iv)}] \emph{astheno-K\"ahler} \emph{(AK)}, if $\partial\bar{\partial}F^{n-2}=0$.
\end{itemize}
\end{definition}

It is worth to note that there exist some interesting relations among the aforementioned types of Hermitian metrics.
For instance, one has
$$\big\{ \text{K\"ahler metrics} \big\}
 \ \ \subsetneq \ \ 
\big\{ \text{balanced metrics} \big\}
 \ \ \subsetneq \ \ 
\big\{ \text{sG metrics} \big\}
 \ \ \subsetneq \ \ 
\big\{ \text{Gauduchon metrics} \big\},$$
therefore, both balanced and sG metrics can be seen as generalizations of K\"ahler metrics. Moreover, it is proved in \cite{Popovici-Pisa} that the above relations also hold for compact complex manifolds. For instance,  the Iwasawa manifold is balanced and does not admit any Kähler metric. In \cite[Propositions 5.3 and 5.5]{COUV} one can find sG nilmanifolds that do not support any balanced metric as well as complex nilmanifolds not admitting any sG metric.

\vskip.15cm

A natural class where studying the special Hermitian metrics presented in Definition~\ref{def-metricas} is that of 
nilmanifolds endowed with complex structures.
A \emph{nilmanifold} $M=\nilm$ is a compact quotient of a connected and simply connected real nilpotent Lie 
group and a lattice. 
By~\cite{BG, Has}, it is well-known that the only K\"ahler nilmanifolds are complex tori.
In this work we focus on nilmanifolds endowed with \emph{invariant complex structures}, namely, 
those defined on the Lie algebra $\frg$ of $G$. For the seek of simplicity, we will refer
to a nilmanifold with invariant complex structure as, simply, a \emph{complex nilmanifold}. 
Due to~\cite{Malcev},
the existence of a basis of $\frg$ where the structure constants are rational numbers characterizes the
existence of a lattice $\Gamma$ that makes $M=\nilm$ a compact quotient. Hence, the classification of
nilmanifolds with invariant complex structures is directly related to the classification of 
rational nilpotent Lie algebras with complex structures.
We also recall that, by Nomizu's Theorem \cite{Nomizu}, the de Rham cohomology $H^*_{\mathrm{dR}}(M;\mathbb R)$ of the nilmanifold $M$ is isomorphic to the cohomology of its underlying Lie algebra $\frg$. So, the Betti numbers of $M$ can be directly obtained from $\frg$, in particular, $b_1(M)=\dim \{\ker d\colon \frg^* \longrightarrow \bigwedge^2\frg^*\}$.

In this work, we use the standard abbreviated notation to describe a real Lie algebra $\frg$. 
In particular, each $m$-dimensional $\frg$ is given by an $m$-tuple whose $k$-th component
contains the differential of the $k$-th element of its basis. 
For instance, the tuple $(0,0,12,13)$ denotes the $4$-dimensional real Lie algebra $\frg$ with a basis 
$\{e^k\}_{k=1}^4$ for $\frg^*$ satisfying
$de^1=de^2=0$, $de^3=e^1\wedge e^2$, $de^4=e^1\wedge e^3$. For the seek of simplicity, 
in the following we will
simply write $e^{ij}$ instead of $e^i\wedge e^j$ (similarly when the wedge product of complex forms arises).

Let $\frg$ be a Lie algebra admitting a complex structure $J$, i.e. an endomorphism of $\frg$ satisfying $J^2=-Id$ 
and $N_J\equiv 0$, where $N_J$ denotes the Nijenhuis tensor~\eqref{nijenhuis}. 
Given any basis $\{\omega^k\}_{k=1}^n$ for $\frg^{1,0}$, 
the differentials of the forms $\omega^k$, $k=1,\ldots,n$, 
define a set of 
\emph{complex structure equations} for $(\frg,J)$. 
Conversely, a set of complex structure equations completely determines 
both a Lie algebra $\frg$ and a complex structure $J$ on $\frg$.
A characterization of those $(\frg,J)$, with $\frg$ \emph{nilpotent} and $J$ complex structure on $\frg$, in terms of a certain (1,0)-basis $\{\omega^k\}_{k=1}^n$ is given by Salamon in \cite{Salamon}. 

\vskip.1cm

Let $(\frg,J)$ be a nilpotent Lie algebra with complex structure. The \emph{ascending $J$-compatible series} $\{\fra_k(J)\}$
of $\frg$ introduced in~\cite{CFGU-dolbeault} is given by
\begin{equation}\label{sucesion-a}
\fra_0(J)=\{0\},\quad\  
\fra_k(J)=\{X\in\frg \mid [X,\frg]\subseteq \fra_{k-1}(J)\ {\rm and\ } [JX,\frg]\subseteq \fra_{k-1}(J)\}, \text{ for } k\geq 1.
\end{equation}
Each term $\fra_k(J)\neq \{0\}$ is a $J$-invariant ideal of $\frg$ and so it has even dimension. Moreover, 
$\fra_1(J)$ is the largest $J$-invariant subspace contained in the center $\mathcal Z(\frg)$ of $\frg$, i.e. $\fra_1(J)=\mathcal Z(\frg)\cap J(\mathcal Z(\frg))$. The interest in this series
comes from the fact that it allows us to distinguish different types of complex structures on nilpotent Lie algebras:

\begin{definition}\cite{CFGU-dolbeault, LUV1}\label{def-tipos-Js}
A complex structure $J$ on a nilpotent Lie algebra $\frg$ is called
\begin{itemize}
\item[\textit{(i)}] \emph{nilpotent}, if there exists an integer $t>0$ such that $\fra_t(J)=\frg$;
\item[\textit{(ii)}] \emph{non-nilpotent}, if for every integer $t>0$ one has $\fra_t(J)\neq\frg$; moreover, $J$ is called 
	\begin{itemize}
	\item[\textit{(a)}] \emph{strongly non-nilpotent (}\emph{SnN} for short), if $\fra_1(J)=\{0\}$,
	\item[\textit{(b)}] \emph{weakly non-nilpotent (}\emph{WnN} for short), if $\fra_1(J)\neq\{0\}$;
	\end{itemize}
\item[\textit{(iii)}] \emph{quasi-nilpotent}, if $\fra_1(J)\neq\{0\}$. 
\end{itemize}
\end{definition}

We observe that every nilpotent complex structure is quasi-nilpotent. 
Moreover, in dimensions greater than or equal to 8, there exist quasi-nilpotent complex structures that are non-nilpotent: they are of WnN type 
(see \cite{LUV1} for more details).

We also recall 
that a special class of nilpotent complex structures is the one constituted by the \emph{abelian} complex structures $J$, which satisfy $[Jx,Jy]=[x,y]$, for every $x,y\in \frg$ \cite{ABD}.

\smallskip

Concerning the existence of (non-K\"ahler) Hermitian metrics on complex nilmanifolds, it is worth to
observe that, for some of them, it can be directly detected at the Lie algebra level. More precisely, let $X=(\nilm,J)$
be a complex nilmanifold and denote $\frg$ the Lie algebra of $G$. Given
a Hermitian metric $F$ on $X$, the symmetrization
process introduced in~\cite{Belgun} allows to find an \emph{invariant} Hermitian metric $\tilde F$ on $X$, namely,
a metric defined on $(\frg,J)$.
Indeed, if $F$ is a balanced metric then $\tilde F$ is also balanced \cite{FG}, and similar results hold for
sG metrics \cite{COUV} and for SKT metrics \cite{Ug}. Unfortunately, the symmetrization process does not apply, a priori, to AK metrics. For this
reason, obstruction results such as the one given below can be of special interest: 

\begin{lemma}\label{lemaST}\cite[Lemma 3.5.]{ST-MathZ}
Let $X$ be a compact complex manifold of complex dimension~$n$. Let $\alpha$ be a
$(2n - 2p - 2)$-form which is not $dd^c$-closed and such that
\begin{equation}\label{ST-condition}
(dd^c\alpha)^{n-p,n-p}= \sum_{k} c_k\,\psi^k\wedge \bar{\psi}^k,
\end{equation}
with $\psi^k$
simple $(n - p, 0)$-covectors and $c_k \neq 0$ constants having the same sign. Then $X$ does
not admit any $p$-pluriclosed form.
In particular,

$\bullet$ for $p = 1$, $X$ does not admit SKT metrics;

$\bullet$ for $p = n - 2$, $X$ does not admit astheno-K\"ahler metrics.

\end{lemma}

We recall that a \emph{$p$-pluriclosed form} on $X$ is a $\partial\bar{\partial}$-closed, transverse $(p,p)$-form on $X$ (see \cite{ST-MathZ} for more details).
We also remind that for complex manifolds one has $dd^c=2\,i\,\partial\bar{\partial}$, and  simple $(n - p, 0)$-covectors are just decomposable $(n - p, 0)$-forms. Motivated by the previous result, we introduce the following:

\begin{definition}\label{condition-no-AK}
We say that a compact complex manifold $X$ satisfies the \textbf{AK obstruction} if there exists a real $2$-form 
$\alpha$ on $X$ satisfying 
\begin{equation}\label{2-condition}
(dd^c\alpha)^{2,2}= \sum_{k} c_k\,\psi^k\wedge \bar{\psi}^k,
\end{equation}
with $\psi^k$ decomposable $(2, 0)$-forms and 
$c_k \in \R^*$ having the same sign. 
\end{definition}

It is worth to remark that Chiose and R\u{a}sdeaconu provided in \cite{Chiose-Rasdeaconu} 
an obstruction to the existence of AK metrics in terms of the Bott-Chern and Aeppli cohomologies of $X$. Due to the fact that complex nilmanifolds are determined by explicit complex structure equations, the obstruction in Definition~\ref{condition-no-AK} appears to be more suitable for the class of nilmanifolds.

\vskip.1cm

Let $X=(M,J)$ be a complex nilmanifold with Lie algebra $\frg$. The following implications hold:

\vskip.3cm

\hskip.15cm  AK obstruction on $(\frg,J)$  $\Longrightarrow$ 
 AK obstruction on $X$ 
$\Longrightarrow$  
$X$ does not admit any AK metric.

\vskip.3cm

In fact, the first implication is a direct consequence of the natural injection of the Lie algebra complex $(\Lambda^*\frg^*,d)$ into the De Rham complex $(\Omega^*(X),d)$ of the nilmanifold. Note that this holds for every quotient manifold $\nilm$ where $G$ is not necessarily nilpotent. The second implication is given in Lemma~\ref{lemaST} for $p=n-2$.

We do not know if the converses of these implications hold in general.  
Nevertheless, our first result is that for complex nilmanifolds of
complex dimension 3 the above conditions are equivalent. 
Note that in this dimension the AK metrics are the same as the SKT metrics. 

Let $\frg$ and $\frg'$ be Lie algebras endowed with respective complex structures $J$ and $J'$. We recall that the pairs $(\frg,J)$ and $(\frg',J')$ are said to be 
\emph{equivalent} (or \emph{isomorphic}) if there exists an isomorphism of Lie algebras $f:\frg\to\frg'$ satisfying $f\circ J=J'\circ f$.

\begin{proposition}\label{classif-SKT}
Let $X=(M,J)$ be a complex nilmanifold of complex dimension $3$, 
and let $\frg$ be the underlying Lie algebra. Then, 
the AK obstruction on $(\frg,J)$ is equivalent to the non-existence of SKT metrics on $X$.
In particular, the AK obstruction is not satisfied
if and only if 
the pair $(\frg,J)$ is isomorphic to one (and only one)
in the following list: 
\begin{equation*}
\begin{split}
(\frh_{1},J) \ \ & d\omega^1 =d\omega^2 =d\omega^3=0 \text{\bf \ \ (unique)},\\
(\frh_{2},J) \ \ & d\omega^1 =d\omega^2 =0,\ 
			d\omega^3=\omega^{1\bar1}+i\,\omega^{2\bar2} \text{\bf \ \ (unique)},\\
(\frh_{2},J'_{y}) \ \ & d\omega^1 =d\omega^2 =0\\[-2pt]
	& d\omega^3=\omega^{12} + \omega^{1\bar1}+ \omega^{1\bar2}+ (1+iy)\,\omega^{2\bar2},
		\ y\in(0,+\infty) \text{\bf \ \ (1-parameter family)},\\
(\frh_{4},J) \ \ & d\omega^1 =d\omega^2 =0,\ 
			d\omega^3=\omega^{12} + \omega^{1\bar1}+ \omega^{1\bar2}+ \omega^{2\bar2} \text{\bf \ \ (unique)},\\
(\frh_{5},J_{y}) \ \ & d\omega^1 =d\omega^2 =0,\ 
			d\omega^3=\omega^{12} +\omega^{1\bar1}+(\nicefrac12+iy)\,\omega^{2\bar2},
			\ y\!\in\!  [0,\nicefrac{\sqrt{3}}{2}\big) \text{\bf \ \ (1-parameter family)},\\
(\frh_{8},J) \ \ & d\omega^1 =d\omega^2 =0,\ d\omega^3=\omega^{1\bar1} \text{\bf \ \ (unique)}.
\end{split}
\end{equation*} 
Here $\frh_1=(0,\ldots,0)$ is the Abelian Lie algebra, and  $\frh_k$ are given by 
$\frh_2=(0,0,0,0,12,34)$,   $\frh_4=(0,0,0,0,12,14+23)$,   $\frh_5=(0,0,0,0,13-24,14+23)$ and  $\frh_8=(0,0,0,0,0,12)$.
\end{proposition}

\begin{proof}
We recall that, in complex dimension 3, 
complex structures on nilpotent Lie algebras
can be described in terms of four families, with structure equations given as (see \cite{COUV} for more details):
\begin{equation*}
\begin{split}
\text{(Eq-i)}\quad & d\omega^1=d\omega^2=0, \  
			d\omega^3=\rho\,\omega^{12}, \quad \rho\in\{0,1\};\\
\text{(Eq-ii)}\quad & d\omega^1=d\omega^2=0,\  
			d\omega^3=\rho\,\omega^{12} + \omega^{1\bar1} + \lambda\,\omega^{1\bar2} + D\,\omega^{2\bar2},\\[-2pt]
	& \text{where }\rho\in\{0,1\}, \lambda\in\mathbb R^{\geq 0}, \text{ and }D\in\mathbb C \text{ with }\Imag\!(D)\geq 0;\\
\text{(Eq-iii)}\quad & d\omega^1=0,\  d\omega^2=\omega^{1\bar1},\  
			d\omega^3=\rho\,\omega^{12} + B\,\omega^{1\bar2} + c\,\omega^{2\bar1},\\[-2pt]
	& \text{where }\rho\in\{0,1\}, B\in\mathbb C, c\in\mathbb R^{\geq 0}, \text{ with } (\rho,B,c)\neq(0,0,0);\\
\text{(Eq-iv)}\quad & d\omega^1=0, \  d\omega^2=\omega^{13} + \omega^{1\bar3}, \  
			d\omega^3= i\,\varepsilon\,\omega^{1\bar 1} + i\, \delta (\omega^{1\bar2} - \omega^{2\bar1}),\\[-2pt]
	& \text{where }\varepsilon\in\{0,1\}, \text{ and }\delta=\pm 1.
\end{split}
\end{equation*}

We also remember that, by the well-known symmetrization process \cite{Belgun}, the existence of an SKT metric on $X$ is equivalent to its existence on $(\frg,J)$.

Now, for complex structures in (Eq-i) we first observe that $dd^c(\Lambda^2\frg^*)=\langle dd^c(i\,\omega^{3\bar{3}}) \rangle$. 
If we denote $\alpha=i\,\omega^{3\bar{3}}$, then 
$(dd^c\alpha)^{2,2}=2i\partial\db(i\omega^{3\bar{3}})
	=-2\partial\db(\omega^{3\bar{3}})
	=2\rho\,\omega^{12}\!\wedge \omega^{\bar{1}\bar{2}}$.
Hence, the AK obstruction \eqref{2-condition} holds for complex structures in (Eq-i) if and only if $\rho=1$. 
It is well known \cite{FPS} that the existence of an SKT metric for (Eq-i) is equivalent to the condition $\rho=0$. 
The latter corresponds to $(\frh_1,J)$. 

For complex structures in (Eq-ii) one has $dd^c(\Lambda^2\frg^*)=\langle dd^c(i\,\omega^{3\bar{3}}) \rangle$, 
where
$
\big(dd^c (i\,\omega^{3\bar{3}})\big)^{2,2}=2i\partial\db(i\omega^{3\bar{3}})=-2\partial\db(\omega^{3\bar{3}})=
2(\rho+\lambda^2-2\,\Real\!(D))\,\omega^{12}\!\wedge \omega^{\bar{1}\bar{2}}$.
Hence, the condition \eqref{2-condition} is satisfied for (Eq-ii) if and only if $\rho+\lambda^2-2\,\Real\!(D)\neq0$. 
Moreover, the equality $\rho+\lambda^2=2\,\Real\!(D)$ is equivalent to the existence of an SKT metric
 for (Eq-ii), by \cite{FPS} . 
Now, it suffices to apply \cite[Proposition 2.10]{LUV-DGA} (see also \cite{COUV}) for the classification
of complex structures~$J$ in (Eq-ii) that admit SKT metrics. This gives the cases $(\frh_{2},J)$, $(\frh_{2},J'_{y})$, $(\frh_{4},J)$, $(\frh_{5},J_{y})$, and $(\frh_{8},J)$ listed in the proposition. 

To finish the proof, we will show that the AK obstruction \eqref{2-condition} is always satisfied for complex structures in the remaining cases (Eq-iii) and (Eq-iv). It is known that those complex structures do not admit SKT metrics \cite{FPS}. 
Indeed, for (Eq-iii) 
we get that $dd^c(\Lambda^2\frg^*)=\langle dd^c(i\,\omega^{3\bar{3}}) \rangle$, where 
$\big(dd^c (i\,\omega^{3\bar{3}})\big)^{2,2}=2i\partial\db(i\omega^{3\bar{3}})=-2\partial\db(\omega^{3\bar{3}})=
2(\rho+|B|^2+c^2)\,\omega^{12}\!\wedge \omega^{\bar{1}\bar{2}}$. Recall that 
$(\rho,B,c)\neq(0,0,0)$. 
Finally, for (Eq-iv) we have $dd^c(\Lambda^2\frg^*)=\langle dd^c(i\,\omega^{2\bar{2}}), dd^c(i\,\omega^{3\bar{3}}) \rangle$.  In particular, one has
$
\big(dd^c(i\omega^{2\bar{2}})\big)^{2,2}=2i\partial\db(i\omega^{2\bar{2}})=-2\partial\db(\omega^{2\bar{2}})=
4\,\omega^{13}\!\wedge \omega^{\bar{1}\bar{3}}$.
\end{proof}

\section{The structure of AK nilmanifolds up to complex dimension 4}\label{sec:estructurageneral}

\noindent
To study the existence of AK metrics in complex dimensions higher than 3, the notion of $\frb$-extension that we next introduce will be especially useful.

Let $\frg$ be a Lie algebra endowed with a complex structure $J$, and let $\frb$ be a proper $J$-invariant ideal, i.e. $J(\frb)=\frb$ and $[\frb,\frg]\subseteq \frb$, with $\{0\}\neq \frb\neq \frg$. Note that such an ideal $\frb$ always exists when $\frg$ is nilpotent (see Remark~\ref{existe-b} below). 
Then, the quotient Lie algebra $\frg_{\frb}:=\frg\slash\frb$
has an induced complex structure
\begin{equation}\label{induced-Jb}
J_{\frb}(\tilde x):=\widetilde{Jx}, \quad \text{ for all }\tilde x\in\frg_{\frb},
\end{equation}
where $\tilde x$ and $\widetilde{Jx}$ denote, respectively, the classes of $x$ and $Jx$ in $\frg_{\frb}$. That is, the complex structures $J$ and $J_{\frb}$ commute with the natural projection $\pi_\frb \colon\frg\longrightarrow \frg_{\frb}$, i.e. $\pi_\frb \circ J=J_{\frb}\circ \pi_\frb$.

\begin{definition}\label{b-extension}
{\rm
Let $(\frg,J)$ and $(\tilde{\frg},\tilde{J})$ be Lie algebras with complex structures.
Given  
a proper $J$-invariant ideal $\frb$ of $\frg$, 
we say that $(\frg,J)$ is
a} 
\textbf{$\frb$-extension} 
{\rm
of $(\tilde{\frg},\tilde{J})$
if the pairs $(\frg_{\frb},J_{\frb})$ and $(\tilde{\frg},\tilde{J})$ are equivalent. 
In addition, if $\frb$ can be chosen in the center of the Lie algebra $\frg$, then $(\frg,J)$ is called 
a}
\textbf{central $\frb$-extension} 
{\rm 
of $(\tilde{\frg},\tilde{J})$.
Moreover, 
we simply say that $(\frg,J)$ is a}
\textbf{(central) extension} 
{\rm 
of $(\tilde{\frg},\tilde{J})$
if there exists a proper (central) $J$-invariant ideal $\frb$ in $\frg$ such that $(\frg,J)$ is a (central) 
$\frb$-extension of $(\tilde{\frg},\tilde{J})$.  
}
\end{definition}

The following result shows that the AK obstruction is preserved on extensions. More precisely: 

\begin{proposition}\label{extensions-of-non-AK} 
Let $X=(M,J)$ be a complex nilmanifold of complex dimension $n\geq 4$ with
underlying Lie algebra $\frg$. Suppose that
$(\frg,J)$ is an extension of a pair $(\tilde{\frg},\tilde{J})$ that satisfies the AK obstruction.  
Then, $(\frg,J)$ also satisfies the AK obstruction, and so the nilmanifold $X=(M,J)$ does not admit any AK metric.
\end{proposition}

\begin{proof}
Let $\frb$ be a proper $J$-invariant ideal in $\frg$ such that $(\frg,J)$ is a 
$\frb$-extension of $(\tilde{\frg},\tilde{J})$, that is, there exists an isomorphism of Lie algebras $f:\frg_\frb\to\tilde\frg$ satisfying $f\circ J_\frb=\tilde J\circ f$, as the pairs $(\tilde{\frg},\tilde{J})$  and 
$(\frg_\frb,J_\frb)$ are equivalent. Define $\pi=f\circ\pi_\frb\colon \frg \longrightarrow \tilde{\frg}$. Then, the complex structures $J$ and $\tilde J$ conmute with $\pi$, i.e. $\pi\circ J=\tilde J\circ\pi$. 
Next, we denote by $d$, resp. $\tilde{d}$, the Lie algebra differential on $\frg$, resp. $\tilde{\frg}$.

Let $\tilde{\alpha}$ be a 
$2$-form on $\tilde{\frg}$ such that 
$(\tilde{d}\tilde{d}^c\tilde{\alpha})^{2,2}= \sum_{k} c_k\,\widetilde{\psi}^k\wedge \overline{\widetilde{\psi}^k}$,
with $\widetilde{\psi}^k$ 
simple $(2, 0)$-covectors and $c_k \neq 0$ constants having the  same sign. Consider $\alpha=\pi^* \tilde{\alpha}$. 
Since $\pi^*$ conmutes with the differentials, the wedge product, and the complex structures, it preserves the bidegree and the simpleness condition. Thus, we get 
$(dd^c\alpha)^{2,2}= \sum_{k} c_k\,\psi^k\wedge \bar{\psi}^k$,
with $\psi^k=\pi^* \widetilde{\psi}^k$
simple $(2, 0)$-covectors and $c_k \neq 0$ the same constants as above. 
\end{proof}

\begin{remark}\label{N-P-obstruction}
One can define, similarly to Definition~\ref{condition-no-AK}, the notion of {\rm $p$-pluriclosed obstruction} by using \eqref{ST-condition} for $(2n - 2p - 2)$-forms. 
Then, if $(\frg,J)$ is an extension of a pair $(\tilde{\frg},\tilde{J})$ that satisfies the $\tilde{p}$-pluriclosed obstruction, then $(\frg,J)$ satisfies the $(\tilde{p}+b)$-pluriclosed obstruction, 
where $b=n-\tilde{n}$ and $2\tilde{n}=\dim \tilde{\frg}$. 
In particular, the nilmanifold $X=(M,J)$ does not admit any $(\tilde{p}+b)$-pluriclosed metric.

Indeed, let $\tilde{\alpha}$ be a 
$(2\tilde{n} - 2\tilde{p} - 2)$-form on $\tilde{\frg}$ such that 
$(\tilde{d}\tilde{d}^c\tilde{\alpha})^{\tilde{n}-\tilde{p},\tilde{n}-\tilde{p}}= \sum_{k} c_k\,\widetilde{\psi}^k\wedge \overline{\widetilde{\psi}^k}$,
with $\widetilde{\psi}^k$ 
simple $(\tilde{n}-\tilde{p}, 0)$-covectors and $c_k \neq 0$ constants having the  same sign. 
As in the proof of Proposition~\ref{extensions-of-non-AK}, consider $\alpha=\pi^* \tilde{\alpha}$ and let $p=\tilde{p}+b$, i.e. $2b=\dim\frb$. Then, $n=\tilde{n}+b$ and $\alpha$ is a $(2n - 2p - 2)$-form on the Lie algebra $\frg$ satisfying  
$(dd^c\alpha)^{n-p,n-p}= \sum_{k} c_k\,\psi^k\wedge \bar{\psi}^k$,
with $\psi^k=\pi^* \widetilde{\psi}^k$ 
simple $(n - p, 0)$-covectors and $c_k \neq 0$ the same constants as above. 

Notice that 
the result for AK metrics given in Proposition~\ref{extensions-of-non-AK} is recovered when $\tilde{n}-\tilde{p}=2\,(=n-p)$. 
\end{remark}

The following result gives characterizations of quasi-nilpotent complex structures (see Definition~\ref{def-tipos-Js}). 

\begin{proposition}\label{no} 
Let $J$ be a complex structure on a nilpotent Lie algebra $\frg$. The following conditions are equivalent:
\begin{itemize}
\item[\textit{(i)}] 
$J$ is quasi-nilpotent;
\item[\textit{(ii)}] 
$(\frg,J)$ is a central extension of some pair $(\tilde{\frg},\tilde{J})$;
\item[\textit{(iii)}]  
There exist positive integer numbers $\tilde{n}, b$ and an ordered basis $\{\omega^{1},\ldots,\omega^{\tilde{n}},\omega^{\tilde{n}+1},\ldots,\omega^{\tilde{n}+b} \}$ of $(1,0)$-forms on $(\frg,J)$ such that 
\begin{equation}\label{caracterizacion}
d\omega^j\in \bigwedge\hskip-0.6cm{\phantom{\bigwedge}}^{(2,0)+(1,1)}_{\C} \langle \omega^{1},\ldots,\omega^{\tilde{n}},\overline{\omega}^{1},\ldots,\overline{\omega}^{\tilde{n}} \rangle, 
\end{equation}
for any $1\leq j\leq \tilde{n}+b$.
\end{itemize}
\end{proposition}

\begin{proof}
The equivalence of (i) and (ii) follows directly from the definitions. We next see that (iii) implies (i).
Suppose that there exists a basis $\{\omega^{1},\ldots,\omega^{\tilde{n}},\omega^{\tilde{n}+1},\ldots,\omega^{\tilde{n}+b} \}$ of $(1,0)$-forms on $(\frg,J)$ satisfying \eqref{caracterizacion}. Let $\{z_{1},\ldots,z_{\tilde{n}},z_{\tilde{n}+1},\ldots,z_{\tilde{n}+b} \}$ be its dual basis. From \eqref{caracterizacion} we get that $[w,z_{\tilde{n}+k}]=0$ for any $w\in\frg_{\C}$ and any $1\leq k\leq b$.
Let $z_{\tilde{n}+k}=x_{\tilde{n}+k}+i\,y_{\tilde{n}+k}$. 
Then, 
$\langle x_{\tilde{n}+k},y_{\tilde{n}+k}\mid 1\leq k\leq b \rangle$ is a $J$-invariant subspace contained in the center $\mathcal Z(\frg)$ of $\frg$, so $J$ is quasi-nilpotent.

We finally prove that (ii) implies (iii). 
Let $\frb$ be a $J$-invariant subspace of real dimension $2b$ in the center $\mathcal Z(\frg)$ of $\frg$. Consider the quotient Lie algebra $\tilde{\frg}\cong \frg/\frb$ and denote by $2\tilde{n}$ its real dimension. 
Let $\{\tilde{\omega}^{1},\ldots,\tilde{\omega}^{\tilde{n}}\}$ be a basis of $(1,0)$-forms on $(\tilde{\frg},\tilde{J})$, where $\tilde{J}$ is the complex structure induced by $J$.  The differential of $\tilde{\omega}^j$ satisfies $\tilde{d} \tilde{\omega}^j\in \bigwedge^{(2,0)+(1,1)}_{\C} \langle \tilde{\omega}^{1},\ldots,\tilde{\omega}^{\tilde{n}},\overline{\tilde{\omega}^{1}},\ldots,
\overline{\tilde{\omega}^{\tilde{n}}} \rangle$, 
for any $1\leq j\leq \tilde{n}$. Recall that $\pi^*\circ \tilde{J}=J\circ\pi^*$, where $\pi\colon \frg \longrightarrow \tilde{\frg}$ denotes the natural projection. Now, we define $\omega^j=\pi^* \tilde{\omega}^j$, for $1\leq j\leq \tilde{n}$, and then complete $\{\omega^{j}\}_{j=1}^{\tilde{n}}$ to a basis $\{\omega^{1},\ldots,\omega^{\tilde{n}},\omega^{\tilde{n}+1},\ldots,\omega^{\tilde{n}+b} \}$ of $(1,0)$-forms on $(\frg,J)$.  If $\{z_{j}=x_{j}+i\,y_{j} \}_{j=1}^{{\tilde{n}+b}}$ denotes its dual, it is easy to check that $\pi(x_{\tilde{n}+k})=0=\pi(y_{\tilde{n}+k})$ for $1\leq k\leq b$. This means that $\frb=\langle x_{\tilde{n}+k},y_{\tilde{n}+k}\mid 1\leq k\leq b \rangle$ and, since $\frb$ is contained in the center of $\frg$, we have that $[w,z_{\tilde{n}+k}]=0$ for any $w\in\frg_{\C}$ and any $1\leq k\leq b$.
Hence, the condition \eqref{caracterizacion} is satisfied.
\end{proof}

\begin{remark}\label{ecus-extensiones}
As a consequence of the previous proof, if $(\frg,J)$ is a central extension of $(\tilde{\frg},\tilde{J})$ then there
is a basis $\{\omega^j\}_{j=1}^n$ of $(1,0)$-forms on $(\frg,J)$ such that the first $\tilde n$ differentials 
$d\omega^{1},\ldots,d\omega^{\tilde{n}}$ coincide with the complex structure equations of $(\tilde{\frg},\tilde{J})$ 
and $d\omega^{\tilde{n}+1},\ldots,d\omega^{\tilde{n}+b}$ follow~\eqref{caracterizacion}.
\end{remark}

\begin{remark}\label{existe-b}
Suppose $(\frg,J)$ has a closed $(1,0)$-form $\tau$, i.e. $d\tau=0$ (this property is always satisfied when $\frg$ is nilpotent~\cite[Theorem 1.3]{Salamon}). 
Take $\omega^1=\tau$ and complete it to a basis of $(1,0)$-forms
 $\{\omega^k\}_{k=1}^n$ for $(\frg,J)$. 
Similarly to the proof of Proposition~\ref{no}, we consider its dual basis $\{z_{j}=x_{j}+i\,y_{j} \}_{j=1}^{n}$. 
The condition $d\omega^1=0$ implies that $\frb=\langle x_{k},y_{k}\mid 2\leq k\leq n \rangle$ is 
a proper $J$-invariant ideal of dimension $2n-2$. 
The most interesting cases arise when such an ideal $\frb$ exists and has dimension $\leq 2n-4$. 
\end{remark}

In the case of complex structures of SnN type, it is worth noting that their structure equations are much 
more twisted. This might be the reason why they have only been classified up to complex dimension 4. In this paper,
we will need the following result.

\begin{proposition}\label{clasifJ-SnN} \cite[Theorem 3.3]{LUV2}
Let $J$ be an SnN complex structure in real dimension 8. There exists a basis of $(1,0)$-forms $\{\omega^k\}_{k=1}^4$ in terms of which the complex structure equations are one (and only one) of the following:

\vskip.3cm

\hskip.6cm
$
\text{Family I:} \quad
\begin{cases}
d\omega^1 = 0,\\
d\omega^2 = \varepsilon\,\omega^{1\bar 1},\\
d\omega^3 = \omega^{14}+\omega^{1\bar 4}+a\,\omega^{2\bar 1}+ i\,\delta\,\varepsilon\,b\,\omega^{1\bar 2},\\
d\omega^4 = i\,\nu\,\omega^{1\bar 1} +b\,\omega^{2\bar 2}+ i\,\delta\,(\omega^{1\bar 3}-\omega^{3\bar 1}),
\end{cases}
$

\medskip\noindent
where $\delta=\pm 1$, $(a,b)\in \mathbb R^2-\{(0,0)\}$ with $a\geq 0$,
and the tuple 
$(\varepsilon, \nu, a, b)$ takes the following values:

\smallskip
\centerline{
$(0,0,0,1),\, (0,0, 1, 0),\, (0,0, 1, 1),\, (0,1, 0, \nicefrac{b}{|b|}),\, (0,1, 1,b),\,
(1,0, 0,1),\,  (1,0, 1,|b|) \text{ or }(1,1, a, b)$;}

\vskip.3cm

\hskip.6cm
$
\text{Family II:} \quad
\begin{cases}
d\omega^1=0,\\
d\omega^2=\omega^{14}+\omega^{1\bar 4},\\
d\omega^3=a\,\omega^{1\bar 1}
                 +\varepsilon\,(\omega^{12}+\omega^{1\bar 2}-\omega^{2\bar 1})
                 +i\,\mu\,(\omega^{24}+\omega^{2\bar 4}),\\
d\omega^4=i\,\nu\,\omega^{1\bar 1}-\mu\,\omega^{2\bar 2}+i\,b\,(\omega^{1\bar 2}-\omega^{2\bar 1})+i\,(\omega^{1\bar 3}-\omega^{3\bar 1}),
\end{cases}
$

\medskip\noindent
where $a, b\in\mathbb R$, and the tuple $(\varepsilon, \mu, \nu, a, b)$ takes the following values:

\smallskip\centerline{
$(1, 1, 0, a, b),\, (1, 0, 1, a, b),\, (1, 0, 0, 0, b), (1, 0, 0, 1, b),\, (0, 1, 0, 0, 0) \text{ or } (0, 1, 0, 1, 0)$.}

\end{proposition}

The theorem below provides the general structure of AK nilmanifolds of complex dimension~4. 
In particular, we observe that the existence of an AK metric makes the complex structure to be of a very specific nilpotent type. 
This suggests that AK metrics might not exist on highly twisted complex nilmanifolds.

\begin{theorem}\label{general-structure-of-aK-dim8}
Let $X=(M,J)$ be an $8$-dimensional nilmanifold endowed with an invariant complex structure $J$.
If there exists an (invariant or not) AK metric on $X=(M,J)$, then $M$ has first Betti number $b_1(M)\geq 6$ and its underlying Lie algebra is at most $2$-step.
More precisely, $X$ 
admits a $(1,0)$-coframe $\{\omega^k\}_{k=1}^4$ satisfying 
\begin{equation}\label{caso-i}
d\omega^1 =  d\omega^2  =  d\omega^3  =  0,\quad	\quad 
d\omega^4 = \sum_{1\leq j < k\leq 3} A_{jk}\,\omega^{jk} + \sum_{1\leq j, k\leq 3} B_{jk}\,\omega^{j\bar{k}},  
\end{equation}
for some complex coefficients $A$'s and $B$'s. 
\end{theorem}

\begin{proof}
We first prove that if an AK metric exists on $X$ then $J$ is quasi-nilpotent. More specifically, 
we use Proposition~\ref{clasifJ-SnN} to show that every SnN complex structure $J$ satisfies the AK obstruction.

On the one hand, for an SnN complex structure $J$ in Family I we have:

$\bullet$ if $b\neq0$, the real $(1,1)$-form $\alpha=i(\omega^{1\bar{3}}-\omega^{\bar{1}3})$ satisfies
$$
dd^c\alpha=2i\partial\db\big(i(\omega^{1\bar{3}}-\omega^{\bar{1}3})\big)=
-2
\partial\db(\omega^{1\bar{3}}-\omega^{\bar{1}3})=4b\,\omega^{12}\wedge \omega^{\bar{1}\bar{2}};
$$

$\bullet$ if $b=0$, then $a\neq 0$ and the real form $\alpha=i\omega^{3\bar{3}}$ satisfies
$$
dd^c\alpha=2i\partial\db(i\omega^{3\bar{3}})=-2\partial\db(\omega^{3\bar{3}})=
2a^2\,\omega^{12}\wedge \omega^{\bar{1}\bar{2}}+4\,\omega^{14}\wedge \omega^{\bar{1}\bar{4}}.
$$

On the other hand, for those complex structures in Family II we get:

$\bullet$ if $\mu\neq0$, then the real $(1,1)$-form $\alpha=i(\omega^{1\bar{2}}-\omega^{\bar{1}2})$ satisfies
$$
dd^c\alpha=2i\partial\db\big(i(\omega^{1\bar{2}}-\omega^{\bar{1}2})\big)=
-2
\partial\db(\omega^{1\bar{2}}-\omega^{\bar{1}2})=
-4
\mu\,\omega^{12}\wedge \omega^{\bar{1}\bar{2}};
$$

$\bullet$ if $\mu=0$, then $\varepsilon=1$ and for the real form $\alpha=i\omega^{3\bar{3}}$ we have 
$$
dd^c\alpha=2i\partial\db(i\omega^{3\bar{3}})=-2\partial\db(\omega^{3\bar{3}})=
6\,\omega^{12}\wedge \omega^{\bar{1}\bar{2}}.
$$

 So, if $X$ admits an AK metric, then $J$ must be quasi-nilpotent. Let $\frg$ be the Lie algebra underlying $M$.
By Proposition~\ref{no}  the pair $(\frg,J)$ is a central extension of some pair $(\tilde\frg,\tilde J)$. 
In fact, one can always choose a proper central $J$-invariant ideal $\mathfrak b$ in $\frg$
of real dimension $2$, in such a way that $(\frg_{\frb},J_{\frb})$ is equivalent to a $6$-dimensional
nilpotent Lie algebra with complex structure $(\tilde{\frg},\tilde{J})$.
Hence, it follows from Proposition~\ref{extensions-of-non-AK} that 
the existence of an AK metric on $(\frg,J)$ requires that 
$(\frg,J)$ is  
a central extension of some of the cases listed in Proposition~\ref{classif-SKT}.
Therefore, from now on, we will focus on those
quasi-nilpotent complex structures that are extensions of $(\frh_{1},J)$, $(\frh_{2},J)$, $(\frh_{2},J'_{y})$, $(\frh_{4},J)$, $(\frh_{5},J_{y})$ or $(\frh_{8},J)$, as listed in Proposition~\ref{classif-SKT}. 

Let us first study the extensions of $(\frh_{2},J)$, $(\frh_{2},J'_{y})$, $(\frh_{4},J)$ and $(\frh_{5},J_{y})$. 
Due to Remark~\ref{ecus-extensiones},
any such extension can be described by complex structure equations of the form 
\begin{equation}\label{ecus-extension}
 \left\{
 \begin{array}{rcl}
  d\omega^1 &\!\!\!=\!\!\!& d\omega^2 =  0,\\[2pt]
  d\omega^3 &\!\!\!=\!\!\!& \rho\,\omega^{12} + \omega^{1\bar1} + \lambda\,\omega^{1\bar2} + D\,\omega^{2\bar2},\\[2pt]  
  d\omega^4 &\!\!\!=\!\!\!& A_{12}\,\omega^{12}+A_{13}\,\omega^{13}+A_{23}\,\omega^{23}+
         B_{11}\,\omega^{1\bar{1}}+B_{12}\,\omega^{1\bar{2}}+B_{13}\,\omega^{1\bar{3}}\\[2pt]
   && + B_{21}\,\omega^{2\bar{1}}+ B_{22}\,\omega^{2\bar{2}}+ B_{23}\,\omega^{2\bar{3}}+
          B_{31}\,\omega^{3\bar{1}}+ B_{32}\,\omega^{3\bar{2}}+ B_{33}\,\omega^{3\bar{3}}.
 \end{array}\right.
\end{equation}
Here $(\rho,\lambda,D)$ takes the concrete values $(0,0,i)$ for $(\frh_{2},J)$, $(1,1,1+iy)$ with $y>0$ for $(\frh_{2},J'_{y})$, $(1,1,1)$ for $(\frh_{4},J)$, and $(1,0,\nicefrac{1}{2}+iy)$ with $y\in[0,\nicefrac{\sqrt{3}}{2})$ for $(\frh_{5},J_{y})$. 
Furthermore, the coefficients $A$'s and $B$'s belong to $\mathbb C$ and satisfy the condition $d(d\omega^4)=0$ imposed by the Jacobi identity. 

Let us first observe that one can always take $B_{11}=0$.
In fact, if $B_{11}\neq0$ we can consider the $(1,0)$-form $\omega'^{\,4}=\omega^4-B_{11}\,\omega^3$ so that the coefficient of $\omega^{1\bar{1}}$ in $d\omega'^{\,4}$ vanishes; then, simply denote again by $A_{12},B_{12},B_{22}$ the new complex coefficients $A'_{12},B'_{12},B'_{22}$.

Now, the condition $d(d\omega^4)=0$ is equivalent to 
the following equations
\begin{equation}\label{Jacobi}
\left\{\begin{array}{l}
B_{33}=0,\\
A_{23}+\lambda B_{13}-B_{23}+\rho B_{31}=0,\\
D A_{13}-\lambda A_{23}-\bar{D} B_{13}-\rho B_{32}=0,\\
\rho B_{13}+\lambda B_{31} - B_{32}=0,\\
\rho B_{23}+D B_{31}=0.
\end{array}\right.
\end{equation}

A direct calculation from \eqref{ecus-extension} 
shows that 
$\partial\db(\omega^{2\bar{4}})=-\bar A_{23}\,\omega^{12}\!\wedge \omega^{\bar{1}\bar{2}}$. Then, 
\begin{equation}\label{A23}
dd^c(\omega^{2\bar{4}}+\omega^{\bar{2}4})=-4\,\Imag\!(A_{23})\,\omega^{12}\!\wedge \omega^{\bar{1}\bar{2}}, \quad
dd^c \big(i(\omega^{2\bar{4}}-\omega^{\bar{2}4})\big)=4\,\Real\!(A_{23})\,\omega^{12}\!\wedge \omega^{\bar{1}\bar{2}}, 
\end{equation}
which implies that $A_{23}=0$ is a necessary condition for the existence of AK metrics.

Similarly, from \eqref{ecus-extension} 
and using that $A_{23}=0$, we get 
$\partial\db(\omega^{1\bar{4}})=-\bar D \bar A_{13} \,\omega^{12}\!\wedge \omega^{\bar{1}\bar{2}}$. Consequently, 
$D A_{13}=0$ is another necessary condition for the existence of AK metrics.

Therefore, the equations~\eqref{Jacobi} reduce to $D A_{13}=A_{23}=B_{33}=0$ and the following system
$$\begin{pmatrix}
\rho & 0 & \lambda & -1\\
0 & \rho & D & 0\\
\lambda & -1 & \rho & 0\\
\bar D & 0 & 0 & \rho
\end{pmatrix}
\begin{pmatrix}
B_{13} \\ B_{23} \\ B_{31} \\ B_{32}
\end{pmatrix} =
\begin{pmatrix}
0 \\ 0 \\ 0 \\ 0
\end{pmatrix}.$$
Since $\rho\in\{0,1\}$, the determinant of the previous matrix is
$$\Delta = \rho\,(2\,\Real\!(D)-\lambda^2+1)+|D|^2.$$
Using the concrete values $(\rho,\lambda,D)$ for 
$\frh_{2}$, $\frh_{4}$ and $\frh_{5}$ given above we easily get that $\Delta\neq0$ and, since $D\neq 0$ for all these cases, we
conclude that the coefficients $A_{13},B_{13},B_{23},B_{31},B_{32}$ vanish. 
To sum up, at this point we have 
\begin{equation}\label{coefs-nulos}
A_{13}=A_{23}=B_{11}=B_{13}=B_{23}=B_{31}=B_{32}=B_{33}=0.
\end{equation}
Therefore, the last differential in \eqref{ecus-extension} reduces to 
$$
d\omega^4 = A_{12}\,\omega^{12}+B_{12}\,\omega^{1\bar{2}}+ B_{21}\,\omega^{2\bar{1}}+ B_{22}\,\omega^{2\bar{2}},
$$
and then one has 
$$
\begin{array}{rl}
&\partial\db(\omega^{3\bar{4}})=-(\rho\, \bar A_{12}+\lambda \bar B_{12} - \bar B_{22} )\,\omega^{12}\!\wedge \omega^{\bar{1}\bar{2}},\\[4pt]
&\partial\db(\omega^{4\bar{4}})=-\big( |A_{12}|^2 +|B_{12}|^2 +|B_{21}|^2 \big)\,\omega^{12}\!\wedge \omega^{\bar{1}\bar{2}}.
\end{array}
$$
If some of the two coefficients above is not zero then the AK obstruction holds. 
Hence, the existence of AK metrics requires
$$
 B_{22} = \rho\, A_{12} +\lambda B_{12},\quad
 |A_{12}|^2 +|B_{12}|^2 +|B_{21}|^2=0.
$$
Note that these conditions imply $d\omega^4=0$. In other words, 
the only AK extensions of the SKT cases $(\frh_{2},J)$, $(\frh_{2},J'_{y})$, $(\frh_{4},J)$, and $(\frh_{5},J_{y})$ are the trivial ones. Thus, the corresponding nilmanifolds are products of an SKT 6-nilmanifold and a complex torus, which can be expressed by equations of the form~\eqref{caso-i}.

\vskip.1cm

It remains to study the cases $\frh_1$ and $\frh_8$. For the first one, it is clear that any (AK or not) extension is given by~\eqref{caso-i}. 
For the second one, we observe that any extension of $(\frh_{8},J)$ satisfies  
\eqref{ecus-extension}-\eqref{Jacobi} with $\rho=\lambda=D=0$. Thus, 
$A_{23}=B_{23}$ and $B_{32}=0$, in addition to $B_{11}=B_{33}=0$. 
Moreover, from \eqref{A23} we get $A_{23}=B_{23}=0$ to avoid the AK obstruction. Observe that $\Delta=0$ in this case. Moreover, from $\partial\db(\omega^{3\bar{4}})= \bar B_{22} \,\omega^{12}\!\wedge \omega^{\bar{1}\bar{2}}$ one has 
that $B_{22}=0$, so we arrive at complex equations of the form
\begin{equation}\label{caso-ii}
 \left\{
 \begin{array}{rcl}
  d\omega^1 &\!\!\!=\!\!\!& d\omega^2  =  0, \\ 
  d\omega^3 &\!\!\!=\!\!\!& \omega^{1\bar1},\\
  d\omega^4 &\!\!\!=\!\!\!& A_{12}\,\omega^{12}+A_{13}\,\omega^{13}
  +B_{12}\,\omega^{1\bar{2}}+B_{13}\,\omega^{1\bar{3}}+ B_{21}\,\omega^{2\bar{1}}+ B_{31}\,\omega^{3\bar{1}}. 
 \end{array}
 \right.
\end{equation}
Now, a direct calculation gives 
$$
\begin{array}{rl}
\partial\db(\omega^{4\bar{4}})= \!\!&\!\! -( |A_{12}|^2 +|B_{12}|^2 +|B_{21}|^2 )\,\omega^{12}\!\wedge \omega^{\bar{1}\bar{2}}  -(A_{12}\bar A_{13}+\bar B_{12} B_{13}+B_{21}\bar B_{31})\, \omega^{12}\!\wedge \omega^{\bar{1}\bar{3}}\\[4pt]
\!\!&\!\!
-(\bar A_{12} A_{13}+B_{12} \bar B_{13}+\bar B_{21} B_{31})\, \omega^{13}\!\wedge \omega^{\bar{1}\bar{2}}\ -( |A_{13}|^2 +|B_{13}|^2 +|B_{31}|^2 )\,\omega^{13}\!\wedge \omega^{\bar{1}\bar{3}}\\[4pt]
= \!\!&\!\! - \big[ \omega^{1}\!\wedge ( A_{12} \,\omega^{2} + A_{13} \,\omega^{3}) \big] \wedge \big[ \omega^{\bar{1}}\!\wedge
( \bar A_{12} \,\omega^{\bar{2}} + \bar A_{13} \,\omega^{\bar{3}})  \big]\\[4pt]
\!\!&\!\! - \big[ \omega^{1}\!\wedge ( \bar B_{12} \,\omega^{2} + \bar B_{13} \,\omega^{3}) \big] \wedge \big[ \omega^{\bar{1}}\!\wedge
( B_{12} \,\omega^{\bar{2}} + B_{13} \,\omega^{\bar{3}})  \big]\\[4pt]
\!\!&\!\! - \big[ \omega^{1}\!\wedge (  B_{21} \,\omega^{2} + B_{31} \,\omega^{3}) \big] \wedge \big[ \omega^{\bar{1}}\!\wedge
( \bar B_{21} \,\omega^{\bar{2}} + \bar B_{31} \,\omega^{\bar{3}})  \big],
\end{array}
$$
so all the coefficients $A_{12}, A_{13},B_{12}, B_{13},B_{21},B_{31}$ must be zero  for the existence of an AK metric. Hence, we get $d\omega^4=0$ in \eqref{caso-ii}, i.e. a product of the SKT nilmanifold 
associated to $(\frh_8,J)$ and a torus.
\end{proof}

\begin{remark-question}\label{tesis} 
{\rm 
We notice that the results in \cite[Corollaries 5.1.9 and 5.1.10]{Latorre-thesis} must be corrected according to the previous theorem, since there does not exist any 3-step nilmanifold in eight dimensions admitting AK structures. 
Hence, a natural question arises:
\emph{is any AK nilmanifold at most 2-step?} 
}
\end{remark-question}

We observe that not all the complex nilmanifolds given in Theorem~\ref{general-structure-of-aK-dim8} and defined by the
equations~\eqref{caso-i} admit AK metrics. For instance, the choice of all the coefficients equal to zero except for $B_{11}=B_{22}=B_{33}=1$ defines an abelian complex structure satisfying the AK obstruction, since $\partial\db(\omega^{4\bar{4}})= 
2(\omega^{12}\!\wedge \omega^{\bar{1}\bar{2}} + \omega^{13}\!\wedge \omega^{\bar{1}\bar{3}}+ \omega^{23}\!\wedge \omega^{\bar{2}\bar{3}})$. 
In Corollary~\ref{cor-taus} below we provide a simple condition, depending on three real numbers $(\tau_{12},\tau_{13},\tau_{23})$ associated to the equations~\eqref{caso-i}, that ensures the existence of AK metrics. 
We next consider higher-dimensional  extensions of~\eqref{caso-i} as follows:

\begin{definition}\label{condition-no-AK-NO}
For every $m\geq 3$, let $X=(M,J)$ be a $2(m+1)$-dimensional nilmanifold endowed with an invariant complex structure $J$ defined by the complex equations
\begin{equation}\label{caso-i-extended}
d\omega^1 =  \cdots  =  d\omega^m  =  0,\quad	\quad 
d\omega^{m+1} = \sum_{1\leq j < k\leq m} A_{jk}\,\omega^{jk} + \sum_{1\leq j, k\leq m} B_{jk}\,\omega^{j\bar{k}},  
\end{equation}
with $A_{jk},B_{jk}\in\C$. 
Moreover, for every $1\leq u < v \leq m$, we define the real numbers
$$
\tau_{uv} = |A_{uv}|^2+|B_{uv}|^2+|B_{vu}|^2 -2\,\Real (B_{uu}\bar B_{vv}), 
\quad \mbox{ and } \quad
{\mathcal T} = \sum_{1\leq u < v \leq m} \tau_{uv}.
$$
\end{definition}

Notice that 
\begin{equation}\label{aK-invariant}
{\mathcal T} 
=\sum_{1\leq j < k\leq m} |A_{jk}|^2 + \sum_{1\leq j, k\leq m} |B_{jk}|^2 
- \Big| \sum_{1\leq j\leq m} B_{jj} \Big|^2.
\end{equation}

The following result generalizes the constructions in \cite[Theorem 2.3]{Chiose-Rasdeaconu} and \cite[Theorem 2.7]{FT}.

\begin{proposition}\label{taus-of-aK-dim-m}
Let $X=(M,J)$ be a complex nilmanifold defined by~\eqref{caso-i-extended}, and let $F$ be the Hermitian metric given by
\begin{equation}\label{F-n-dim-astheno-K}
F =  \frac{i}{2}\, \big( x_{11}\, \omega^{1\bar{1}} + \cdots +  x_{mm}\, \omega^{m\,\overline{m}} +\omega^{m+1\,\overline{m+1}} \big), 
\end{equation}
with $x_{11}, \ldots, x_{mm} \in \R^+$.  Then, $F$ is AK if and only if 
$\displaystyle\sum_{1\leq u<v\leq m} \frac{\tau_{uv}}{x_{uu} \, x_{vv}}=0$. 
In particular, the canonical metric (as well as any $F$ with $x_{11}=\cdots=x_{mm}$) is AK if and only if  
${\mathcal T}= 0$. 
\end{proposition}

\begin{proof}
By the definition of the metric $F$ it follows
$$
\begin{array}{ccl}
F^{m-1}  \!\!&\!\!=\!\!&\!\!   (m\!-\!1)!\, (\frac{i}{2})^{m-1} 
\Big(\sum_{1\leq u\leq m} x_{11}\cdots \widehat{x_{uu}}\cdots x_{mm}\,  \omega^{1\bar{1}\cdots \widehat{u\bar{u}}\cdots m\overline{m}} \\[4pt]
\!\!&\!\! \!\!&\!\!  
\phantom{(m\!-\!1)!\, (\frac{i}{2})^{m-1} \Big(} 
+ \sum_{1\leq u<v\leq m} x_{11}\cdots \widehat{x_{uu}}\cdots \widehat{x_{vv}}\cdots x_{mm} \, \omega^{1\bar{1}\cdots \widehat{u\bar{u}}\cdots \widehat{v\bar{v}}
\cdots m\overline{m}\, m+1\overline{m+1}} 
\Big).
\end{array}
$$
From the complex structure equations \eqref{caso-i-extended} we get
$$
\begin{array}{ccl}
\partial\db F^{m-1}  \!\!&\!\!=\!\!&\!\!   (m\!-\!1)!\, (\frac{i}{2})^{m-1} \!\!
\displaystyle\sum_{1\leq u<v\leq m} \! r_{uv} \,\, \omega^{1\bar{1}\cdots \widehat{u\bar{u}}\cdots \widehat{v\bar{v}}
\cdots m\overline{m}} \wedge\big(\bar\partial\omega^{m+1}\!\wedge\partial\omega^{\overline{m+1}} -
   \partial\omega^{m+1}\!\wedge\bar\partial\omega^{\overline{m+1}} \big),
\end{array}
$$
where $r_{uv}=\frac{x_{11}\cdots x_{mm}}{x_{uu}x_{vv}}$.
Using again \eqref{caso-i-extended}, we have
$$
\bar\partial\omega^{m+1}\!\wedge\partial\omega^{\overline{m+1}} 
= \displaystyle\sum_{1\leq j,k,r,s \leq m} \!\! B_{jk} \overline{B}_{rs} \,\, \omega^{j\bar{k} \bar{r} s}
, 
\quad\quad 
\partial\omega^{m+1}\!\wedge\bar\partial\omega^{\overline{m+1}}= \sum_{
\begin{smallmatrix}
1\leq a<b\leq m\\
1\leq c<d\leq m
\end{smallmatrix}
} \!\! A_{ab}\,\bar{A}_{cd}\, \omega^{ab\bar{c}\bar{d}}.
$$
To compute $\partial\db F^{m-1}$, the possibilities are 
$(j,k,r,s)=(u,u,v,v)$, $(u,v,u,v)$, $(v,u,v,u)$ or $(v,v,u,u)$ for the first summand and 
$(a,b,c,d)=(u,v,u,v)$ for the second one. Therefore,
$$
\begin{array}{ccl}
\partial\db F^{m-1}  \!\!&\!\!=\!\!&\!\!   (m\!-\!1)!\, (\frac{i}{2})^{m-1} 
\Big( \displaystyle\sum_{1\leq u<v\leq m} r_{uv} \, ( - B_{uu} \overline{B}_{vv} + B_{uv} \overline{B}_{uv} + B_{vu} \overline{B}_{vu} \\[4pt]
\!\!&\!\! \!\!&\!\! \phantom{ (m\!-\!1)!\, (\frac{i}{2})^{m-1} 
\Big( \displaystyle\sum_{1\leq u<v\leq m} r_{uv} \, ( - B_{uu} \overline{B}_{vv}  }   - B_{vv} \overline{B}_{uu} + A_{uv} \overline{A}_{uv} ) \Big) 
 \, \omega^{1\bar{1}\cdots m\overline{m}}.
\end{array}
$$
Hence, the Hermitian metric $F$ is AK if and only if 
$$
\displaystyle\sum_{1\leq u<v\leq m} r_{uv} \, ( |A_{uv}|^2 - 2\, \Real( B_{uu} \overline{B}_{vv} ) + |B_{uv}|^2 + |B_{vu}|^2 )=0,
$$
that is, if and only if $\displaystyle\sum_{1\leq u<v\leq m} \frac{\tau_{uv}}{x_{uu}\,x_{vv}} =\displaystyle\sum_{1\leq u<v\leq m} \frac{r_{uv} \, \tau_{uv}}{x_{11}\cdots x_{mm}} =0$.
\end{proof}

When the real dimension of the nilmanifold is $8$, one has the following:

\begin{corollary}\label{cor-taus}
In the conditions of Proposition~\ref{taus-of-aK-dim-m}, let $m=3$ and consider ${\mathcal C}=\{(x,y,z)\in \R^3 \mid x,y,z \leq 0 \ \mbox{  or  }\  x,y,z \geq 0\}$ and ${\mathcal B}=(\R^3\backslash{\mathcal C})\cup \{(0,0,0)\}$. 
If $(\tau_{12},\tau_{13},\tau_{23}) \in {\mathcal B}$, then there exist Hermitian metrics $F$ that are AK. 
\end{corollary}

\begin{proof}
For $m=3$, the metric $F$ given by~\eqref{F-n-dim-astheno-K} is AK 
iff $x_{33}\tau_{12} + x_{22}\tau_{13} + x_{11}\tau_{23}=0$. Equivalently,
$
(\tau_{12},\tau_{13},\tau_{23})\cdot (x_{33},x_{22},x_{11}) = 0,
$
i.e. the vector $(\tau_{12},\tau_{13},\tau_{23})$ is orthogonal to $(x_{33},x_{22},x_{11})$ with respect to the standard Euclidean product in $\R^3$. Since one needs $(x_{33},x_{22},x_{11}) \in (\R^+)^3$, it is clear that there are solutions whenever $(\tau_{12},\tau_{13},\tau_{23}) \in {\mathcal B}$.
\end{proof}

\begin{remark}\label{rem} 
{\rm 
As a consequence of Proposition~\ref{taus-of-aK-dim-m}, for every complex nilmanifold $X=(M,J)$ defined by~\eqref{caso-i-extended} such that $(\frg,J)$ satisfies the AK obstruction (which implies that no AK metric exists on $X$), 
one has ${\mathcal T}\neq 0$. 
Notice that the converse of this assertion does not hold, since there exist AK metrics on nilmanifolds $X=(M,J)$ defined by complex equations of the form~\eqref{caso-i-extended} for which ${\mathcal T}\neq0$.
For instance, in \cite[Remark 2.6]{LU-cr} it is showed that the complex structures given by the equations 
$$
d\omega^1=d\omega^2=d\omega^{3}=0,\quad
d\omega^{4}= 
\omega^{1\bar{1}}+\omega^{2\bar{2}}-2\,\omega^{3\bar{3}},
$$
admit AK metrics (and also balanced metrics). In this case one has 
${\mathcal T}>0$. 

An interesting question is if the non-existence of AK metric on $X=(M,J)$ implies that $(\frg,J)$ satisfies the AK obstruction.
}
\end{remark}

\subsection{An application to geometric Bott-Chern formality}

Recall that  a Hermitian metric~$g$ on a compact complex manifold~$X$ 
is said to be \emph{geometrically-Bott-Chern-formal} if the wedge product of any pair of 
forms which are harmonic with respect to the associated Bott-Chern Laplacian $\Delta_{BC}$ is again harmonic, i.e. 
$({\mathcal H}^{\bullet,\bullet}_{\Delta_{BC}}(X,g),\wedge)$ is an algebra.
The complex manifold $X$ is called \emph{geometrically-Bott-Chern-formal} is there exists a geometrically-Bott-Chern-formal metric~$g$ on $X$. 
(For further details see, for instance, \cite{AT, ST-MathZ, ST-JGA}.)

In \cite{ST-MathZ, ST-JGA} Sferruzza and Tomassini investigate the interplay
between SKT metrics and geometrically-Bott-Chern-formal metrics on nilmanifolds endowed
with an invariant complex structure. 
In \cite{ST-MathZ}, a complex structure $J$ defined by the equations \eqref{caso-i-extended} is said to be \emph{``of special type''}. 
The authors prove that, for any nilmanifold endowed with a complex structure of special type, 
the existence of an SKT metric implies the existence of a geometrically-Bott-Chern-formal metric 
(see~\cite[Theorem 7.4]{ST-MathZ}). 
Note that in complex dimension 3 any $J$ admitting an SKT metric must be of special type, 
but this is not true in any higher dimensions. 

As a consequence 
of  Theorem~\ref{general-structure-of-aK-dim8}, we obtain the following general result in complex dimension~4.

\begin{proposition}\label{Extension-ST}
Let $X=(M, J)$ be an 8-dimensional nilmanifold endowed with an invariant complex structure. 
If an SKT metric and an AK metric coexist on $X$, then $X$ is geometrically-Bott-Chern-formal. 
\end{proposition}

\begin{proof}
By Theorem~\ref{general-structure-of-aK-dim8}, the existence of an AK metric on $X$ implies that the complex structure $J$ must be of special type. Now, given any SKT metric on $X$, the symmetrization process produces an \emph{invariant} SKT metric $\tilde{g}$ on $X$. Therefore, we are in the conditions of \cite[Theorem 7.4]{ST-MathZ}, where it is proved that such an invariant $\tilde{g}$ is geometrically-Bott-Chern-formal because $J$ is of special type. 
\end{proof}

It is worthy to note that there are examples of SKT metrics that are also AK, but the hypothesis in the proposition above is weaker since the metrics can be different. 
However, for any nilmanifold equipped with $J$ of special type, if an SKT metric exists, then all the \emph{left-invariant} Hermitian metrics are SKT and AK (see \cite[Lemma 3.2 and Theorem 3.3]{LLT}).

Examples of complex 4-dimensional SKT nilmanifolds with nilpotent complex structures
admitting no geometrically-Bott-Chern-formal metrics are constructed in \cite[Proposition 4.8]{ST-JGA}. One can easily check that these examples are not AK, accordingly to Proposition~\ref{Extension-ST}; in fact, their nilpotent complex structures are not of special type.

\smallskip

Moreover, for nilmanifolds endowed with complex structures of special type, one has that any \emph{invariant} geometrically-Bott-Chern-formal metric is SKT~ \cite[Proposition 4.2]{ST-JGA}.
Using this result and Theorem~\ref{general-structure-of-aK-dim8} we obtain the following:

\begin{proposition}\label{Extension-ST-bis}
Let $X$ be an AK nilmanifold of complex dimension $4$. Then,
there exists an invariant geometrically-Bott-Chern-formal metric on $X$ if and only if $X$ is also SKT. 
\end{proposition}

Propositions~\ref{Extension-ST} and~\ref{Extension-ST-bis} suggest that the AK condition might play an important role in relation to geometric-Bott-Chern-formality (see also \cite{CT-arxiv}). In this setting, the following natural question arises: 
\emph{Do Proposition~\ref{Extension-ST} and/or Proposition~\ref{Extension-ST-bis} hold 
in any dimension?}

\section{An obstruction to the existence of strongly Gauduchon metrics}\label{sec:no-sG}

\noindent 
In this section we are interested in the relation of AK structures with other special Hermitian metrics, specifically, strongly Gauduchon (sG) and balanced metrics. 
Recall that by \cite{Belgun,FG,COUV} the well-known symmetrization process can be applied to sG and balanced metrics, in such a way 
that the existence of these metrics on $X$ is equivalent to its existence on $(\frg,J)$. Nonetheless, the following result will also
be useful to distinguish those complex structures on nilmanifolds that do not admit any sG (or balanced) metric.

\begin{proposition}\label{condition-for-non-sG} {\bf{\emph{(sG obstruction)}}}
Let $X$ be a compact complex manifold of complex dimension $n\geq 2$.
Suppose that there is a $\db$-closed $(0,1)$-form $\alpha$ such that
\begin{equation}\label{sG-condition}
\partial\alpha= \sum_{k} c_k\,\eta^k\wedge \bar{\eta}^k,
\end{equation}
with $\eta^k$ being $(1,0)$-covectors and with $c_k=a_k+i b_k \neq 0$ being complex constants such that either all their non-zero real parts $a_k$'s or all their non-zero imaginary parts $b_k$'s have the same sign. Then $X$ does not admit any sG metric.
\end{proposition}

\begin{proof}
Suppose $F$ is sG, i.e. $\partial F^{n-1}=\db\gamma$ for some form $\gamma$ of bidegree $(n,n-2)$. Then we get  
$$
\begin{array}{rl}
\partial (F^{n-1}\wedge \alpha) \!\!&\!\! = \partial F^{n-1}\wedge\alpha + F^{n-1}\wedge\partial\alpha \\[4pt]
\!\!&\!\! =
\db\gamma \wedge\alpha + F^{n-1}\wedge\partial\alpha \\[4pt] 
\!\!&\!\!=
\db(\gamma \wedge\alpha) + \sum_{k} (a_k+i b_k)\, F^{n-1}\!\wedge \eta^k\wedge \bar{\eta}^k,
\end{array}
$$
where the last equality follows from the $\db$-closedness of the form $\alpha$. By Stokes theorem 
$$
0 = \int_X \partial (F^{n-1}\wedge\alpha) = \sum_{k} a_k\, \int_X F^{n-1}\!\wedge \eta^k\wedge \bar{\eta}^k 
+i \left( \sum_{k} b_k\, \int_X F^{n-1}\!\wedge \eta^k\wedge \bar{\eta}^k \right),$$
because $\int_X \db(\gamma \wedge\alpha)=0$. 
Now, the condition for $c_k=a_k+i b_k$ in the statement gives a contradiction.
\end{proof}

The following result shows that the sG obstruction on a complex nilmanifold is transferred to any of its  extensions.  

\begin{proposition}\label{extensions-of-non-sG} 
Let $X=(M,J)$ be a complex nilmanifold of complex dimension $n\geq 3$ with underlying Lie algebra $\frg$. Suppose that $(\frg,J)$ is an extension of a pair $(\tilde{\frg},\tilde{J})$ on which there exists a $(0,1)$-form 
satisfying the conditions in Proposition~\ref{condition-for-non-sG}. 
Then, the nilmanifold $X=(M,J)$ does not admit any sG metric.
\end{proposition}

\begin{proof}
The proof is similar to that of Proposition~\ref{extensions-of-non-AK}.
\end{proof}

Let us denote by $\frk\frt$ the (real) 4-dimensional Lie algebra endowed with the complex structure  defined by the equations 
\begin{equation}\label{kt-complex-ecus}
d\omega^1=0,\quad\quad d\omega^2=\omega^{1\bar{1}}.
\end{equation}
The real Lie algebra underlying $\frk\frt$ is $(0,0,0,12)$, and this nilpotent Lie algebra only admits one complex structure up to equivalence.  We will refer to $\frk\frt$ as \emph{the Kodaira-Thurston algebra}.

\smallskip

The following result provides a simple condition that can be applied to many complex nilmanifolds
to discard the existence of an sG metric on them, 
even when the complex structure $J$ is of non-nilpotent type (see Subsection~\ref{subsec-non-nilp} below).

\begin{corollary}\label{cor-extensions-of-non-sG} 
Let $X=(M,J)$ be a complex nilmanifold of complex dimension $n\geq 3$ with underlying Lie algebra $\frg$. If $(\frg,J)$ is an extension of the Kodaira-Thurston algebra, then $X$ does not admit any sG metric.
\end{corollary}

\begin{proof}
It follows from \eqref{kt-complex-ecus} that the $\db$-closed $(0,1)$-form $\alpha=\omega^{\bar{2}}$ satisfies \eqref{sG-condition}, so the result is a direct consequence of Proposition~\ref{extensions-of-non-sG}. 
\end{proof}

\smallskip 

In the rest of this section we will study 
the existence of AK and sG metrics  on large families of complex nilmanifolds. 

\smallskip 

We first recall the well-known result, noticed in \cite{MT}, that 
a compact complex parallelizable manifold cannot be AK unless it is a complex torus. 
In fact, let $X$ be a compact complex parallelizable manifold of complex dimension $n$, i.e. there exist $n$ global holomorphic $1$-forms on $X$ which are linearly independent  at every point.  
Then, by \cite{Wang} 
$X$ is isomorphic to the quotient of a simply connected \emph{complex} Lie group by a lattice, and 
by \cite{JY} every holomorphic $1$-form on a compact AK manifold is closed. 
On the other hand, any left-invariant Hermitian metric on an unimodular complex Lie group is balanced by \cite{AG}. 
Therefore, any compact complex parallelizable manifold, which is not a complex torus, is an sG (indeed, balanced) manifold not admitting any AK metric.

\smallskip 

Next, we study some interesting classes of complex nilmanifolds satisfying the obstruction given in Proposition~\ref{condition-for-non-sG} 
to the existence of sG metrics. 
Specifically, in Subsection~\ref{subsec-yesAK-nonSG} 
we provide nilmanifolds that are AK but not sG, whereas the class considered in Subsection~\ref{subsec-max-nilp} is neither AK nor sG. 
Subsection~\ref{subsec-almost-abelian} is devoted to complex
nilmanifolds whose underlying Lie algebra is almost abelian, 
and 
Subsection~\ref{subsec-non-nilp} to the class of nilmanifolds endowed with non-nilpotent complex structures in eight dimensions.

\subsection{A class of AK nilmanifolds not admitting sG metrics}\label{subsec-yesAK-nonSG} 
For every integer $n\geq 3$, 
let $X$ be a complex nilmanifold defined by a $(1,0)$-coframe 
$\{\omega^k\}_{k=1}^n$ 
satisfying the equations  
\begin{equation}\label{abel-dim-n}
d\omega^1 = \cdots  =  d\omega^{n-1}  =  0, \quad\quad  
d\omega^n = \sum_{k=1}^{n-1} B_{kk}\,\omega^{k\bar{k}},
\end{equation}
where $B_{kk}\in \C$. 

Let us consider the Hermitian metric $F=\frac{i}{2}\,\sum_{k=1}^{n} \omega^{k\bar{k}}$. 
As a consequence of Proposition~\ref{taus-of-aK-dim-m}, one has that 
$F$ is an AK metric 
if and only if $\mathcal T=0$, namely,
$$
\sum_{k=1}^{n-1} |B_{kk}|^2 = \left|\sum_{k=1}^{n-1} B_{kk}\right|^2.
$$
If we choose, for instance, $B_{11}=\cdots=B_{n-2\,n-2}=1$, then the latter condition is equivalent to $\Real B_{n-1\,n-1}=-\frac{n-3}{2}$. Furthermore, taking $\Imag B_{n-1\,n-1}\ne 0$, we get that the nilmanifold satisfies the sG obstruction. In conclusion, we have proved the following

\begin{proposition}\label{abel-AK-noSG-dim-n}
Given any integer $n\geq 3$, let $X=(M,J)$ be a complex $n$-dimensional nilmanifold defined by \eqref{abel-dim-n} with $B_{11}=\cdots=B_{n-2\,n-2}=1$ and  $B_{n-1\,n-1}=-\frac{n-3}{2} + i\, b$, where $b\in\R^*$. Then, $X$ is 
an AK nilmanifold that does not admit any sG (or balanced) metric.
\end{proposition}

Observe that the complex structure $J$ given in Proposition~\ref{abel-AK-noSG-dim-n} is abelian. 
We also note that the Lie algebra underlying $X$ is isomorphic to 
$(0^{2n-2},12, \,34+\cdots+ 2n\!-\!3\, 2n\!-\!2)$.

\subsection{Nilmanifolds with maximal nilpotent complex structures}\label{subsec-max-nilp} 

In \cite{GaoZZ}, Gao, Zhao and Zheng introduce the notion of \emph{maximal nilpotent} complex structures $J$. These are nilpotent complex structures for which the length of the ascending $J$-compatible series $\{\fra_{k}(J)\}_{k}$ 
given by  
\eqref{sucesion-a} is maximal, that is, $t=n$. Here $t$ denotes the smallest integer such that $\fra_{t-1}(J)\neq \fra_{t}(J)=\frg$, and $2n$ is the (real) dimension of $\frg$. Notice that $t$ is denoted as $\nu(J)$ in \cite{GaoZZ}. In the paper, they study the structure of maximal nilpotent complex structures and, in particular, solve a question in \cite{CFGU-dolbeault}.

In the following result we prove that maximal nilpotent complex structures do not admit AK or sG metrics. More generally:

\begin{proposition}\label{max-nilpotent}
Let $X=(M,J)$ be a complex nilmanifold of complex dimension $n\geq 3$, and let $\frg$ be its underlying Lie algebra. 
If the ascending $J$-compatible series $\{\fra_{k}(J)\}_{k=0}^t$ satisfies 
$$
\dim \fra_{t-3}(J)=2n-6 \ < \ \dim \fra_{t-2}(J)=2n-4 \ < \ \dim \fra_{t-1}(J)=2n-2 \ < \ \dim \fra_{t}(J)=2n,
$$
then $X$ does not admit neither AK nor sG (or balanced) metrics. 
\end{proposition}

\begin{proof}
Let us consider the proper $J$-invariant ideal $\frb=\fra_{t-2}(J)$. 
Then, the pair 
$(\frg,J)$ is an extension of $(\frg_\frb,J_\frb)$. Notice that the ascending $J_\frb$-compatible series of the 4-dimensional pair $(\frg_\frb,J_\frb)$ satisfies 
$$
\dim \fra_{1}(J_\frb) = 2 \ < \ \dim \fra_{2}(J_\frb) = 4 = \dim \frg_\frb,
$$
so it admits a basis of $(1,0)$-forms 
$\{\omega^1_\frb,\omega^2_\frb\}$ 
satisfying 
$d\omega^1_\frb=0$ and $d\omega^2_\frb=\omega^{1\bar{1}}_\frb$. 
In other words, $(\frg,J)$ is an extension of the Kodaira-Thurston algebra.
Therefore, by Corollary~\ref{cor-extensions-of-non-sG}, the complex nilmanifold $X$ does not admit any sG metric.

Let us now suppose that $n\geq 4$ and consider the proper $J$-invariant ideal $\frb=\fra_{t-3}(J)$. 
Then, the pair $(\frg,J)$ 
can also be seen as 
an extension of the 6-dimensional 
nilpotent Lie algebra with complex structure 
$(\frg_\frb,J_\frb)$, for which the ascending $J_\frb$-compatible series satisfies 
$$
\dim \fra_{1}(J_\frb) = 2 \ < \ \dim \fra_{2}(J_\frb) = 4 \ < \ \dim \fra_{3}(J_\frb) = 6 = \dim \frg_\frb.
$$
Therefore, $(\frg_\frb,J_\frb)$ admits a basis of $(1,0)$-forms 
$\{\omega^1_\frb,\omega^2_\frb,\omega^3_\frb\}$ 
 satisfying $$
d\omega^1_\frb=0, \quad 
d\omega^2_\frb=\omega^{1\bar{1}}_\frb,\quad 
d\omega^3_\frb=A\, \omega^{12}_\frb+ B\, \omega^{1\bar{1}}_\frb+ C\,\omega^{1\bar{2}}_\frb+ D\, \omega^{2\bar{1}}_\frb,
$$ 
for some $A,B,C,D\in \C$ with $(A,C,D)\neq(0,0,0)$. 
Thus, for the real $2$-form $\alpha=i\,\omega^{3\bar{3}}_\frb$ we get 
$$
(dd^c\alpha)^{2,2}=2i\partial\db(i\omega^{3\bar{3}}_\frb)=-2\partial\db(\omega^{3\bar{3}}_\frb)=
2(|A|^2+|C|^2+|D|^2)\,\omega^{12}_\frb\!\wedge \omega^{\bar{1}\bar{2}}_\frb \neq 0,
$$
so the AK obstruction~\eqref{2-condition} is satisfied on $(\frg_\frb,J_\frb)$.  Hence, Proposition~\ref{extensions-of-non-AK} implies that 
the complex nilmanifold $X$ does not admit any AK metric.

Finally, we notice that the remaining case $n=3$ corresponds, by the hypothesis of the proposition, to a 6-dimensional pair $(\frg,J)$ with $t=3$ and ascending $J$-compatible series with 
$\dim \fra_{1}(J) = 2$, $\dim \fra_{2}(J) = 4$, and $\fra_{3}(J) =\frg$. Hence, a similar argument as above allows us to conclude that $X$ does not admit any AK (SKT) metric.
\end{proof}

\begin{remark}\label{nonSKT}
We note here that the complex nilmanifolds in the conditions of Proposition~\ref{max-nilpotent} do not admit any SKT metric. 
The result has been proved above for $n=3$, so let us suppose that $n\geq 4$. 
On the one hand, Arroyo and Nicolini showed in \cite[Theorem 1.2]{AN} that any complex nilmanifold
admitting an SKT metric is either a torus or 2-step nilpotent. On the other hand, by \cite[Theorem 3]{GaoZZ},
any complex structure on a 2-step nilpotent Lie algebra has $2 \leq t \leq 3$. 
Now, simply observe that the condition on the ascending $J$-compatible series $\{\fra_{k}(J)\}_{k}$ in Proposition~\ref{max-nilpotent} implies that $t\geq 4$; in particular, the nilmanifold is $s$-step with 
$s\geq 3$.
\end{remark}

\begin{corollary}\label{max-nilpotent-cor1}
Let $X=(M,J)$ be a complex nilmanifold of complex dimension $n\geq 3$. If $J$ is maximal nilpotent, then $X$ does not admit neither AK nor sG (or balanced) metrics. 
\end{corollary}

\begin{proof}
By definition, $J$ is maximal nilpotent if $t=n$, so one necessarily has that 
$\dim \fra_{k}(J)=2k$ for $0\leq k\leq n$. Thus, for $0\leq l\leq 3$, we get $\dim \fra_{t-l}(J)=2t-2l=2n-2l$ and Proposition~\ref{max-nilpotent} can be applied.
\end{proof}

\begin{corollary}\label{max-nilpotent-cor2}
Let $X=(M,J)$ be a complex nilmanifold of complex dimension $n\geq 3$, and let $\frg$ be its underlying Lie algebra. 
If $J$ is nilpotent and the ascending central series $\{\frg_{k}\}_{k}$ of the Lie algebra satisfies 
$\dim \frg_{k}=2k$ for $0\leq k\leq n$, then $X$ does not admit neither AK nor sG (or balanced) metrics. 
\end{corollary}

\begin{proof}
By the definition of the ascending $J$-compatible series $\{\fra_{k}(J)\}_{k}$ one always has that 
$\fra_{k}(J) \subseteq \frg_k=\{X\in\frg \mid [X,\frg]\subseteq \frg_{k-1}\}$ for every $k\geq 0$, 
so $\dim \fra_{k}(J)\leq\dim \frg_{k}=2k$. 
Moreover, each $\fra_{k}(J)$ is $J$-invariant and, since $J$ is nilpotent, $\{\fra_{k}(J)\}_{k=0}^t$ is strictly increasing. This implies that 
$2k\leq \dim \fra_{k}(J)$ for $0\leq k\leq t$. Therefore, $\fra_{k}(J)=\frg_k$ for every $k\geq 0$ and, in particular, $t=n$. Therefore, $J$ is maximal nilpotent and one can apply Corollary~\ref{max-nilpotent-cor1}.
\end{proof}

\subsection{Almost abelian complex nilmanifolds}\label{subsec-almost-abelian}  

We now focus our attention on the class of complex
nilmanifolds whose underlying Lie algebra is almost abelian, that is, it admits
an abelian ideal of codimension one.
Notice that by~\cite[Proposition 3.10]{AABRW} the complex structure is unique up to equivalence.

Let $(\frg,J)$ be a nilpotent almost abelian Lie algebra with complex
structure. Let $2n$ be the real dimension of $\frg$. It is worthy to note that the nilpotency step of $\frg$ is at most $n$. 
By~\cite[Corollary~3.12]{AABRW}, there exist integer numbers $k_0\geq 0$ and 
$k_i>0$ $(i=1,\ldots,r)$, with   
$k_0 + k_1 + \cdots+ k_r=n-1$, so that the complex structure equations of $(\frg,J)$ are given as follows:

\vskip.2cm

\noindent $\bullet$ 
If $k_0=0$, there is a basis of $(1,0)$-forms 
$\{ \alpha,\beta_1^1,\ldots,\beta^1_{k_1},\ldots,\beta_1^r,\ldots,\beta^r_{k_r} \}$ 
satisfying 
\begin{equation}\label{structure-equs-non-abelian-k0-0}
d\alpha=0,\qquad 
d\beta_i^l =
\left\{\begin{aligned}
& 0, & \textit{ for }  i=1,\\[-2pt]
& (\alpha+\overline\alpha)\wedge \beta_{i-1}^l, & \textit{ for } i\neq 1,
\end{aligned}\right.
\quad \mbox{ {\it where} } 1\leq l \leq r.
\end{equation}

\vskip.2cm

\noindent $\bullet$ 
If $k_0\geq 1$, there exists a basis of $(1,0)$-forms 
$\{ \alpha,\beta_1^0,\ldots,\beta^0_{k_0},\ldots,\beta_1^r,\ldots,\beta^r_{k_r} \}$ 
such that 
\begin{equation}\label{structure-equs-non-abelian-k0-no-0}
d\alpha=0,\qquad 
d\beta_i^l =
\left\{
\begin{aligned}
& \alpha\wedge\overline\alpha, & \textit{ for } i=1  \textit{ and } l=0,\\
& 0, & \textit{ for } i=1 \textit{ and }  0<l\leq r,\\[-2pt]
& (\alpha+\overline\alpha)\wedge \beta_{i-1}^l, & \textit{ for }  i\neq 1 \textit{ and }0\leq l \leq r.
\end{aligned}
\right.
\end{equation}

\vskip.2cm

\begin{proposition}\label{almost-ab-prop}
Let $X=(M^{2n},J)$ be an almost abelian complex nilmanifold, not a torus.  
\begin{itemize}
\item[\textit{(i)}] $X$ has an sG (or balanced) metric if and only if $k_0=0$. 
\item[\textit{(ii)}] $X$ has an AK metric if and only if $k_0=k_1=\cdots=k_r=1$ (i.e. its underlying Lie algebra is isomorphic to $\frk\frt\times \fra$, where $\frk\frt$ is the Kodaira-Thurston algebra and $\fra$ is abelian).
\end{itemize}
\end{proposition}

\begin{proof}
For the proof of (i), we first notice that if $k_0\geq 1$ then \eqref{structure-equs-non-abelian-k0-no-0} implies that the Lie algebra of $X$ is an extension of $\frk\frt$, so it  follows from Corollary~\ref{cor-extensions-of-non-sG} that $X$ does not admit any sG metric. For $k_0=0$, we directly get from \eqref{structure-equs-non-abelian-k0-0} that the Hermitian metric 
$$
F=\frac{i}{2}\,\alpha\wedge\overline\alpha + \frac{i}{2}\,\sum_{l=1}^r \sum_{j=1}^{k_l} \beta_j^l\wedge\overline\beta_j^l 
$$
satisfies $d F^{n-1}=0$, so it is balanced and thus, sG.

For the proof of (ii), let us first suppose
that there is 
some integer number $k_j$, with $1\leq j\leq r$, such that $k_j\geq 2$. 
Then, by \eqref{structure-equs-non-abelian-k0-0}--\eqref{structure-equs-non-abelian-k0-no-0} we have $d\beta^l_2= (\alpha+\overline\alpha)\wedge \beta^l_{1}$. Since $d\alpha=0$, we observe that the real 2-form 
$i\,\beta_2^l\wedge\overline\beta_2^l$ satisfies 
$$
\big( dd^c (i \beta_2^l\wedge\overline\beta_2^l) \big)^{2,2}
=-2\partial\db(\beta_2^l\wedge\overline\beta_2^l) 
=4\, \alpha\wedge\beta_1^l\wedge\overline\alpha\wedge\overline\beta_1^l, 
$$
so the AK obstruction \eqref{2-condition} holds. 
Now, if $k_0\geq 2$ then \eqref{structure-equs-non-abelian-k0-no-0} implies that $d\alpha=0$, $d\beta^0_1=\alpha\wedge\overline\alpha$, and $d\beta^0_2= (\alpha+\overline\alpha)\wedge \beta^0_{1}$, so a direct calculation shows that the real 2-form $i\,\beta_2^0\wedge\overline\beta_2^0$ satisfies the AK obstruction. 
Consequently, one needs $k_j\leq 1$ for every $j=0,\ldots,r$ to avoid the AK obstruction.
Since $k_0\geq 0$, $k_i>0$ for $1\leq i\leq r$, and $k_0+\ldots+k_r=n-1$,
one indeed has $k_1=\ldots=k_r=1$ and $k_0=n-1-r\in\{0,1\}$. 
At the sight of~\eqref{structure-equs-non-abelian-k0-0}--\eqref{structure-equs-non-abelian-k0-no-0} 
and considering that $X$ is not a torus, so $\frg$ is not abelian, the only possibility 
is having $k_0=1$ and thus $r=n-2$. 
This means that $\frg\simeq \frk\frt\times \fra$, and then  
the metric
$$
F=\frac{i}{2}\big( \alpha\wedge\overline\alpha + \beta_1^0\wedge\overline\beta_1^0 + \beta_1^1\wedge\overline\beta_1^1 +\cdots +\beta_1^r\wedge\overline\beta_1^r \big)
$$
is astheno-K\"ahler, i.e. $\partial\db F^{n-2}=0$. 
\end{proof}

\subsection{Nilmanifolds with non-nilpotent complex structures}\label{subsec-non-nilp}  

Now, we turn our attention to the class of non-nilpotent complex structures in real dimension 8. 
We already know from Section~\ref{sec:estructurageneral} that there are no AK metrics on those complex nilmanifolds having this type of complex structures.
In this section, we study 
sG metrics and identify those non-nilpotent complex structures not admitting these metrics. The existence of sG metrics 
in the remaining non-nilpotent complex structures is later analyzed in Section~\ref{sec:Finvariante}.

Recall that the complex structure equations for SnN complex structures appear as 
Families I and II in Proposition~\ref{clasifJ-SnN}.

\begin{proposition}\label{SnN-no-sG}
For any SnN complex structure with $(\varepsilon,\nu)\ne(0,0)$ in Family I or $\nu\ne0$
in Family II,  the corresponding complex nilmanifolds $X$  
do not admit any sG metric.
The same holds for all the extensions of $X$.
\end{proposition}

\begin{proof}
Let $(\frg,J)$ be an $8$-dimensional real nilpotent Lie algebra with SnN complex structure 
defined by the equations in the Family I. 
Note that the parameters $\varepsilon,\nu$ in these
equations satisfy $\varepsilon,\nu\in\{0,1\}$.
Let $\{z_k=x_k+i\,y_k\}_{k=1}^4$ 
denote the dual basis of $\{\omega^k\}_{k=1}^4$. 
Then, $\frb=\langle x_3,y_3,x_4,y_4 \rangle$ is a proper $J$-invariant (non-central) ideal in $\frg$. If $\varepsilon=1$, we observe that 
$(\frg,J)$ is a $\frb$-extension of the Kodaira-Thurston algebra so, by Corollary~\ref{cor-extensions-of-non-sG}, the corresponding complex nilmanifold does not admit any sG metric. 
Furthermore, the $(0,1)$-form $\omega^{\bar{4}}$  
satisfies $\db\omega^{\bar{4}}=0$ and we can write the $(1,1)$-form $\partial\omega^{\bar{4}}$
as follows:
$$
\begin{array}{rl}
\partial\omega^{\bar{4}} \!\!&\!\!  = i \nu\, \omega^{1\bar{1}} - b\, \omega^{2\bar{2}} - i \delta (\omega^{1\bar{3}}  - \omega^{3\bar{1}} ) \\[4pt] 
\!\!&\!\!  =
(1+i \nu)\, \omega^{1\bar{1}} - b\, \omega^{2\bar{2}} + \omega^{3\bar{3}} 
- (i \omega^{1}+ \delta \omega^{3})\wedge (-i \omega^{\bar{1}}+ \delta \omega^{\bar{3}}).
\end{array}
$$
Hence, we get an equality as in \eqref{sG-condition} with coefficients $c_1=1+i \nu$, $c_2=-b$, $c_3=1$ and $c_4=-1$. Clearly, if $\nu=1$ then the $(0,1)$-form $\alpha=\omega^{\bar{4}}$  
satisfies the sG obstruction of Proposition~\ref{condition-for-non-sG}, so there are no sG metrics when $\nu=1$.

Now, let $(\frg,J)$ be defined by the equations in the Family II. We consider the $\db$-closed $(0,1)$-form $\omega^{\bar{4}}$ and, similarly to the previous case, we write the $(1,1)$-form $\partial\omega^{\bar{4}}$ as 
$$
\begin{array}{rl}
\partial\omega^{\bar{4}} \!\!&\!\!  = i \nu\, \omega^{1\bar{1}} +\mu\, \omega^{2\bar{2}} 
- i b (\omega^{1\bar{2}}  - \omega^{2\bar{1}} )
- i  (\omega^{1\bar{3}}  - \omega^{3\bar{1}} ) \\[4pt]
\!\!&\!\!  =
(2 \!+\! i \nu)\, \omega^{1\bar{1}} + ( \mu \!+\! b^2)\, \omega^{2\bar{2}} + \omega^{3\bar{3}} 
- (i \omega^{1}+ b\, \omega^{2})\wedge (\overline{i \omega^{1}+ b\, \omega^{2}}) 
- (i \omega^{1}+  \omega^{3})\wedge (\overline{i \omega^{1}+  \omega^{3}}).
\end{array}
$$
Therefore, there are no sG metrics when $\nu=1$.

Finally, the non-existence of sG metrics on any extension follows from Proposition~\ref{extensions-of-non-sG}.
\end{proof}

In eight dimensions there are also non-nilpotent complex structures $J$ that are quasi-nilpotent, which  
by Definition~\ref{def-tipos-Js}, are known as WnN complex structures. 
By \cite[Theorem~3.1]{LU-racsam}, every $(\frg,J)$ with $J$ of WnN type admits a basis of $(1,0)$-forms 
$\{\omega^k\}_{k=1}^4$ 
in terms of which the complex structure equations of $(\frg,J)$ express as 
\begin{equation}\label{structure-equs-WnN}
\left\{
\begin{split}
d\omega^1 &= 0,\\[-4pt]
d\omega^2 &= \omega^{13} + \omega^{1\bar3},\\[-4pt]
d\omega^3 &= i\,\varepsilon\,\omega^{1\bar1} + i\,\delta\,\omega^{1\bar2} 
	- i\,\delta\,\omega^{2\bar1},\\[-4pt]
d\omega^4 &= a\,\omega^{12}+B\,\omega^{1\bar 1} 
	+\nu\,\big(\omega^{23}+2\,\delta\,\varepsilon\,\omega^{1\bar 3}+ \omega^{2\bar 3}\big),
\end{split}
\right.
\end{equation}
for some tuple $(\varepsilon,\delta,\nu,a,B)$ where $\varepsilon,\nu\in\{0,1\}$, $\delta=\pm 1$, $a\in\R$ and $B\in \C$. 
Now, if $\{z_k=x_k+i\,y_k\}_{k=1}^4$ denotes the dual
basis of $\{\omega^k\}_{k=1}^4$, then $\frb=\langle x_4,y_4 \rangle$ is a proper $J$-invariant central ideal in $\frg$. So, $(\frg,J)$ is a central extension of $\frh_{19}^-$ when $\varepsilon=0$ and a central extension of $\frh_{26}^+$ when $\varepsilon=1$.

Let us recall that $\frh_{19}^-$ and $\frh_{26}^+$ have exactly two complex structures each, up to equivalence, 
defined by the first three equations in \eqref{structure-equs-WnN} with $\delta=\pm 1$ and either $\varepsilon=0$ or $\varepsilon=1$, respectively. We next prove that $\frh_{26}^+$ endowed with any complex structure $J$, as well as any higher-dimensional extension of $(\frh_{26}^+,J)$, admits no sG metrics.

\begin{proposition}\label{WnN-no-sG}
Let $J$ be any complex structure on the Lie algebra $\frh_{26}^+$. Then, the corresponding complex nilmanifolds $X$ (or any extension of them) do not admit any sG metric. In particular, there are no sG metrics when $\varepsilon\ne0$
in the equations~\eqref{structure-equs-WnN}.
\end{proposition}

\begin{proof}
The $(0,1)$-form $\omega^{\bar{3}}$  
satisfies $\db\omega^{\bar{3}}=0$ and the $(1,1)$-form $\partial\omega^{\bar{3}}$
can be expressed as
$$
\partial\omega^{\bar{3}}= i \varepsilon\, \omega^{1\bar{1}} - i \delta (\omega^{1\bar{2}}  - \omega^{2\bar{1}} )
=
(1+i \varepsilon)\, \omega^{1\bar{1}} + \omega^{2\bar{2}} 
- (i \omega^{1}+ \delta \omega^{2})\wedge (-i \omega^{\bar{1}}+ \delta \omega^{\bar{2}}).
$$
Thus, we get an equality as in \eqref{sG-condition} with coefficients $c_1=1+i \varepsilon$, $c_2=1$, and $c_3=-1$. Clearly, if $\varepsilon=1$,  the $(0,1)$-form $\alpha=\omega^{\bar{3}}$  
satisfies the sG obstruction of Proposition~\ref{condition-for-non-sG}. 
Finally, Proposition~\ref{extensions-of-non-sG} implies the non-existence of sG metrics on extensions. 
\end{proof}

In Section~\ref{sec:Finvariante} we complete the classification of non-nilpotent complex structures in dimension~8 that admit sG (or balanced) metrics. Nonetheless, the 
results we have obtained at this point are enough to derive a (topological) obstruction for the existence of sG metrics on complex nilmanifolds of complex dimension less than or equal to 4. 
As far as we know, the only known generic obstruction to the existence of sG metrics is  the following intrinsic
characterization 
in terms of non-existence of certain currents, due to Popovici \cite[Proposition 4.3]{Popovici-invent}: a compact complex manifold
X carries an sG metric if and only if there is no non-zero current $T$ of
type $(1, 1)$ on $X$ such that $T \geq 0$ and $T$ is $d$-exact on $X$.

\smallskip

The following result suggests a possible restriction on the first Betti number of sG nilmanifolds~$(M,J)$. Another restriction would be given in terms of the nilpotency step of $M$.

\begin{corollary}\label{Betti-for-sG-dim8}
Let $M$ be a $2n$-dimensional nilmanifold with $n\geq 2$ 
endowed with an invariant complex structure $J$.
Suppose that $(M,J)$ carries an sG metric. 
\begin{itemize}
\item[$\bullet$] If $J$ is nilpotent, then $M$ is $s$-step with $s\leq n-1$; moreover, its first Betti number satisfies $b_{1}(M)\geq 4$.
\item[$\bullet$] If $J$ is WnN and $n\leq 5$, then $M$ is $s$-step with $s\leq n$ and $b_{1}(M)\geq 3$.
\item[$\bullet$] If $J$ is SnN and $n\leq 4$, then $M$ is $s$-step with $s\leq n$ and $b_{1}(M)\geq 3$.
\end{itemize}
\end{corollary}

\begin{proof}
Recall that the de Rham cohomology $H^*_{\mathrm{dR}}(M;\mathbb R)$ of the nilmanifold $M$ is isomorphic to the cohomology of its underlying Lie algebra $\frg$ \cite{Nomizu}. 

Let $J$ be a nilpotent complex structure on a $2n$-dimensional nilmanifold $M$, with $n\geq 2$.
We know from \cite[Proposition 15]{CFGU-dolbeault} that $b_{1}(M)\geq 3$. 
Moreover, 
by \cite[Theorem 12]{CFGU-dolbeault}, there is a basis $\{\omega^k\}_{k=1}^n$ of 
$(1,0)$-forms such that the structure equations 
of $(\frg,J)$ are 
$$d\omega^k=\sum_{i<j<k}A_{ij}\omega^{ij}+\sum_{r,s<k}B_{rs}\omega^{r\bar s},
\text{ \ \ for } 1\leq k\leq n.
$$
If $\{z_k=x_k+i\,y_k\}_{k=1}^n$ denotes the dual basis of $\{\omega^k\}_{k=1}^n$,
then $\frb=\langle x_3,y_3,\ldots, x_n,y_n \rangle$ is a proper $J$-invariant (possibly non-central) ideal in $\frg$
and $(\frg_{\frb},J_{\frb})$ is a $4$-dimensional real nilpotent Lie algebra with complex structure. 
Since $(M,J)$ has an sG metric, it follows from Corollary~\ref{cor-extensions-of-non-sG} that $(\frg,J)$ cannot be an extension of the Kodaira-Thurston algebra. Thus, it is an extension of a complex 2-torus, 
which implies that $b_1(M)\geq 4$. Concerning the nilpotency step $s$, one has by 
\cite[Proposition 10 (iii)]{CFGU-dolbeault} that $s\leq n$ for any nilpotent complex structure $J$. 
We next show that this upper bound can be improved  
to $s\leq n-1$ when $(\frg,J)$ is an extension of a complex 2-torus and $J$ is nilpotent. 
Indeed, let us suppose that $\frg$ is not abelian and $n\geq 3$ (otherwise the result is trivial).
In this case the ascending $J$-compatible series $\{\fra_{k}(J)\}_{k}$, 
given by  
\eqref{sucesion-a}, satisfies $\fra_{t-1}(J)\neq \fra_{t}(J)=\frg$ for some $t\geq 2$. 
For $n=3$, if $(\frg,J)$ is an extension of a complex 2-torus and $J$ is nilpotent, then 
$t= 2$ and 
$\frg/\fra_{1}(J)$ is abelian, so necessarily $s= 2$. For any complex dimension $n\geq 4$, since $\frg/\fra_{t-1}(J)$ 
inherits a nilpotent complex structure and it is an extension of a complex 2-torus, by induction we get 
that $s\leq n-1$.  

We now turn our attention to SnN complex structures when the complex dimension of the nilmanifold is $n\leq 4$.
We first recall that this type of complex structures do not exist in complex dimension 2. 
When $n=3$, sG metrics only exist for $\frg=\frh_{19}^-$, which is $3$-step  
and has $b_1=3$. In complex dimension 4, we can apply Proposition~\ref{SnN-no-sG} and so $J$ is a complex structure in Family I with $(\varepsilon,\nu)=(0,0)$, 
or in Family II with $\nu=0$. In the first case, $s=3$ and $b_1=5$, whereas in the second case one has $s=4$ and $b_1=3$. 
In conclusion, for SnN complex structures admitting sG metrics, up to complex dimension $n\leq 4$, we get that $b_{1}\geq 3$ and $s\leq n$.

Finally, it remains to study the case when the complex structure $J$ is WnN and $n\leq 5$. 
Recall that this type of complex structures only exists in complex dimension $\geq 4$~\cite{LUV1}. 
In complex dimension 4, Proposition~\ref{WnN-no-sG} implies that the Lie algebra must be an extension of $\frh_{19}^-$, so $b_{1}\geq 3$ and $s\leq 4$. 
In complex dimension 5, when $J$ is WnN then we have an extension of SnN complex structures in dimension 3 or 4~\cite{LU-racsam}. So, by the results already proved we get that $s\leq 5$ and $b_{1}(M)\geq 3$. 
\end{proof}

\begin{remark}
There exist $8$-dimensional nilmanifolds $M$ endowed with non-nilpotent complex structures that are $5$-step or whose first Betti number is $b_{1}(M)=2$. Note that Corollary~\ref{Betti-for-sG-dim8} provides a necessary but not sufficient condition: the nilmanifolds endowed with a WnN complex structure defined by \eqref{structure-equs-WnN} with $\varepsilon=1$, $a=\nu=0$ and any $B\in\C$, are $4$-step and have first Betti number equal to $4$, but they are not sG by Proposition~\ref{WnN-no-sG}.
\end{remark}

\begin{question}\label{obstruccion-sG} 
{\rm 
The following question seems natural in this setting: 
\emph{is the first Betti number of any sG nilmanifold greater than or equal to 3\,?} 
We notice that, by the argument in the proof of Corollary~\ref{Betti-for-sG-dim8}, the answer to this 
question reduces to study the class of SnN structures in complex dimension $n\geq 5$.
}
\end{question}

\section{Invariant Hermitian metrics on nilmanifolds}\label{sec:Finvariante}

\noindent 
In this section we study the properties of those Hermitian metrics on complex nilmanifolds $X=(M,J)$ 
that come from left-invariant ones on the corresponding Lie groups. 
This allows us to restrict our attention to the Lie algebra level.

Let $\frg$ be a $2n$-dimensional nilpotent Lie algebra endowed with a complex structure $J$. 
Our first result is that one can use the extension structure of $(\frg,J)$ to find an appropriate 
basis of $(1,0)$-forms 
in terms of which any Hermitian metric $F$ on $(\frg,J)$ can be expressed in a convenient form. 
To do so, observe that there always exists a proper $J$-invariant ideal $\frb$ in $\frg$ 
such that the pair $(\frg,J)$ is a $\frb$-extension of another pair $(\tilde{\frg},\tilde{J})$
(see Remark~\ref{existe-b}). We next show any basis of $(1,0)$-forms of $(\tilde{\frg},\tilde{J})$ can be extended to a basis of $(1,0)$-forms of $(\frg,J)$ in such a way that the following two properties hold: first, any $J$-Hermitian metric $F$ on $(\frg,J)$ is a ``diagonal extension'' of a Hermitian metric $\tilde{F}$ on $(\tilde{\frg},\tilde{J})$; second, the first $\tilde{n}$ complex structure equations in $(\frg,J)$ are precisely those of $(\tilde{\frg},\tilde{J})$. This will be very useful in the case of a quasi-nilpotent 
complex structure $J$, as $(\frg,J)$ is then a central extension of a pair $(\tilde{\frg},\tilde{J})$ and the characterization given in Proposition~\ref{no} can be applied.

\begin{lemma}\label{reduc-met}
Let $(\frg,J)$ and $(\tilde{\frg},\tilde{J})$ be two nilpotent Lie algebras with complex structures.
Suppose $(\frg,J)$ is an extension of $(\tilde{\frg},\tilde{J})$ with natural projection $\pi\colon \frg \to \tilde{\frg}$ 
and consider a Hermitian metric $F$ on $(\frg,J)$. 
Given any basis of $(1,0)$-forms $\big\{\tilde{\omega}^{k}\big\}_{k=1}^{\tilde{n}}$ for $(\tilde{\frg},\tilde{J})$, there is a basis of $(1,0)$-forms $\big\{\omega^{k}\big\}_{k=1}^{n}$ for $(\frg,J)$ with $\omega^{k}=\pi^*\tilde{\omega}^{k}$, for $1\leq k\leq \tilde{n}$,
such that 
$$
F=\pi^*\tilde{F}+D,
$$
where $\tilde{F}$ is a Hermitian metric on~$(\tilde{\frg},\tilde{J})$ and 
\begin{equation}\label{D-diagonal}
D=\frac{i}{2}\left( \omega^{\tilde{n}+1}\!\wedge\overline{\omega}^{\tilde{n}+1} +\cdots+\omega^{n}\!\wedge\overline{\omega}^{n} \right).
\end{equation}
Furthermore, if $J$ is quasi-nilpotent then 
the extended basis of $(1,0)$-forms $\big\{\omega^{k}\big\}_{k=1}^{n}$ can be chosen to additionally
satisfy the condition~\eqref{caracterizacion}.
\end{lemma}

\begin{proof}
Let $\frb\subset \frg$ be a proper $J$-invariant ideal of real dimension $2b$. Consider the quotient Lie algebra $\tilde{\frg}\cong \frg/\frb$ and denote by $2\tilde{n}$ its real dimension. Recall that if $\tilde{J}$ is the complex structure on $\tilde{\frg}$ induced by $J$ then $\pi^*\circ \tilde{J}=J\circ\pi^*$, where $\pi\colon \frg \to \tilde{\frg}$ denotes the natural projection. 
Let $\{\tilde{\omega}^{1},\ldots,\tilde{\omega}^{\tilde{n}}\}$ be a basis of $(1,0)$-forms 
for $(\tilde{\frg},\tilde{J})$.  The differential of $\tilde{\omega}^j$ satisfies $\tilde{d} \tilde{\omega}^j\in \bigwedge^{(2,0)+(1,1)}_{\C} \langle \tilde{\omega}^{1},\ldots,\tilde{\omega}^{\tilde{n}},\overline{\tilde{\omega}^{1}},\ldots,
\overline{\tilde{\omega}^{\tilde{n}}} \rangle$, 
for any $1\leq j\leq \tilde{n}$. Now, we define $\omega^j=\pi^* \tilde{\omega}^j$, $1\leq j\leq \tilde{n}$, and complete $\{\omega^{j}\}_{1}^{\tilde{n}}$ to a basis 
$\big\{ \eta^1=\omega^{1},\ldots,\eta^{\tilde{n}}=\omega^{\tilde{n}},\eta^{\tilde{n}+1},\ldots,\eta^{n=\tilde{n}+b} \big\}$ of $(1,0)$-forms on $(\frg,J)$.  

The Hermitian metric $F$ on $(\frg,J)$ can be written in this basis as
$$
F\ = \ \sum_{k=1}^n i\,x_{k\bar k}\,\eta^{k\bar k}
   \ +\sum_{1\leq k<l\leq n}\big( x_{k\bar l}\,\eta^{k\bar l}-\bar x_{k\bar l}\,\eta^{l\bar k} \big),
$$
where $x_{k\bar k}\in\mathbb R^{>0}$ and $x_{k\bar l}\in\mathbb C$, for $1\leq k< l\leq n$, are coefficients 
satisfying the appropriate conditions to ensure the positive definiteness of the metric. 

For simplicity, let us first suppose that $b=1$, i.e. $\tilde{n}=n-1$. Take the $(1,0)$-basis given by
\begin{equation}\label{cambio-base}
\omega^k=\eta^k,\ k=1,\ldots,n-1,\qquad 
\omega^n=\sqrt{2\,x_{n\bar n}}\left(\eta^n-i\,\sum_{k=1}^{n-1}\frac{x_{k\bar n}}{x_{n\bar n}}\,\eta^k\right).
\end{equation}
Now, in terms of the new basis $\big\{ \omega^{1},\ldots,\omega^{\tilde{n}},\omega^n \big\}$ 
of $(1,0)$-forms on $(\frg,J)$, the Hermitian metric $F$ expresses as $F=\pi^*\tilde{F}+D$, with 
$$
D= \frac{i}{2}\,\omega^{n\bar n}, 
\quad \mbox{ and }\quad 
\tilde{F} = \sum_{k=1}^{n-1} i\,\tilde{x}_{k\bar k}\,\tilde{\omega}^{k\bar k} 
+\!\sum_{1\leq k<l\leq n-1}\!\!\big(\tilde{x}_{k\bar l}\, \tilde{\omega}^{k\bar l}-\overline{\tilde{x}}_{k\bar l}\, \tilde{\omega}^{l\bar k}\big), 
$$
where for every $k,l$ one has
$$
\tilde{x}_{k\bar k}=\frac{x_{k\bar k}\,x_{n\bar n}-|x_{k\bar n}|^2}{x_{n\bar n}}, \qquad
    \tilde{x}_{k\bar l}=\frac{x_{k\bar l}\,x_{n\bar n}-i\,x_{k\bar n}\,\bar x_{l\bar n}}{x_{n\bar n}}.
$$
Observe that the positive definiteness of $F$ implies that
$\tilde{F}$ is a Hermitian metric on $(\tilde{\frg},\tilde{J})$. 

The proof for $b\geq 2$ is similar, but taking $b$ ($=n-\tilde{n}$) successive changes of basis as above to write the metric as $F=\pi^*\tilde{F}+D$. 
In greater detail, in the first step we consider \eqref{cambio-base} to express $F=F_{n-1}+D_{n}$, where $F_{n-1}\in \bigwedge^{1,1}_{\R} \langle {\omega}^{1},\ldots,{\omega}^{{n-1}},\overline{{\omega}^{1}},\ldots,
\overline{{\omega}^{{n-1}}} \rangle$ is a metric in the first $n-1$ terms of the basis,  and $D_{n}= \frac{i}{2}\,\omega^{n\bar n}$. In the second step we express similarly $F_{n-1}=F_{n-2}+D_{n-1}$, where $D_{n-1}= \frac{i}{2}\,\omega^{n-1\overline{n-1}}$. Finally, in the $b$-th step, one can write the metric $F$ as $F=F_{n-b}+D$, where $D=D_{\tilde{n}+1}+\cdots+D_{n}$. Notice that $F_{n-b}=\pi^*\tilde{F}$, where $\tilde{F}$ is a Hermitian metric on $(\tilde{\frg},\tilde{J})$.

\smallskip

Notice that the first $\tilde{n}$ forms in the basis (i.e. $\omega^1,\ldots,\omega^{\tilde{n}}$) remain unchanged in the process, so they satisfy the same complex structure equations as the 
basis of $(1,0)$-forms $\big\{\tilde{\omega}^{k}\big\}_{k=1}^{\tilde{n}}$ of $(\tilde{\frg},\tilde{J})$. 
Furthermore, the process preserves the condition~\eqref{caracterizacion} when the complex structure $J$ is quasi-nilpotent. 
\end{proof}

Thanks to this lemma, we can write any Hermitian metric on an extension $(\frg,J)$ of 
$(\tilde{\frg},\tilde{J})$ 
as $F=\pi^*\tilde{F}+D$. This allows to study the AK or the balanced condition in terms of the lower dimensional Lie algebra $(\tilde{\frg},\tilde{J})$ and of the complex structure equations that characterize the extension to $(\frg,J)$.  
This point of view turns to be especially useful for quasi-nilpotent complex structures:

\begin{proposition}\label{reducc-cond-met-prop}
Let $\frg$ be a nilpotent Lie algebra of real dimension $2n\geq 4$ with a quasi-nilpotent complex structure $J$.
Let $F$ be a Hermitian metric on $(\frg,J)$. 
By Proposition~\ref{no}, 
$(\frg,J)$ is a central extension of a $2\tilde{n}$-dimensional nilpotent Lie algebra $\tilde{\frg}$ with complex structure $\tilde{J}$.
By Lemma~\ref{reduc-met}, the metric $F$ can be written as $F=\pi^*\tilde{F}+D$, where $\tilde{F}$ is a Hermitian metric on~$(\tilde{\frg},\tilde{J})$ and 
$D$ is given by \eqref{D-diagonal}. Set $b=n-\tilde{n}$. Then:
\begin{itemize}
\item[\textit{(i)}] $F$ is balanced on $(\frg,J)$ if and only if $\tilde{F}$ is balanced on $(\tilde{\frg},\tilde{J})$, 
and 
\begin{equation}\label{balanced-en-extension}
\begin{array}{rl}
&\left[ (b-1) \binom{n-1}{b-1} \, \pi^*\tilde{F}^{n-b}\wedge D^{b-2} 
+ b\, \binom{n-1}{b} \, \pi^*\tilde{F}^{n-1-b}\wedge D^{b-1} \right]  \wedge d D=0.
\end{array}
\end{equation}
\item[\textit{(ii)}] $F$ is astheno-K\"ahler on $(\frg,J)$ if and only if $\tilde{F}$ is astheno-K\"ahler on $(\tilde{\frg},\tilde{J})$, 
and
\begin{equation}\label{AK-en-extension}
\begin{array}{rl}
0= \!\!\!& \, \Big[ b\binom{n-2}{b} \pi^*(\db\tilde{F}^{n-2-b})\wedge D^{b-1} + (b-1)\binom{n-2}{b-1} \pi^*(\db\tilde{F}^{n-1-b})\wedge D^{b-2} \Big]\wedge\partial D \\[7pt]
\!\!\!&\!\!\! - \Big[ b\binom{n-2}{b} \pi^*(\partial\tilde{F}^{n-2-b})\wedge D^{b-1} + (b-1)\binom{n-2}{b-1} \pi^*(\partial\tilde{F}^{n-1-b})\wedge D^{b-2} \Big]\wedge\db D \\[7pt]
\!\!\!&\!\!\! - \Big[ b(b-1)\binom{n-2}{b} \pi^*\tilde{F}^{n-2-b}\wedge D^{b-2} + (b-1)(b-2)\binom{n-2}{b-1} \pi^*(\tilde{F}^{n-1-b})\wedge D^{b-3} \\[4pt]
&
\quad + (b-2)(b-3)\binom{n-2}{b-2} \pi^*(\tilde{F}^{n-b})\wedge D^{b-4}  \Big]\wedge\partial D \wedge\db D \\[7pt]
\!\!\!&\!\!\! - \Big[ b\binom{n-2}{b} \pi^*\tilde{F}^{n-2-b}\wedge D^{b-1} + (b-1)\binom{n-2}{b-1} \pi^*(\tilde{F}^{n-1-b})\wedge D^{b-2} \\[4pt]
&
\quad + (b-2)\binom{n-2}{b-2} \pi^*(\tilde{F}^{n-b})\wedge D^{b-3}  \Big]\wedge\partial\db D.
\end{array}
\end{equation}
\end{itemize}
In \eqref{balanced-en-extension} and \eqref{AK-en-extension} 
we use the conventions
$\tilde{F}^{0}=1=D^{0}$ and $\tilde{F}^{\lambda}=0=D^{\lambda}$ if $\lambda<0$.  
\end{proposition}

\begin{proof}
For any $0\leq k\leq n-1$, we have $F^{n-k}= \sum_{j=0}^{n-k} \binom{n-k}{j}\,\pi^*\tilde{F}^{n-k-j}\wedge D^j$. 
Since $\tilde F$ is a metric on $(\tilde{\frg},\tilde{J})$ and $\dim\tilde{\frg}=2\tilde{n}$,
we observe that 
$\tilde{F}^{n-k-j}=0$ whenever $n-k-j>\tilde{n}$. 
Moreover, from~\eqref{D-diagonal} one has 
$D^j=0$ for every $j>b=n-\tilde{n}$. Hence,
\begin{equation}\label{potencia-met-gen}
F^{n-k}= \sum_{j=\max\{0,b-k\}}^{b} \binom{n-k}{j}\, \pi^*\tilde{F}^{n-k-j}\wedge D^j.
\end{equation}

For the proof of (i), we need to consider $k=1$, so \eqref{potencia-met-gen} reduces to
$$
\begin{array}{rl}
&F^{n-1}= \binom{n-1}{b-1}\, \pi^*\tilde{F}^{n-b}\wedge D^{b-1} + \binom{n-1}{b}\, \pi^*\tilde{F}^{n-1-b}\wedge D^b.
\end{array}
$$
Taking the differential we get 
$$
\begin{array}{rcl}
dF^{n-1}  \! \! \! &=&\! \! \!  \binom{n-1}{b-1}\, \pi^*(d\tilde{F}^{n-b})\wedge D^{b-1}  + \binom{n-1}{b}\, \pi^*(d\tilde{F}^{n-1-b})\wedge D^b \\[4pt]
&&\! \! \! + \binom{n-1}{b-1}\, \pi^*\tilde{F}^{n-b}\wedge d D^{b-1} + \binom{n-1}{b}\, \pi^*\tilde{F}^{n-1-b}\wedge dD^b.
\end{array}
$$
Note that $d\tilde{F}^{n-b}=0$ as $n-b=\tilde{n}$, which makes $\tilde{F}^{n-b}$ a volume form on~$(\tilde{\frg},\tilde{J})$. Thus, $F$ is balanced, i.e. $dF^{n-1}=0$, if and only if 
\begin{equation}\label{reduccion-condicion-balanced}
\begin{array}{ll}
&\binom{n-1}{b}\,\pi^*(d\tilde{F}^{n-1-b})\wedge D^b + \binom{n-1}{b-1}\, \pi^*\tilde{F}^{n-b}\wedge d D^{b-1} + \binom{n-1}{b}\, \pi^*\tilde{F}^{n-1-b}\wedge dD^b =0.
\end{array}
\end{equation}
Now, it follows from \eqref{D-diagonal} that $D^b=(\frac{i}{2})^b\, b!\,  \omega^{\tilde{n}+1}\!\wedge\overline{\omega}^{\tilde{n}+1}\wedge\cdots\wedge\omega^{n}\!\wedge\overline{\omega}^{n}$. 
Since $J$ is quasi-nilpotent, by \eqref{caracterizacion} in Proposition~\ref{no} we get that the first summand 
in the left-hand side of~\eqref{reduccion-condicion-balanced} is linearly independent with the other two. This means that the equality~\eqref{reduccion-condicion-balanced} holds 
if and only if $d\tilde{F}^{n-1-b}=0$ (i.e. $\tilde{F}$ is balanced on $(\tilde{\frg},\tilde{J})$)  
and $\binom{n-1}{b-1}\, \pi^*\tilde{F}^{n-b}\wedge d D^{b-1} + \binom{n-1}{b}\, \pi^*\tilde{F}^{n-1-b}\wedge dD^b=0$. The latter condition is precisely~\eqref{balanced-en-extension}. This proves (i).

In order to prove (ii), we take $k=2$ in \eqref{potencia-met-gen} and calculate $\partial\db F^{n-2}$ to get 
$$
\begin{array}{rl}
\partial\db F^{n-2}= \sum_{j=\max\{0,b-2\}}^{b} \binom{n-2}{j} \Big[ \!\!\!&\!\!\! \pi^*(\partial\db\tilde{F}^{n-2-j})\wedge D^j - \pi^*(\db\tilde{F}^{n-2-j})\wedge \partial D^j \\[4pt]
&+ \pi^*(\partial\tilde{F}^{n-2-j})\wedge \db D^j + \pi^*\tilde{F}^{n-2-j}\wedge \partial\db D^j \Big].
\end{array}
$$
When $j=b$ the term $\pi^*(\partial\db\tilde{F}^{n-2-b})\wedge D^b$ appears, and a similar argument to that given in the first part of this proof allows to prove that $\partial\db F^{n-2}=0$ implies $\partial\db\tilde{F}^{n-2-b}=0$. 
Since $n-2-b=\tilde{n}-2$, this means that $\tilde{F}$ is an astheno-K\"ahler metric on $(\tilde{\frg},\tilde{J})$.  
We also notice that $\partial\tilde{F}^{n-b}=\db\tilde{F}^{n-b}=\partial\db\tilde{F}^{n-b}=0$, because $\tilde{F}^{n-b}$ has top degree on $\tilde{\frg}$. Moreover, $\partial\db\tilde{F}^{n-1-b}=0$ as the only exact form of top degree on the nilpotent Lie algebra $\tilde{\frg}$ is the zero one. 
Therefore, the terms $\pi^*(\partial\db\tilde{F}^{n-1-b})\wedge D^{b-1}$ and $\pi^*(\partial\db\tilde{F}^{n-b})\wedge D^{b-2}$ in the expression of $\partial\db F^{n-2}$ are zero. 
Expanding $\partial\db F^{n-2}$ one arrives to the condition~\eqref{AK-en-extension}.
\end{proof}

The following result is a direct consequence of Proposition~\ref{reducc-cond-met-prop} and provides obstructions to the existence of balanced or AK metrics in the setting of quasi-nilpotent complex structures.

\begin{corollary}\label{reducc-cond-met-cor1}
Let $J$ be a quasi-nilpotent complex structure on a nilpotent Lie algebra $\frg$, and let~$\frb$ be any 
proper $J$-invariant subspace in the center of $\frg$. If $(\frg,J)$ is balanced (resp. astheno-K\"ahler), then $(\frg_\frb,J_\frb)$ is also balanced (resp. astheno-K\"ahler). 
In particular, if there exists 
some $\frb$ such that $(\frg_\frb,J_\frb)$ is not balanced (resp. astheno-K\"ahler), then $(\frg,J)$ does not admit any balanced (resp. astheno-K\"ahler) metric.
\end{corollary}

The case where the proper $J$-invariant central ideal $\frb$ has real dimension $2$ 
will be particularly useful.

\begin{corollary}\label{reducc-cond-met-cor2}
Let $J$ be a quasi-nilpotent complex structure on a nilpotent Lie algebra $\frg$ of real dimension $2n\geq 4$. Let $\frb$ be a $2$-dimensional $J$-invariant subspace in the center of $\frg$. 
By Lemma~\ref{reduc-met}, any Hermitian metric $F$ on $(\frg,J)$ can be written as $F=\pi^*\tilde{F}+D$,
where $\tilde{F}$ is a Hermitian metric on~$(\frg_\frb,J_\frb)$ and 
$D=\frac i 2\,\omega^{n\bar n}$. Furthermore:
\begin{itemize}
\item[\textit{(i)}] $F$ is balanced if and only if $\tilde{F}$ is balanced on~$(\frg_\frb,J_\frb)$ 
and $\tilde{F}^{n-2}\wedge d\omega^n=0$;
\item[\textit{(ii)}] $F$ is astheno-K\"ahler if and only if $\tilde{F}$ is astheno-K\"ahler on~$(\frg_\frb,J_\frb)$ 
and
$$
\begin{array}{rl}
&\tilde{F}^{n-3}\wedge\partial\bar\partial\omega^n+ \partial \tilde{F}^{n-3}\wedge\bar\partial\omega^n - \db\tilde{F}^{n-3}\wedge\partial\omega^n = 0,\\[4pt]
&\tilde{F}^{n-3}\wedge\big(\bar\partial\omega^n\wedge\partial\omega^{\bar n} -
   \partial\omega^n\wedge\db\omega^{\bar n}\big) = 0.
\end{array}
$$
\end{itemize}
Here, we use the conventions
$\tilde{F}^{0}=1$ and $\tilde{F}^{\lambda}=0$ if $\lambda<0$.
\end{corollary}

\begin{proof}
Since the real dimension of $\frb$ is $2$, we observe that the
real dimension of $\tilde{\frg}$ is $2\tilde{n}=2(n-1)$.
We then apply Proposition~\ref{reducc-cond-met-prop} for $b=1$, 
taking into account that $D=\frac{i}{2}\omega^{n\bar{n}}$. 

Firstly, we note that 
\eqref{balanced-en-extension} reduces to 
$$
(n-1)\frac{i}{2}\big[ (\pi^*\tilde{F}^{n-2} \wedge d \omega^{n})\wedge \omega^{\bar{n}} -  (\pi^*\tilde{F}^{n-2} \wedge d\omega^{\bar{n}}) \wedge\omega^{n} \big]=0.
$$
This is equivalent to $\pi^*\tilde{F}^{n-2}\wedge d\omega^n=0$, because $d\omega^n$ is a 2-form living on~$(\frg_\frb,J_\frb)$, i.e. $d\omega^n\in \bigwedge^2_{\C} \frg_\frb^*$, due to the fact that $J$ is quasi-nilpotent. 
This fact also allows us to remove $\pi^*$ from the equation, thus proving (i).

Similarly, \eqref{AK-en-extension} reduces to 
$$
\begin{array}{rcl}
0 &=&\!\!\!   \left( \pi^*(\db\tilde{F}^{n-3})\wedge \partial\omega^{\bar n} -
    \pi^*(\partial\tilde{F}^{n-3})\wedge\db\omega^{\bar n}- \pi^*\tilde{F}^{n-3}\wedge\partial\db\omega^{\bar n}\right)\wedge\omega^n  \\[4pt]
 \!\!\! &&\!\!\! + \left(\pi^*(\partial\tilde{F}^{n-3})\wedge \db\omega^n -
    \pi^*(\db\tilde{F}^{n-3})\wedge\partial\omega^n + \pi^*\tilde{F}^{n-3}\wedge\partial\db\omega^n\right)\wedge\omega^{\bar n} \\[5pt]
\!\!\! &&\!\!\!  +  \  \pi^*\tilde{F}^{n-3}\wedge(\db\omega^n\wedge\partial\omega^{\bar n} -
   \partial\omega^n\wedge\db\omega^{\bar n}).
\end{array}
$$
Since $J$ is quasi-nilpotent, the previous condition is equivalent to
$$
\begin{array}{rl}
&\pi^*\tilde{F}^{n-3}\wedge\partial\bar\partial\omega^n+ \pi^*(\partial \tilde{F}^{n-3})\wedge\bar\partial\omega^n - \pi^*(\db\tilde{F}^{n-3})\wedge\partial\omega^n = 0,\\[4pt]
&\pi^*\tilde{F}^{n-3}\wedge\big(\bar\partial\omega^n\wedge\partial\omega^{\bar n} -
   \partial\omega^n\wedge\db\omega^{\bar n}\big) = 0.
\end{array}
$$
Removing $\pi^*$ we get (ii).
\end{proof}

\subsection{Balanced and sG nilmanifolds with non-nilpotent complex structures}\label{subsec-clasif-dim8-sG-balanced-non-nilp} 

As a first application of the results above we classify the WnN complex structures in complex dimension~4 that admit balanced or sG metrics. 
Together with previous classifications for SnN complex structures, in Theorem~\ref{main-theorem-2} we present a complete study of the existence of balanced or sG metrics on nilmanifolds of real dimension 8 endowed with non-nilpotent complex structures.

Let $(\frg,J)$ be a nilpotent Lie algebra with $\dim_{\R}\frg=8$ and $J$ of WnN type. Recall that $J$ is defined by the structure equations \eqref{structure-equs-WnN}.
It follows from Proposition~\ref{WnN-no-sG} that $\varepsilon=0$ is a necessary condition for $J$ to admit sG metrics. Therefore, $(\frg,J)$ must be a central extension of $\frh_{19}^-$, 
i.e. $(\frg,J)$ is parametrized by 
\begin{equation}\label{structure-equs-WnN-bis}
d\omega^1 =0,\ \ 
d\omega^2 = \omega^{13} \!+ \omega^{1\bar3},\ \ 
d\omega^3 = i\,\delta (\omega^{1\bar2} \!- \omega^{2\bar1}),\ \ 
d\omega^4 = a\,\omega^{12} \!+ B\,\omega^{1\bar 1} 
	\!+ \nu (\omega^{23} \!+ \omega^{2\bar 3}),
\end{equation}
for some tuple $(\nu,a,B)$ where $\nu\in\{0,1\}$, $a\in\R$ and $B\in \C$. 
The possibilities for the tuple are given in \cite{LU-racsam}. Specifically, we will make use of the following result.

\begin{proposition}\label{probado-racsam}   {\rm  (\cite[ Theorem 3.1, Table 1 for $\varepsilon=0$, and Theorem 4.1]{LU-racsam}) }
Let $(\frg,J)$ be an $8$-dimensional nilpotent Lie algebra with WnN complex structure 
defined by \eqref{structure-equs-WnN-bis}. 
Up to equivalence, the tuple $(\nu, a, B)$ takes one the following values and $\frg$ is then isomorphic to the indicated Lie algebra:
\begin{equation*}
\begin{split}
&(0,0,0) \text{ on } \mathfrak f_1= \frh^-_{19}\times\mathbb R^2= (0^3,\,12,\,23,\,14-35,\,0,\,0), \\	
&(0,0,1) \text{ on }  \mathfrak f_3 = (0^3,\,12,\,13,\,23,\,15+26,\,0),\\
&(0,1,B) \text{ with $B\in\{0,1\}$ on }  \mathfrak f_4^0 = (0^3,\,13,\,14,\,23,\,26,\,16+24),\\
&(1,0,B) \text{ with $B\in\{0,1\}$ on }  
	\mathfrak f_5^B = (0^3,\,13,\,23,\,34,\,\frac{B}{2}\!\cdot\! 12+35,\,14+25), \\
&(1,1,B) \text{ with either $B\in\R^{\geq 0}$\! or $B\in{\mathbb H}^+$ on }  \mathfrak f_7 = (0^3,\,13,\,23,\,14+25,\, 2\!\cdot\! 14+34,\, 15+24+35).
\end{split}
\end{equation*}
Here, ${\mathbb H}^+ := \{ B\in\C \mid \Imag B>0 \}$ denotes the upper half-plane in the complex space.
\end{proposition}

In the following result we classify the WnN complex structures in eight dimensions admitting sG or balanced metrics. 

\begin{proposition}\label{clasif-WnN-sG}
An $8$-dimensional nilpotent Lie algebra $\frg$ endowed with a WnN complex structure $J$ admits
an sG metric if and only if $(\frg,J)$ is equivalent to a pair defined by~\eqref{structure-equs-WnN-bis} with $(\nu,a,B)$ taking the values  
$$
(0,0,0),\ (0,1,B\in\{0,1\}),\ (1,0,B\in\{0,1\}),\ (1,1,B\in\R^{\geq0}), \mbox{ or }  (1,1,B\in{\mathbb H}^+).
$$
In particular, every $J$ on the nilpotent Lie algebras 
$\mathfrak f_1$, $\mathfrak f_4^0$, $\mathfrak f_5^0$, $\mathfrak f_5^1$ and $\mathfrak f_{7}$ 
admits sG metrics. Moreover, all the cases above admit
balanced metrics with the exception of $(\nu,a,B)=(0,1,1)$,
which corresponds to a complex structure 
on the Lie algebra $\mathfrak f_4^0$. 
\end{proposition}

\begin{proof}
By Proposition~\ref{WnN-no-sG} it suffices to investigate those WnN complex structures 
defined by~\eqref{structure-equs-WnN-bis}; these correspond to the case when 
$(\frg,J)$ is a central extension of $\frh_{19}^-$. Note that the possible values for 
the tuple $(\nu,a,B)$ in~\eqref{structure-equs-WnN-bis}
are listed in Proposition~\ref{probado-racsam}. Moreover, recall that $\delta=\pm1$.

Let $F'$ be any Hermitian metric on $(\frg,J)$. 
As in the proof of Lemma~\ref{reduc-met}, we first apply \eqref{cambio-base} 
to get $F'=\pi^* F+\frac i 2\,\omega'^{4\bar 4}$, where $F$ is a Hermitian metric on 
$\frh_{19}^-$ that we will here write as
\begin{equation}\label{h19fundform}
2F=i(r^2\,\omega^{1\bar1} + s^2\,\omega^{2\bar2} +
t^2\,\omega^{3\bar3}) + u\,\omega^{1\bar2} -\bar u\,\omega^{2\bar1}
+ v\,\omega^{2\bar3} -\bar v\,\omega^{3\bar2}+ z\,\omega^{1\bar3}
-\bar z\,\omega^{3\bar1}.
\end{equation}
Since $F$ is positive definite, we have $r^2,\,s^2,\,t^2 \in \R^+$ and
$u,\,v,\,z\in \mathbb{C}$ satisfying $r^2s^2>|u|^2$,
$s^2t^2>|v|^2,\,r^2t^2>|z|^2$, and $r^2s^2t^2 + 2\,\Real(i\bar u\bar
vz)>t^2|u|^2 + r^2|v|^2 + s^2|z|^2$.

It is important to notice that \eqref{cambio-base} applied to our setting expresses as 
$$
\omega'^{\,k}=\omega^k,\ 1\leq k\leq 3,\qquad 
\omega'^{\,4} = \sqrt{2x'_{4\bar 4}}\, \omega^4 - i\,\sum_{k=1}^{3}\frac{ \sqrt{2} \, x'_{k\bar 4} }{\sqrt{x'_{4\bar 4}}}\,\omega^k,
$$
so the complex structure equations~\eqref{structure-equs-WnN-bis} with respect to the new basis 
$\{\omega'^{\,k}\}_{k=1}^4$ become
\begin{equation}\label{structure-equs-WnN-biss}
\left\{
\begin{split}
d\omega'^{\,1} &= 0,\\[-4pt]
d\omega'^{\,2} &= \omega'^{\,13} + \omega'^{\,1\bar3},\\[-4pt]
d\omega'^{\,3} &= i\,\delta\,(\omega'^{\,1\bar2} - \omega'^{\,2\bar1}),\\[-4pt]
d\omega'^{\,4} &= \sqrt{2x'_{4\bar 4}} \Big( a\,\omega'^{\,12}+B\,\omega'^{\,1\bar 1} 
	+\nu\,\big(\omega'^{\,23}+ \omega'^{\,2\bar 3}\big) \Big) \\
	&\ \  - i\,\frac{ \sqrt{2} \, x'_{2\bar 4} }{\sqrt{x'_{4\bar 4}}}\, \Big(\omega'^{\,13} + \omega'^{\,1\bar3}\Big) 
  - i\,\frac{ \sqrt{2} \, x'_{3\bar 4} }{\sqrt{x'_{4\bar 4}}}\, \Big(i\,\delta\,(\omega'^{\,1\bar2} - \omega'^{\,2\bar1})\Big) .
\end{split}
\right.
\end{equation}

According to Corollary~\ref{reducc-cond-met-cor2}, $F'$ is balanced if and only if 
$F$ satisfies $dF^2=0$ and $d\omega'^{\,4}\wedge F^2=0$.
A direct calculation using \eqref{h19fundform} and \eqref{structure-equs-WnN-biss} shows that these two conditions are equivalent to
\begin{equation}\label{aniquila}
u+\bar u=0, \quad\quad s^2 z + i uv =0,\quad\quad (s^2t^2-|v|^2)\,B +(u \bar{z}-ir^2\bar{v})\,\nu = 0.
\end{equation}
When $B=0$, the Hermitian metric $2F'=i(\omega'^{\,1\bar1} + \omega'^{\,2\bar2} + \omega'^{\,3\bar3}+ \omega'^{\,4\bar4})$ is balanced for any $\nu\in\{0,1\}$.

Suppose now that $B\ne 0$. If $\nu=0$ then there does not exist any balanced metric, because $s^2t^2-|v|^2>0$ and the third equation in \eqref{aniquila} is never satisfied. 
According to Proposition~\ref{probado-racsam}, this excludes 
the existence of balanced metrics for $(\nu,a,B)= (0,0,1)$ and $(0,1,1)$.

If one has $B\ne 0$ and $\nu=1$, then one can check that there is a space of balanced metrics. 
In particular, the choice 
$$
r=\sqrt{3}, \quad s=1, \quad u=i \sqrt{2}, \quad v= \frac{i(\sqrt{1+4t^2|B|^2}-1)}{2B}, \quad z=\sqrt{2}\,v,
$$
with any non-zero $t\in\R^*$, satisfies \eqref{aniquila} and defines a one-parameter family of balanced metrics. 

\smallskip

Therefore, we have proved the existence of balanced metrics when the tuple $(\nu,a,B)$ takes the following  values: 
$
(0,0,0),\ (0,1,0), \ (1,0,B\in\{0,1\}),\ (1,1,B\in\R^{\geq0}),\ (1,1,B\in{\mathbb H}^+).
$
Since balanced implies sG, it remains to check the existence of (non-balanced) sG metrics in the cases $(\nu,a,B)= (0,0,1)$ and $(0,1,1)$. 
The first case is not sG, since it is clear from \eqref{structure-equs-WnN-bis} that it is a (non-central) extension of the Kodaira-Thurston algebra (see Corollary~\ref{cor-extensions-of-non-sG}). 
Let us then study what happens for $(\nu,a,B)= (0,1,1)$.

Take $(\nu,a,B)= (0,1,1)$ in \eqref{structure-equs-WnN-bis} and consider the Hermitian metric $2F=i(\omega^{1\bar1} + \omega^{2\bar2} + \omega^{3\bar3}+ \omega^{4\bar4})$. 
One has
$$
\partial F^3= 
-\frac{3i}{4} \partial (\omega^{1\bar1 2\bar2 3\bar3} + \omega^{1\bar1 2\bar2 4\bar4} +\omega^{1\bar1 3\bar3 4\bar4} + 
\omega^{2\bar2 3\bar3 4\bar4})= -\frac{3i}{4}\omega^{1234\bar1\bar2\bar3}.
$$
Since $\db \omega^{1234\bar3\bar4} = -\omega^{1234\bar1\bar2\bar3}$, we conclude that $\partial F^3= \db\gamma$, 
for $\gamma=\frac{3i}{4}\omega^{1234\bar3\bar4}$, so the metric $F$ is sG. 

Finally, the nilpotent Lie algebras underlying the different complex structures admitting sG or balanced metrics 
come from Proposition~\ref{probado-racsam}.
\end{proof}

Now, we consider SnN complex structures in real dimension 8. Recall that their complex structure
equations are given in Proposition~\ref{clasifJ-SnN}, in terms of two families.
Moreover, by Proposition~\ref{SnN-no-sG}, the study of sG metrics reduces to investigate the Family I with $(\varepsilon,\nu)=(0,0)$ and the Family~II with $\nu=0$.  
The following classification proved in \cite{LUV2} will be needed. 

\begin{proposition}\label{prop-snn}  {\rm (\cite[Theorems 1.1 and 3.3 for Family I, $\varepsilon=\nu=0$, and Family II, $\nu=0$]{LUV2} } 

\noindent Let $(\frg,J)$ be an $8$-dimensional nilpotent Lie algebra with SnN complex structure. 

\smallskip
$\bullet$ \ If $J$ is defined by Family I with $\varepsilon=\nu=0$ then, up to equivalence, the pair $(a, b)$ takes
one of the following values and $\frg$ is isomorphic to the Lie algebra given in each case:
$$(1,0) \text{ on } \mathfrak g_1^0,
\qquad
(1,1) \text{ on } \mathfrak g_1^1,
\qquad 
(0,1) \text{ on } \mathfrak g_7,
$$

where 
\begin{equation*}
\begin{split}
&\mathfrak \frg_{1}^{\gamma} = (0^5,\, 13+15+24,\, 14-23+25,\, 16+27+\gamma\!\cdot\! 34), \ \ \gamma\in\{0,1\},\\
&\mathfrak g_7 = (0^5,\,15,\,25,\,16+27+34).
\end{split}
\end{equation*}

$\bullet$ \ If $J$ is defined by Family II with $\nu=0$ then, up to equivalence, the tuple $(\varepsilon,\mu, a, b)$ 
takes one of the following values and $\frg$ is isomorphic to the Lie algebra described in each case:
$$\begin{array}{ll}
(0, 1, a, 0) \text{ with $a\in\{0,1\}$ on } \mathfrak g_9^a, 
&(1, 0, a, 0) \text{ with $a\in\{0,1\}$ on } \mathfrak g_{10}^0,
\\[3pt]
(1, 0, a, b) \text{ with $a\in\{0,1\}$} \text{ and $b\in\R^*$ on } \mathfrak g_{10}^1, 
&(1, 1, 0, 0) \text{ on } \mathfrak g_{11}^{0,0},
\\[-4pt]
(1, 1, a, 0) \text{ with $a\in\R^*$ on } \mathfrak g_{11}^{1,0}, 
&(1, 1, a, b) \text{ with $a\in\R$ and $b\in\R^*$ on } \mathfrak g_{11}^{\frac{2\sqrt{3}|a|}{|b|},1},
\end{array}$$

where 
\begin{equation*}
\begin{split}
\mathfrak \frg_9^{\gamma} =&\, (0^3,\, 13,\, 23,\, 35,\, \gamma\!\cdot\! 12-34,\, 16+27+45),\\
\mathfrak \frg_{10}^{\gamma} =&\, (0^3,\, 13,\, 23,\, 14+25,\, 15+24,\, 16+ \gamma\!\cdot\! 25 +27),\\
\mathfrak \frg_{11}^{\alpha, \beta} =&\, \big(0^3,\, 13,\, 23,\, 14+25-35,\, \alpha \!\cdot\! 12+15+24+34,\, 16+27-45-\beta\!\cdot\!(2\!\cdot\! 25 + 35)\big),
\end{split}
\end{equation*}

with $\gamma\in\{0,1\}$, and  
$(\alpha, \beta)=(0,0)$, $(1,0)$, or $(\alpha, 1)$ for $\alpha\in[0,+\infty)$.
\end{proposition}

The SnN complex structures above admitting balanced metrics are classified in \cite[Proposition~3.2]{LUV-Fro}.
It remains to identify which of them carry a (non-balanced) sG metric.

\begin{proposition}\label{clasif-SnN-sG}
An $8$-dimensional nilpotent Lie algebra $\frg$ endowed with an SnN complex structure $J$ admits
an sG metric if and only if $(\frg,J)$ is equivalent to a pair determined by 
\begin{itemize}
\item Family I with $\varepsilon=\nu=0$, or
\item Family II with $\nu=0$.
\end{itemize}
In fact, every $J$ on the nilpotent Lie algebras 
$\frg_1^{\gamma}$, $\frg_7$, $\frg_9^{\gamma}$, $\frg_{10}^{\gamma}$, and $\frg_{11}^{\alpha,\beta}$
admits sG metrics. Furthermore, there are balanced metrics in all the cases above
except for those $(\frg,J)$ given by Family~II with $\nu=0$ and $(\varepsilon,\mu, a, b)=(1, 0, 1, b\in\R)$;
these correspond to complex structures
on $\mathfrak g_{10}^0$ for $(\varepsilon,\mu, a, b)=(1, 0, 1, 0)$
and $\mathfrak g_{10}^1$ for $(\varepsilon,\mu, a, b)=(1, 0, 1, b\in\R^*)$. 
\end{proposition}

\begin{proof}
By \cite[Proposition 3.2]{LUV-Fro}, every $J$ equivalent to a complex structure in Family~I with $\varepsilon=\nu=0$ admits balanced metrics. Moreover, if $J$ is equivalent to a complex structure in Family II with $\nu=0$, then it admits a balanced metric if and only if $(\varepsilon,\mu, a, b)\ne (1, 0, 1, b\in\R)$. 
Therefore, by Proposition~\ref{SnN-no-sG}, to
prove our result it suffices to show that there are sG metrics
on those $(\frg,J)$ defined by Family II with $\nu=0$ and $(\varepsilon,\mu, a, b)= (1, 0, 1, b\in\R)$.

Let $F$ be the Hermitian metric given by $2F=i(\omega^{1\bar1} + \omega^{2\bar2} + \omega^{3\bar3}+ \omega^{4\bar4})$. 
A direct calculation using the structure equations of Family II with 
$\nu=0$ and $(\varepsilon,\mu, a, b)= (1, 0, 1, b\in\R)$
given in Proposition~\ref{clasifJ-SnN}
shows  
$$
\partial F^3= 
-\frac{3i}{4} \partial (\omega^{1\bar1 2\bar2 3\bar3} + \omega^{1\bar1 2\bar2 4\bar4} +\omega^{1\bar1 3\bar3 4\bar4} + 
\omega^{2\bar2 3\bar3 4\bar4})= \frac{3i}{4}\omega^{1234\bar1\bar2\bar4},
$$
independently of the value of $b$.
Since $\db \omega^{1234\bar3\bar4} = \omega^{1234\bar1\bar2\bar4}$, we conclude that $\partial F^3= \db\gamma$ 
for $\gamma= \frac{3i}{4}\omega^{1234\bar3\bar4}$, so the metric $F$ is sG. 

Finally, the nilpotent Lie algebras underlying the different complex structures admitting sG or balanced metrics are derived from Proposition~\ref{prop-snn}.
\end{proof}

By the well-known symmetrization process, the existence of sG or balanced metrics on a complex nilmanifold $X=(M,J)$ is reduced to their existence on $(\frg,J)$. 
Hence, from Propositions~\ref{clasif-WnN-sG} and~\ref{clasif-SnN-sG} we get:

\begin{theorem}\label{main-theorem-2} 
Let $X=(M,J)$ be a complex 
nilmanifold of complex dimension $4$ 
endowed with a non-nilpotent complex structure $J$. If $X$ admits sG or balanced metrics, then its underlying nilpotent Lie algebra $\frg$ is isomorphic to 
$\mathfrak f_1$, $\mathfrak f_4^0$, $\mathfrak f_5^{\gamma}$, $\mathfrak f_{7}$ (when $J$ is WnN) or to $\frg_1^{\gamma}$, $\frg_7$, $\frg_9^{\gamma}$, $\frg_{10}^{\gamma}$, $\frg_{11}^{\alpha,\beta}$ (when $J$ is SnN), where $\gamma\in\{0,1\}$, and $(\alpha, \beta)=(0,0)$, $(1,0)$, or $(\alpha, 1)$ with $\alpha\in[0,+\infty)$.

\vskip.1cm

\noindent Furthermore, in these cases we have:

\begin{itemize}
\item Every non-nilpotent complex structure $J$ admits sG metrics; 
\item All the non-nilpotent complex structures, except for those equivalent to $(\nu,a,B)=(0,1,1)$ on~$\mathfrak f_4^0$,  $(\varepsilon,\mu, a, b)=(1, 0, 1, 0)$ on $\frg_{10}^{0}$ 
 or $(1, 0, 1, b\in\R^*)$ on $\frg_{10}^{1}$, admit balanced metrics. 
 \end{itemize}
\end{theorem}

\subsection{Nilmanifolds with invariant AK and balanced metrics in complex dimension 4}\label{subsec-dim8-AK-y-sG} 

Here we give a characterization of the 4-dimensional complex nilmanifolds that admit an invariant AK metric (possibly also being sG) and a balanced metric.

\begin{proposition}\label{classif-aK-dim8-invariante}
Let $M$ be an $8$-dimensional nilmanifold endowed with an invariant complex structure $J$. If $X=(M,J)$ admits an \textbf{invariant} astheno-K\"ahler metric $F'$, 
then 
there exists an invariant $(1,0)$-coframe $\{\omega^k\}_{k=1}^4$ on $X$ such that the metric writes in canonical form  
\begin{equation}\label{metrica-canonica}
F'=\frac{i}{2}\,(\omega^{1\bar{1}}+\omega^{2\bar{2}}+\omega^{3\bar{3}}+\omega^{4\bar{4}})
\end{equation}
and the complex structure equations follow~\eqref{caso-i}. 
\end{proposition}

\begin{proof}
By Theorem~\ref{general-structure-of-aK-dim8}, the existence of
an AK metric on $X$ allows us to consider equations of the form~\eqref{caso-i} in terms of some basis of invariant $(1,0)$-forms. 
In particular, we observe that the underlying Lie algebra $\frg$ with complex structure $J$ of 
the complex nilmanifold $X$
is an extension of a complex $3$-dimensional torus. 
Since the AK metric $F'$ is invariant, one can apply Lemma~\ref{reduc-met} to write the metric in canonical form in terms of another invariant $(1,0)$-basis. 
Since the process described in the proof of Lemma~\ref{reduc-met} does not change the form of the complex structure equations, the new equations still follow~\eqref{caso-i}. 
\end{proof}

The following theorem gives a characterization of the 4-dimensional complex nilmanifolds admitting invariant AK metrics. Recall the definition of ${\mathcal T}$ given in Definition~\ref{condition-no-AK-NO}.

\begin{theorem}\label{classif-aK-dim8-invariante-theorem}
An $8$-dimensional nilmanifold endowed with an invariant complex structure $X=(M,J)$ admits
an \textbf{invariant} astheno-K\"ahler metric \textbf{if and only if} there exists an invariant $(1,0)$-coframe
$\{\omega^k\}_{k=1}^4$ 
on $X$ satisfying~\eqref{caso-i} 
with ${\mathcal T} =0$.
\end{theorem}

\begin{proof}
The result comes from Proposition~\ref{classif-aK-dim8-invariante}
taking into account that, 
by Proposition~\ref{taus-of-aK-dim-m} for $m=3$,  the canonical metric \eqref{metrica-canonica} is AK if and only if ${\mathcal T} =0$.
\end{proof}

The following theorem provides the structure of the 4-dimensional complex nilmanifolds admitting invariant metrics that are both AK and (non-balanced) sG. 

\begin{theorem}\label{classif-aK-dim8-invariante-sG}
Let $M$ be an $8$-dimensional nilmanifold endowed with an invariant complex structure $J$. Then, $X=(M,J)$ admits an \textbf{invariant} metric that is both astheno-K\"ahler and (non-balanced) sG \textbf{if and only if} there exists an invariant $(1,0)$-coframe $\{\omega^k\}_{k=1}^4$ on $X$ satisfying~\eqref{caso-i} 
with ${\mathcal T} =0$ and $(A_{12},A_{13},A_{23})\neq (0,0,0)$.
In particular, the complex structure $J$ is nilpotent but not abelian.
\end{theorem}

\begin{proof}
Suppose first that $X$ admits an invariant AK metric $F'$. 
By Theorem~\ref{classif-aK-dim8-invariante-theorem} this is equivalent to the existence of an invariant $(1,0)$-coframe $\{\omega^k\}_{k=1}^4$ satisfying~\eqref{caso-i} with ${\mathcal T} =0$. More precisely, 
by Proposition~\ref{classif-aK-dim8-invariante} the metric $F'$ can be written in the canonical form~\eqref{metrica-canonica} in such a coframe. So, we need to study the sG condition for such metric, which is equivalent to the existence of a $(4,2)$-form $\gamma$ satisfying $\partial F'^{\,3}=\db \gamma$.

Now, a direct calculation using~\eqref{caso-i} gives 
$$
\begin{array}{rcl}
\partial F'^{\,3} &\!\!=\!\!& (\frac{i}{2})^3 \,\partial(\omega^{1\bar{1}}+\omega^{2\bar{2}}+\omega^{3\bar{3}}+\omega^{4\bar{4}})^3 \\[3pt]
&\!\!=\!\!&
 3!\, (\frac{i}{2})^3 (\omega^{1\bar{1}2\bar{2}}+\omega^{1\bar{1}3\bar{3}}+\omega^{2\bar{2}3\bar{3}})\wedge \partial\omega^{4\bar{4}}\\[3pt]
&\!\!=\!\!& \frac{- 3\, i}{4} (\bar{B}_{11}+\bar{B}_{22}+\bar{B}_{33})\, \omega^{1234\bar{1}\bar{2}\bar{3}}.
\end{array}
$$
Notice that $\partial F'^{\,3}=0$, i.e. $F'$ is balanced, if and only if 
$B_{11}+B_{22}+B_{33}=0$. Due to the condition $\mathcal T=0$, one then has that
$F'$ is both balanced and AK if and only if all the structure constants in \eqref{caso-i} vanish, namely, 
$X$ is a torus. Otherwise, $\partial F'^{\,3}\neq 0$
and from~\eqref{caso-i} one computes $\db \omega^{1234{\bar r}{\bar s}} =0$ for $1\leq r < s\leq 3$, and 
$$
\db \omega^{1234{\bar 1}{\bar 4}} = -\bar{A}_{23}\, \omega^{1234{\bar 1}{\bar 2}{\bar 3}},\quad
\db \omega^{1234{\bar 2}{\bar 4}} = \bar{A}_{13}\, \omega^{1234{\bar 1}{\bar 2}{\bar 3}},\quad
\db \omega^{1234{\bar 3}{\bar 4}} = -\bar{A}_{12}\, \omega^{1234{\bar 1}{\bar 2}{\bar 3}}.
$$
Therefore, 
$\db \big(\bigwedge ^{4,2} \frg^*_{\C} \big) = \langle \omega^{1234{\bar 1}{\bar 2}{\bar 3}} \rangle$
if and only if $(A_{12},A_{13},A_{23})\neq (0,0,0)$. 
Consequently, the AK metric $F'$ is also (non-balanced) sG if and only if $(A_{12},A_{13},A_{23})\neq (0,0,0)$.

The converse is straightforward. 
\end{proof}

In the following theorem we describe the structure of the 4-dimensional complex nilmanifolds admitting at the same time an invariant AK metric and a balanced metric.

\begin{theorem}\label{classif-aK-dim8-invariante-balanced}
Let $M$ be an $8$-dimensional nilmanifold endowed with an invariant complex structure $J$. 
If  $X=(M,J)$ admits an \textbf{invariant} astheno-K\"ahler metric $F'$ and a balanced metric~$F''$, then 
there exists an invariant $(1,0)$-coframe $\{\omega^k\}_{k=1}^4$ on $X$ satisfying~\eqref{caso-i} 
with ${\mathcal T} =0$ 
such that the canonical metric 
$
F_0=\frac{i}{2}\,(\omega^{1\bar{1}}+\omega^{2\bar{2}}+\omega^{3\bar{3}}+\omega^{4\bar{4}})
$
is AK and the symmetrization $\tilde{F''}$ of the balanced metric $F''$ writes as 
$
 \tilde{F''}= F + \frac{i}{2}\,\omega^{4\bar 4},
$
 where
 $$
F = i\,(x_{11}\,\omega^{1\bar 1}+x_{22}\,\omega^{2\bar 2}+x_{33}\,\omega^{3\bar 3})  
+x_{1 2}\,\omega^{1\bar 2}-\bar x_{1 2}\,\omega^{2\bar 1}
+x_{1 3}\,\omega^{1\bar 3}-\bar x_{1 3}\,\omega^{3\bar 1}
+x_{2 3}\,\omega^{2\bar 3}-\bar x_{2 3}\,\omega^{3\bar 2},
$$ 
with
\begin{equation}\label{balanced-condition}
\begin{array}{rcl}
 && B_{11}(x_{22}x_{33}-|x_{23}|^2) +
B_{22}(x_{11}x_{33}-|x_{13}|^2) +
B_{33}(x_{11}x_{22}-|x_{12}|^2) \\[3pt]
&&
+B_{12}(-i\,{\bar x}_{12}x_{33} + {\bar x}_{13}x_{23}) +
B_{21}(i\,x_{12}x_{33} + x_{13} {\bar x}_{23}) \\[3pt]
&&
+B_{13}(-i\,{\bar x}_{13}x_{22} - {\bar x}_{12}{\bar x}_{23}) +
B_{31}(i\,x_{13}x_{22} - x_{12}x_{23}) \\[3pt]
&&
+B_{23}(-i\,x_{11}{\bar x}_{23}+x_{12}{\bar x}_{13}) +
B_{32}(i\, x_{11}x_{23}+{\bar x}_{12}x_{13}) = 0.
\end{array} 
\end{equation}

The converse also holds in the following sense: if there exists an invariant $(1,0)$-coframe $\{\omega^k\}_{k=1}^4$ satisfying~\eqref{caso-i} with ${\mathcal T} =0$, then the canonical metric $F'$ is AK and the metric $F''= F + \frac{i}{2}\,\omega^{4\bar 4}$, with $F$ as above and satisfying \eqref{balanced-condition}, is balanced.
\end{theorem}

\begin{proof}
By Proposition~\ref{classif-aK-dim8-invariante} and Theorem~\ref{classif-aK-dim8-invariante-theorem}, the existence of an invariant AK metric $F'$ on $X$ implies the existence of an invariant $(1,0)$-coframe $\{\eta^k\}_{k=1}^4$ on $X$ satisfying 
$$d\eta^1 =  d\eta^2  =  d\eta^3  =  0,\quad	\quad 
d\eta^4 = \sum_{1\leq j < k\leq 3} A_{jk}\,\eta^{jk} + \sum_{1\leq j, k\leq 3} B_{jk}\,\eta^{j\bar{k}},$$
with ${\mathcal T} =0$ in which $F'$ is written in canonical form, namely, 
$F'=\frac{i}{2}\,(\eta^{1\bar{1}}+\eta^{2\bar{2}}+\eta^{3\bar{3}}+\eta^{4\bar{4}})$. 

Let $F''$ be a balanced metric on $X$.
After symmetrization, $\tilde{F''}$ is an invariant balanced metric, so it expresses in the $(1,0)$-coframe 
$\{\eta^k\}_{k=1}^4$ as
$$
\tilde{F''} = i \sum_{k=1}^4 x''_{kk}\,\eta^{k\bar k} +\sum_{1\leq k<l\leq 4}\big( x''_{k l}\,\eta^{k\bar l}-\bar x''_{k l}\,\eta^{l\bar k} \big),
$$
where all the $x''_{jk}$'s are constant. 
If we now define a new $(1,0)$-coframe $\{\omega^k\}_{k=1}^4$ following~\eqref{cambio-base}, 
the balanced metric becomes
$\tilde{F''}=F+ \frac{i}{2}\,\omega^{4\bar 4}$, with $F$ as given in the statement. Moreover, the complex structure equations still have the same form as above, now in terms of $\omega$'s and with complex coefficients in 
$d\omega^4$ given by $\sqrt{2\,x''_{44}} A_{jk}$ and $\sqrt{2\,x''_{44}} B_{jk}$ instead of $A_{jk}$ and $B_{jk}$, respectively.
Even though the metric $F'$ might no longer be canonical in the new 
coframe $\{\omega^k\}_{k=1}^4$, we note that the condition ${\mathcal T} =0$ is still satisfied. 
By Proposition~\ref{taus-of-aK-dim-m} for $m=3$, this is equivalent to the canonical metric 
$F_0$ in this new coframe being astheno-K\"ahler. 

It remains to check that a metric of the form $\tilde{F''}= F + \frac{i}{2}\,\omega^{4\bar 4}$ is balanced if and only if \eqref{balanced-condition} holds. 
By equations~\eqref{caso-i} and Proposition~\ref{reducc-cond-met-prop}~(i) for $n=4$, the balanced condition $d \tilde{F''}^{\,3}=0$ is equivalent to 
$F^{2}\wedge d\omega^{4}=0$. This is exactly \eqref{balanced-condition}.

\smallskip

Finally, we note that the converse is straightforward. 
\end{proof}

In the example below we construct many solutions satisfying the conditions in Theorem~\ref{classif-aK-dim8-invariante-balanced}, i.e. compact complex manifolds $X$ with $\dim_{\C}X=4$ admitting an AK metric (being also sG) and another balanced metric. The underlying complex structures can be of abelian or nilpotent (non-abelian) type.

\begin{example}\label{ex1}
We start with complex equations of the form \eqref{caso-i}, and we consider a diagonal metric 
$F''=i\,x_{11}\,\omega^{1\bar 1}+i\,x_{22}\,\omega^{2\bar 2}+i\,x_{33}\,\omega^{3\bar 3}+ \frac{i}{2}\,\omega^{4\bar 4}$, 
with $x_{11},x_{22},x_{33}\in \R^+$. By \eqref{balanced-condition}, the balanced condition for $F''$ becomes 
$x_{22} x_{33} B_{11} + x_{11} x_{33} B_{22} + x_{11} x_{22} B_{33}=0$, 
so we have
$$
B_{33}= - x\, B_{11} - y\, B_{22},
$$
where we are denoting $x=\frac{x_{33}}{x_{11}}$ and $y=\frac{x_{33}}{x_{22}}$ for simplicity. 
Now, 
the Hermitian metric $2F'=i (\omega^{1\bar 1}+\omega^{2\bar 2}+\omega^{3\bar 3}+\omega^{4\bar 4})$ is AK if and only if ${\mathcal T} =0$, which is equivalent to
$$
-2x\, |B_{11}|^2 - 2y\, |B_{22}|^2 + 2 (1-x-y) \Real\!(B_{11} \overline{B_{22}}) 
= \sum_{1\leq j < k\leq 3} |A_{jk}|^2 + \sum_{1\leq j\ne k\leq 3} |B_{jk}|^2. 
$$
Hence, one can construct many $4$-dimensional complex nilmanifolds on which AK and balanced metrics coexist (see Remark~\ref{sec6-exist} for concrete examples). 
Note that by Theorem~\ref{classif-aK-dim8-invariante-sG} 
the AK metric $F'$ is also sG when $\sum_{1\leq j < k\leq 3} |A_{jk}|^2 \ne 0$.
\end{example}

\subsection{AK nilmanifolds with balanced metrics in every dimension}\label{subsec-dim-general-AK-y-sG} 

In this section we extend the previous results to the construction, in every complex dimension $n\geq 4$, of complex nilmanifolds admitting both an AK metric (possibly sG) and another metric that is balanced.  

\begin{proposition}\label{existenciaAK-balanced-n-dim}
For each $m\geq 3$, let $X$ be a complex nilmanifold defined by 
\begin{equation}\label{ecus-n-dim}
\left\{\begin{array}{cl}
\!\!&\!\!\!\! d\omega^1 = \cdots=d\omega^{m}=0,\\[4pt]
\!\!&\!\!\!\! d\omega^{m+1} = \sum_{1\leq j<k\leq m} A_{jk}\, \omega^{jk} 
+ \sum_{k=1}^{m-1} \omega^{k\bar{k}} + b\, \omega^{m\,\overline{m}},
\end{array}\right.
\end{equation}
where $b\in [1-\frac{m}{2},0)$ and the coefficients $A_{jk}\in\C$ satisfy the condition
\begin{equation}\label{condicion-n-dim}
\sum_{1\leq j<k\leq m} |A_{jk}|^2 = (m-1)(2 b+m-2).
\end{equation}
Then, $X$ is a compact complex manifold with $\dim_{\C} X=m+1$ 
admitting an AK metric $F'$ (which is in addition sG iff $b\neq 1-\frac{m}{2}$) and a balanced metric $F''$. 

More concretely, the following metrics satisfy the aforementioned conditions:
\begin{equation}\label{F'-n-dim}
F'= \frac{i}{2}\, (\omega^{1\bar{1}} +\cdots +\omega^{m\,\overline{m}} + \omega^{m+1\,\overline{m+1}}),
\end{equation}
and
\begin{equation}\label{F''-n-dim}
F''= i\, x_{11}\, \omega^{1\bar{1}} + \cdots + i\, x_{mm}\, \omega^{m\,\overline{m}} + \frac{i}{2}\,\omega^{m+1\,\overline{m+1}}, 
\end{equation}
with $x_{11}, \ldots, x_{mm} >0$ such that $\frac{x_{mm}}{x_{11}}+ \cdots+\frac{x_{mm}}{x_{m-1\,m-1}} =-b$.

Moreover, the first Betti number $b_1(X)$ of $X$ satisfies $2m\leq b_1(X) \leq 2m+1$, and the equality $b_1(X) = 2m+1$ holds if and only if the complex structure is abelian. 
\end{proposition}

\begin{proof}
Note that, due to~\eqref{ecus-n-dim}, the Lie algebra $\frg$ with complex structure $J$ associated 
to the nilmanifold $X$ is a central extension of the $m$-dimensional complex abelian Lie algebra.
First, we study the AK condition for the Hermitian metric $F'$ given by \eqref{F'-n-dim}, i.e. $\partial\db F'^{\,m-1}=0$. From Corollary~\ref{reducc-cond-met-cor2}~(ii) for $n=m+1$ 
and $\tilde F=\frac{i}{2}\, (\omega^{1\bar{1}} +\cdots +\omega^{m\,\overline{m}})$ 
together with the equations~\eqref{ecus-n-dim}, 
we observe that the AK condition  
is equivalent to the expression 
$$
\tilde{F}^{m-2}\wedge\big(\bar\partial\omega^{m+1}\wedge\partial\omega^{\overline{m+1}} -
   \partial\omega^{m+1}\wedge\bar\partial\omega^{\overline{m+1}}\big) = 0.
$$
On the one hand, we have
\begin{equation*}
\begin{split}
\bar\partial\omega^{m+1}\wedge\partial\omega^{\overline{m+1}} 
&= -\left(\sum_{k=1}^{m-1} \omega^{k\bar{k}} + b\, \omega^{m\overline{m}}\right)
	\wedge
	\left(\sum_{r=1}^{m-1} \omega^{r\bar{r}} + b\, \omega^{m\overline{m}}\right)\\
&=
-2\!\!\sum_{1\leq k<r\leq m-1} \omega^{k\bar{k}r\bar{r}} - 2b\, \sum_{k=1}^{m-1} \omega^{k\bar{k}m\overline{m}}
\end{split}
\end{equation*}
and
$$
\partial\omega^{m+1}\wedge\bar\partial\omega^{\overline{m+1}}= 
	\left(\sum_{1\leq j<k\leq m} A_{jk}\, \omega^{jk}\right)
	\wedge
	\left(\sum_{1\leq r<s\leq m} \bar{A}_{rs}\, \omega^{\bar{r}\bar{s}}\right)
=
\sum_{
\begin{smallmatrix}
1\leq j<k\leq m\\
1\leq r<s\leq m
\end{smallmatrix}
} A_{jk}\,\bar{A}_{rs} \omega^{jk\bar{r}\bar{s}}.
$$
On the other hand,
$$
\tilde F^{m-2} = \left(\frac{i}{2}\right)^{m-2}\!\! 
\big(\omega^{1\bar{1}} +\cdots +\omega^{m\overline{m}}\big)^{m-2} 
= (m-2)!\, \left(\frac{i}{2}\right)^{m-2} \!\!\! 
\sum_{1\leq u<v\leq m} \omega^{1\bar{1}\cdots \widehat{u\bar{u}}\cdots \widehat{v\bar{v}}\cdots m\overline{m}}.
$$
Therefore,
$$
\begin{array}{ccl}
\tilde F^{m-2}\wedge\bar\partial\omega^{m+1}\wedge\partial\omega^{\overline{m+1}} 
   \!\!&\!\!=\!\!&\!\!  (m-2)!\, (\frac{i}{2})^{m-2}\, \big(-2\, \frac{(m-1)(m-2)}{2} -2b(m-1) \big)\,  \omega^{1\bar{1}\cdots m\overline{m}},
\end{array}
$$
and
$$
\begin{array}{ccl}
\tilde F^{m-2}\wedge
   \partial\omega^{m+1}\wedge\bar\partial\omega^{\overline{m+1}}
\!\!&\!\!=\!\!&\!\! (m-2)!\, (\frac{i}{2})^{m-2}\, 
\big(- \sum_{1\leq j<k\leq m} |A_{jk}|^2 \big)\,  \omega^{1\bar{1}\cdots m\overline{m}},
\end{array}
$$
so we get
$$
\begin{array}{ccl}
0 \!\!&\!\!=\!\!&\!\!  \tilde F^{m-2}\wedge\big(\bar\partial\omega^{m+1}\wedge\partial\omega^{\overline{m+1}} -
   \partial\omega^{m+1}\wedge\bar\partial\omega^{\overline{m+1}}\big) \\[4pt] 
   \!\!&\!\!=\!\!&\!\! (m-2)!\, (\frac{i}{2})^{m-2}\, 
\big( \sum_{1\leq j<k\leq m} |A_{jk}|^2 -(m-1)(2b+m-2) \big)\,  \omega^{1\bar{1}\cdots m\overline{m}},
\end{array}
$$
which is precisely \eqref{condicion-n-dim}.
Notice that this condition is equivalent to 
$(b+m-1)^2= \sum_{1\leq j<k\leq m} |A_{jk}|^2 + b^2+m-1$.

In order to see when this AK metric is also sG, we next compute $\partial F'^{\,m}$. 
Since $F'= \tilde F + \frac{i}{2}\, \omega^{m+1\,\overline{m+1}}$, 
and taking into account that \eqref{ecus-n-dim} implies 
$\partial \tilde F^{m-1}=0=\partial \tilde F^{m}$, we get
$$
\partial F'^{\,m}=\frac{i\,m}{2}\,
\Big(\tilde F^{m-1} \wedge 
\partial\omega^{m+1}\wedge
\omega^{\overline{m+1}}
-\tilde F^{m-1}\wedge 
\partial\omega^{\overline{m+1}}\wedge
\omega^{m+1}\Big).
$$ 
It is easy to check that $\tilde F^{m-1}\wedge \partial\omega^{m+1}=0$, hence 
$$
\begin{array}{ccl}
\partial F'^{\,m} \!\!&\!\!=\!\!&\!\! -\frac{i\,m}{2}\,
 \tilde F^{m-1}\wedge \partial\omega^{\overline{m+1}}\wedge\omega^{m+1}
 =   \frac{i\,m}{2}\, 
 \tilde F^{m-1}\wedge \big(\sum_{k=1}^{m-1} \omega^{k\bar{k}} + b\, \omega^{m\overline{m}} \big)\wedge \omega^{m+1} \\[4pt]
 \!\!&\!\!=\!\!&\!\!  
 m!\, (\frac{i}{2})^{m} 
\big(\sum_{u=1}^{m} \omega^{1\bar{1}\cdots \widehat{u\bar{u}}\cdots m\overline{m}}\big) \wedge \big(\sum_{k=1}^{m-1} \omega^{k\bar{k}} + b\, \omega^{m\overline{m}} \big)\wedge \omega^{m+1} \\[4pt]
 \!\!&\!\!=\!\!&\!\!  
 m!\, (\frac{i}{2})^{m}\, (m-1+b )\, \omega^{1\bar{1}\cdots m\overline{m}}\wedge \omega^{m+1}.
 \end{array}
$$
Notice that $b\in [1-\frac{m}{2},0)$ implies $b+m-1\neq0$, so the AK metric $F'$ cannot be balanced. 
Furthermore, for any $1\leq j<k\leq m$, the $(m+1,m-1)$-forms 
$\gamma_{jk}=\omega^{1\cdots m+1}\wedge \omega^{\bar{1}\cdots \widehat{\bar{j}}\cdots \widehat{\bar{k}}\cdots\overline{m}\,\overline{m+1}}$ satisfy
$$
\db \gamma_{jk}= (-1)^{j+k}\bar{A}_{jk}\, \omega^{1\cdots m+1}\wedge \omega^{\bar{1}\cdots\overline{m}}, 
$$
therefore, $F'$ is sG if and only if at least one of the coefficients $A_{jk}$ is nonzero. By \eqref{condicion-n-dim}, this is equivalent to $b\neq 1-\frac{m}{2}$. 

\smallskip

Finally, we study the balanced condition for a Hermitian metric $F''$ of the form \eqref{F''-n-dim}. Applying 
Corollary~\ref{reducc-cond-met-cor2}~(i) for $n=m+1$ and using the equations~\eqref{ecus-n-dim}, the balanced condition $d F''^{\,m}=0$ is equivalent to 
$$
F^{m-1}\wedge d\omega^{m+1}=0,
$$ 
where $F= i\, x_{11}\, \omega^{1\bar{1}} + \cdots + i\, x_{mm}\, \omega^{m\,\overline{m}}$. Then, we have 
$$
\begin{array}{ccl}
0 \!\!&\!\!=\!\!&\!\! 
F^{m-1}\wedge d\omega^{m+1}\\[1pt] 
\!\!&\!\!=\!\!&\!\! (m\!-\!1)!\, i^{m-1}\! 
\left(\sum\limits_{u=1}^{m} x_{11}\cdots \widehat{x_{uu}}\cdots x_{mm}\,  \omega^{1\bar{1}\cdots \widehat{u\bar{u}}\cdots m\overline{m}}\right) 
\!\wedge\! \left(\sum\limits_{1\leq j<k\leq m} A_{jk}\, \omega^{jk} 
+ \sum\limits_{k=1}^{m-1} \omega^{k\bar{k}} + b\, \omega^{m\,\overline{m}}\right) \\[6pt]
\!\!&\!\!=\!\!&\!\!  (m\!-\!1)!\, i^{m-1}\! 
\left(  \Big(\sum\limits_{u=1}^{m-1} x_{11}\cdots \widehat{x_{uu}}\cdots x_{m-1\,m-1}\Big) x_{mm} 
	+ b\, x_{11}\cdots x_{m-1\,m-1}  \right)\,  \omega^{1\bar{1} \cdots  m\overline{m}} .
\end{array}
$$
Therefore, $F''$ is balanced if and only if 
$$
\Big(\sum_{u=1}^{m-1} x_{11}\cdots \widehat{x_{uu}}\cdots x_{m-1\,m-1}\Big) x_{mm}= - b\, x_{11}\cdots x_{m-1\,m-1},
$$
which means that we have to choose the metric coefficients for $F''$ satisfying the condition 
$$
\frac{x_{mm}}{x_{11}}+ \cdots+\frac{x_{mm}}{x_{m-1\,m-1}} =-b. 
$$ 
\end{proof}

\begin{remark}\label{existence-over-Q}
There exist complex nilmanifolds defined by \eqref{ecus-n-dim} and satisfying \eqref{condicion-n-dim}. Indeed, if we take all the coefficients $A_{jk}\in\Q[i]$ so that the rational number  $q=\frac{1}{m-1}\sum_{1\leq j<k\leq m} |A_{jk}|^2$ belongs to the interval $[0,m-2)$, then the number $b$ given by $2 b+m-2=q$ is also rational.
\end{remark}

\section{On the Fr\"olicher spectral sequence of compact AK manifolds with balanced metrics}\label{FSS}

\noindent
In this section we construct compact AK nilmanifolds with balanced metrics and with Fr\"olicher spectral sequence not degenerating at the second, or even at the third, page. As far as we know, these seem to be the first compact complex manifolds with such properties.

Let $X$ be a compact complex manifold with $\dim_{\C}X=n$. The Fr\"olicher spectral sequence (FSS) of $X$ is 
the spectral sequence associated to  
the double complex $(\Omega^{*,*}(X), \partial,\bar\partial)$ \cite{Fro}. 
It consists of a collection of complexes 
$$
\cdots\stackrel{d_r}{\longrightarrow} E_r^{p-r,\,q+r-1}(X)\stackrel{d_r}{\longrightarrow}E_r^{p,\,q}(X)\stackrel{d_r}{\longrightarrow}E_r^{p+r,\,q-r+1}(X)\stackrel{d_r}{\longrightarrow}\cdots
$$ 
for every  $r\geq 1$, where 
the {\it $1$-st page} is given by the Dolbeault cohomology of $X$, i.e. $E_1^{p,\,q}(X)=H^{p,\,q}_{\bar\partial}(X)$, and the differentials $d_1$ are of type $(1,0)$ and defined by $d_1([\alpha]) = [\partial\alpha]$, for every Dolbeault class $[\alpha]\in H^{p,\,q}_{\bar\partial}(X)$. For any other $r\geq 2$, the differentials $d_r$ are of type $(r,\,-r+1)$ and are also induced by $\partial$, but acting on a $(p+r-1,q-r+1)$-form associated to every class in $E_r^{p,\,q}(X)$. (See the description below for $p=0$ and $q=n-2$.) One has that $d_r\circ d_r=0$ and the {\it $(r+1)$-th page} is given by the quotient space of the kernel of $d_r$ over the image of the incoming $d_r$. In \cite{PSU} a Hodge theory is developed 
by means of the construction of elliptic pseudo-differential operators, 
associated with any given
Hermitian metric on $X$, so that their kernels are isomorphic to the spaces $E_r^{p,q}(X)$ in every bidegree $(p, q)$.

The FSS {\it degenerates at the $r$-th page} if the differentials $d_s=0$, for all $s\geq r$. 
Equivalently, $E_r^{p,\,q}(X) = E_{r+l}^{p,\,q}(X)$ for every $l\geq 1$ and any $(p,q)$, that is, $E_r(X)=E_\infty(X)$.  Note that such an~$r$ always exists for any
compact complex manifold $X$. Hence, the FSS provides a link between the complex structure of $X$ and its differential structure,  
due to the existence of the isomorphisms $H^k_{dR}(X,\,\C)\cong\bigoplus\limits_{p+q=k}E_\infty^{p,\,q}(X)$, for the de Rham cohomology of $X$.

\smallskip

Next we focus on nilmanifolds endowed with an invariant complex structure. 
The first examples with FSS satisfying $E_2\ne E_{\infty}$ were constructed by Cordero, Fern\'andez and Gray in \cite{CFG-CR}. Moreover, an example in complex dimension $6$ with $E_3(X)\ne E_\infty(X)$ was found  in \cite{CFG-illinois}. 
In the following result we review these examples  
and study the existence of balanced or AK
metrics on them.

\begin{proposition}\label{Prop-CFG}
The complex nilmanifolds constructed by Cordero, Fern\'andez and Gray in \cite{CFG-CR} and \cite{CFG-illinois} are balanced, but they do not admit any AK metric. 
\end{proposition}

\begin{proof}
Let us first consider the example in \cite[Section 3]{CFG-CR}, which is given by the 4-dimensional complex nilmanifold $X$ defined by the complex structure equations  
$$
d\omega^1=d\omega^2=0, \quad
d\omega^3=-\omega^{12}, \quad
d\omega^4=\omega^{23} + \omega^{2\bar{1}}.
$$
It is proved in \cite{CFG-CR} that $E_2(X)\ne E_\infty(X)$. 
Given  any $\rho=(\rho_1,\rho_2,\rho_3,\rho_{4})\in (\R^+)^{4}$, it is easy to check that the Hermitian metrics on $X$ defined by 
$F_{\rho}=\frac{i}{2} \sum_{k=1}^4 \rho_k \,\omega^{k\bar{k}}$ are balanced. 
The non-existence of AK metrics follows from Theorem~\ref{general-structure-of-aK-dim8}.

Now, we consider the example in \cite[Section 5]{CFG-illinois}, which consists of a complex nilmanifold $X$ with $\dim_{\C} X=4$ defined by 
$$
d\omega^1=d\omega^2=0, \quad
d\omega^3=\omega^{12}+\omega^{1\bar{2}}, \quad
d\omega^4=-\omega^{2\bar{1}}.
$$
In \cite{CFG-illinois} it is proved that $E_2(X)\ne E_\infty(X)$. As in the previous example, 
the Hermitian metrics  
$F_{\rho}=\frac{i}{2} \sum_{k=1}^4 \rho_k \,\omega^{k\bar{k}}$ are balanced on $X$. 
Again, by Theorem~\ref{general-structure-of-aK-dim8}, this 4-dimensional complex nilmanifold does not admit any AK metric. 
 
Finally, let us consider the example in \cite[Section 6]{CFG-illinois}. In this case $X$ is a complex nilmanifold with $\dim_{\C} X=6$ defined by the complex structure equations 
$$
d\omega^1=d\omega^2=d\omega^3=0, \quad
d\omega^4=\omega^{12}+\omega^{1\bar{2}}, \quad
d\omega^5=-\omega^{2\bar{1}}, \quad
d\omega^6=\omega^{14}+\omega^{1\bar{3}}.
$$
According to \cite{CFG-illinois}, the FSS of $X$ does not degenerate at the third page, i.e. $E_3(X)\ne E_\infty(X)$. 
The Hermitian metric 
$F_{\rho}=\frac{i}{2} \sum_{k=1}^6 \rho_k \,\omega^{k\bar{k}}$ is balanced on $X$, for every $\rho=(\rho_1,\ldots,\rho_{6})\in (\R^+)^{6}$. 
Notice that $X$ is an extension of the previous 4-dimensional example, so we can apply Proposition~\ref{extensions-of-non-AK} to conclude that no AK metric exists on $X$.
\end{proof} 

In~\cite{CFGU97},  3-dimensional complex nilmanifolds with FSS satisfying $E_2(X) \ne E_{\infty}(X)$ are found. They do not admit balanced or AK (SKT) metrics. Indeed, 
the FSS of any complex nilmanifold~$X$ with $\dim_{\C} X = 3$ is studied in \cite[Theorem 4.1]{COUV} and, as a consequence, one concludes that the existence of a balanced or an AK (SKT) metric on $X$ implies that $E_2(X) = E_{\infty}(X)$. 

By \cite[Theorem 3.4]{LUV-Fro}, 
there exist infinitely many 4-dimensional balanced  
nilmanifolds with FSS not degenerating at the second page and with different complex (hence, real or rational) homotopy types. Moreover, this result is extended in \cite[Theorem 3.8]{LUV-Fro} to non-degeneration at any arbitrary page. We now observe the following:

\begin{proposition}\label{Prop-LUV}
All the balanced nilmanifolds $X$ with $E_2(X) \ne E_{\infty}(X)$ constructed in \cite{LUV-Fro} do not admit AK metrics. 
\end{proposition}

\begin{proof} 
In the case $\dim_{\C} X=4$, the complex structure on $X$ is of SnN type, so  
the result comes directly from Theorem~\ref{general-structure-of-aK-dim8}.

In the case $\dim_{\C} X\geq 5$, the complex nilmanifolds $X$ are constructed as products of 4-dimensional complex nilmanifolds endowed with an SnN structure and another appropriate complex nilmanifold. Hence, 
we can apply Proposition~\ref{extensions-of-non-AK} to conclude that no AK metric exists on them.
\end{proof}

For every $n\geq 2$, Bigalke and Rollenske constructed in \cite{BR} a (2-step) complex nilmanifold $X^{4n-2}$ with $\dim_{\C} X=4n-2$ and FSS satisfying $d_n\ne 0$. 
The nilmanifolds $X^{4n-2}$ are balanced by \cite[Theorem 3.3]{ST}, but they do not admit AK metrics as 
recently proved  in \cite[Theorem 3.1]{CT-arxiv}. In the following result we give a 
different proof of these facts following the lines of \cite{LUV-Fro}.

\begin{proposition}\label{Prop-FSS-Bigalke-Rollenske} {\rm (\cite[Theorem 3.3]{ST} and \cite[Theorem 3.1]{CT-arxiv})} 
The complex nilmanifolds $X^{4n-2}$ constructed by Bigalke and Rollenske in \cite{BR} are balanced, but they do not admit any AK metric. 
\end{proposition}

\begin{proof} 
Let us consider 
the complex structure equations  
in \cite{LUV-Fro} that define the Bigalke-Rollenske nilmanifolds. 
These are given by 
a $(1,0)$-coframe $\{\tau^j\}_{j=1}^{4n-2}$ satisfying 
\begin{equation}\label{ecusBR-2}
\left\{\begin{array}{l}
d\tau^1=\cdots=d\tau^{3n-2}=0, \\[2pt]
d\tau^{3n-1}=\tau^1\wedge \tau^{n+1} \quad\ \,  + \tau^{n}\wedge \overline{\tau^{2n}}, \\[-2pt]
\quad \vdots \qquad\qquad\ \  \vdots \qquad\qquad\qquad\ \ \vdots  \\[-1pt]
d\tau^{4n-3}=\tau^{n-1}\wedge \tau^{2n-1} + \tau^{2n-2}\wedge \overline{\tau^{3n-2}}, \\[2pt]
d\tau^{4n-2}= \qquad\qquad\qquad\quad   \tau^{2n}\wedge \overline{\tau^{n}}.
\end{array}\right.
\end{equation}
We observe that 
the real 2-form $\alpha=i\,\tau^{4n-2}\wedge \overline{\tau^{4n-2}}$  
satisfies 
$$
(dd^c\alpha)^{2,2}
=-2\partial\db(\tau^{4n-2}\wedge \overline{\tau^{4n-2}})=
-(\tau^{n}\wedge \tau^{2n})\!\wedge (\overline{\tau^{n}}\wedge \overline{\tau^{2n}}),$$
so the AK obstruction \eqref{2-condition} holds 
and $X^{4n-2}$ does not admit any AK metric. 
On the other hand, for  every $\rho=(\rho_1,\ldots,\rho_{4n-2})\in (\R^+)^{4n-2}$, the Hermitian metric $F_{\rho}=\frac{i}{2} \sum_{k=1}^{4n-2} \rho_k \,\tau^{k} \wedge\overline{\tau^{k}}$ is balanced.
\end{proof}

Note that the
Bigalke-Rollenske nilmanifolds $X^{4n-2}$ do not admit any SKT metric \cite[Proposition~3.5]{ST}. 
In \cite[Appendix A]{KS}, Kasuya and Stelzig found an extension of the equations \eqref{ecusBR-2} so that the resulting $(6n-4)$-dimensional complex nilmanifold ${\tilde X}^{6n-4}$ $(n\geq 2)$ is SKT and still satisfies $d_n\ne 0$. 
In the following result we prove that these nilmanifolds do not admit any balanced, sG or AK metrics.

\begin{proposition}\label{Prop-FSS-Kasuya-Stelzig}
The complex nilmanifolds ${\tilde X}^{6n-4}$ constructed by Kasuya and Stelzig in \cite{KS} do not admit sG or AK metrics. 
\end{proposition}

\begin{proof}
The complex structure equations considered in \cite[Appendix A]{KS} to define ${\tilde X}^{6n-4}$ are an extension of 
\eqref{ecusBR-2}, by adding the $(1,0)$-forms $\theta^1,\ldots,\theta^{n-1}$ and $\eta^1,\ldots,\eta^{n-1}$ so that 
\begin{equation}\label{ecusBR-2-extendidas}
\left\{\begin{array}{l}
d\theta^j = \tau^{j+n-1}\wedge \overline{\tau^{j+n-1}} + \tau^{j+2n-1}\wedge \overline{\tau^{j+2n-1}}, \\[2pt]
d\eta^j = \tau^{j}\wedge \overline{\tau^{j}} + \tau^{j+n}\wedge \overline{\tau^{j+n}},
\end{array}\right.
\end{equation}
for $1\leq j\leq n-1$. 
Let $\{z_{k}=x_{k}+i\,y_{k} \}_{k=1}^{6n-4}$ be the dual basis of 
$\{ \tau^1,\ldots,\tau^{4n-2},\theta^1, \ldots, \theta^{n-1},$ $\eta^1,\ldots,\eta^{n-1} \}$. 
It follows from the equations \eqref{ecusBR-2}--\eqref{ecusBR-2-extendidas} that the pair $(\tilde{\frg}, \tilde{J})$ underlying ${\tilde X}^{6n-4}$ is a (central) $\frb$-extension of the pair $({\frg}, {J})$ underlying $X^{4n-2}$, where $\frb=\langle x_{k},y_{k}\mid 4n-1\leq k\leq 6n-4 \rangle$. 
Hence, Proposition~\ref{extensions-of-non-AK} implies that the nilmanifold ${\tilde X}^{6n-4}$ does not admit any AK metric.

We now consider the $\db$-closed $(0,1)$-form $\alpha=\overline{\theta^j}$, for some $1\leq j\leq n-1$. Then, 
$$
\partial\overline{\theta^j}= - \tau^{j+n-1}\wedge \overline{\tau^{j+n-1}} - \tau^{j+2n-1}\wedge \overline{\tau^{j+2n-1}},
$$
so the sG obstruction \eqref{sG-condition} is satisfied and ${\tilde X}^{6n-4}$ does not admit any sG metric. 
\end{proof}

By Propositions~\ref{Prop-CFG}, \ref{Prop-LUV}, \ref{Prop-FSS-Bigalke-Rollenske} and~\ref{Prop-FSS-Kasuya-Stelzig}, 
all known examples of complex nilmanifolds with $E_2\neq E_{\infty}$ 
do not admit AK metrics. Note that in Proposition~\ref{almost-ab-prop} we proved that certain almost abelian complex nilmanifolds  are AK, but their FSS degenerates at the first page by~\cite[Theorem 1.1 (iii)]{AABRW}.

In the following results we construct AK nilmanifolds of complex dimension 4 (resp. 5) admitting balanced metrics and with FSS not degenerating at the second (resp. third) page. For the proofs of these results we will use an special case from the description of the terms in the FSS given in~\cite{CFGU97} that we describe below.

Let $X$ be a compact complex manifold with $\dim_{\C} X =n$. 
Take $p=0$ and $q=n-2$. For every $r\geq 1$, the space $E^{0,n-2}_r(X)$ is isomorphic to the quotient $\C$-vector space
\begin{equation}\label{E=X/Y}
E^{0,n-2}_r(X)=\frac{{\mathcal X}^{0,n-2}_r(X)}{{\mathcal Y}^{0,n-2}_r(X)},
\end{equation}
where
${\mathcal Y}^{0,n-2}_r(X)=\db\big(\Omega^{0,n-3}(X)\big)$, 
and
\begin{equation}\label{Xpq}
\begin{array}{rl}
{\mathcal X}^{0,n-2}_r(X)= \{\alpha_{0,n-2} \in \Omega^{0,n-2}(X)  \mid\,  \!\!&\!\! \db\alpha_{0,n-2}=0, \mbox{ and there exist $r-1$ forms } \\[4pt]
& \! \alpha_{1,n-3},\ldots,\,\alpha_{r-2,n-r},\, \alpha_{r-1,n-r-1} \mbox{ satisfying }\\[5pt] 
  0=\partial\!\!\!&\!\!\!\alpha_{0,n-2}+\db\alpha_{1,n-3}
=\cdots
=\partial\alpha_{r-2,n-r}+\db\alpha_{r-1,n-r-1} \}.
\end{array}
\end{equation} 

Moreover, the differential $d_r\colon E_r^{0,n-2}(X)\longrightarrow E_r^{r,n-r-1}(X)$ is explicitly given by 
\begin{equation}\label{dr}
d_r\big( [\alpha_{0,n-2}]\big)= [\partial\alpha_{r-1,n-r-1}],
\end{equation}
for any $[\alpha_{0,n-2}]\in E_r^{0,n-2}(X)$. 

\begin{lemma}\label{lema-0-n-2}
Let $X$ be a compact complex manifold with $\dim_{\C} X =n$. Suppose that ${\mathcal X}^{0,n-2}_r(X)\ne {\mathcal X}^{0,n-2}_{r+1}(X)$. Then, $d_r\ne 0$ and the FSS of $X$ does not degenerate at the $r$-th page. 
\end{lemma}

\begin{proof}
It follows directly from \eqref{E=X/Y} and \eqref{Xpq}, taking into account that ${\mathcal Y}^{0,n-2}_r(X)$ does not depend on $r$, and from \eqref{dr} using the fact that the incoming differential $d_r\colon E_r^{-r,n+r-3}(X)\longrightarrow E_r^{0,n-2}(X)$ is identically zero because $E_r^{-r,n+r-3}(X)=\{0\}$ by bidegree reasons.
\end{proof}

Now, we are in the conditions to prove the following 

\begin{proposition}\label{Prop-FSS-dim4}
Let $X$ be a $4$-dimensional complex nilmanifold defined by
\begin{equation}\label{FSS-ecus-4dim}
d\omega^1 = d\omega^{2}=d\omega^{3}=0,\quad 
d\omega^4 = A_{12}\, \omega^{12} 
	+ B_{11}\,\omega^{1\bar{1}} + B_{22}\,\omega^{2\bar{2}} - (x\,B_{11}+y\,B_{22}) \omega^{3\bar{3}} ,
\end{equation}
where $x,y\in \R^+$,  $(B_{11},B_{22})\in\R^2\setminus\{(0,0)\}$ and 
$A_{12}\in\C\setminus\{0\}$ satisfy
\begin{equation}\label{cond-FSS-ecus-4dim}
x\,B_{11}+y\,B_{22}\ne 0, \quad
-2x\, B_{11}^2 - 2y\, B_{22}^2 + 2 (1-x-y) B_{11} B_{22} = |A_{12}|^2.
\end{equation}
Then, $X$ has an AK metric $F'$ (which is also sG) and a balanced metric $F''$, and its FSS does not degenerate at the second page ($d_2\ne 0$). 
\end{proposition}

\begin{proof}
We first recall that, by~\cite[Proposition 2.1]{LUV-Fro}, the non-degeneration of the FSS of the complex nilmanifold $X$ 
can be detected on its underlying pair $(\frg,J)$. 
Applying Lemma~\ref{lema-0-n-2} with $n=4$ and $r=n-2=2$, it suffices to 
prove that ${\mathcal X}_2^{0,2}(\frg,J)\ne {\mathcal X}_3^{0,2}(\frg,J)$. 

From \eqref{FSS-ecus-4dim} we have
$\db \omega^{{\bar 1}{\bar 4}}=0$, and 
$$
\partial \omega^{{\bar 1}{\bar 4}}= - B_{22}\,\omega^{2{\bar 1}{\bar 2}} + (x\,B_{11}+y\,B_{22}) \, \omega^{3{\bar 1}{\bar 3}} = \db\omega^{4{\bar 1}},
$$
so $\omega^{{\bar 1}{\bar 4}} \in {\mathcal X}_2^{0,2}(\frg,J)$.  
By \eqref{Xpq}, to see that $\omega^{{\bar 1}{\bar 4}} \not\in {\mathcal X}_3^{0,2}(\frg,J)$, it is enough to prove that 
$\partial\omega^{4{\bar 1}}\not\in \db \Lambda^{2,0}(\frg^*)$.
Notice that $\partial\omega^{4{\bar 1}}= A_{12}\,\omega^{12{\bar 1}}$,
while
$$
\begin{array}{rl}
\db \Lambda^{2,0}(\frg^*) \!\!\!&\!\!\!= \langle \db\omega^{14},\db\omega^{24},\db\omega^{34} \rangle \\[4pt]
\!\!\!&\!\!\!= \langle B_{22}\,\omega^{12{\bar 2}} \!-\! (x\,B_{11}\!+y\,B_{22}) \omega^{13{\bar 3}},\ B_{11}\,\omega^{12{\bar 1}}\!+\!(x\,B_{11}\!+y\,B_{22}) \omega^{23{\bar 3}},\ B_{11}\,\omega^{13{\bar 1}}\!+\!B_{22}\,\omega^{23{\bar 2}} \rangle,
\end{array}
$$
so $\partial\omega^{4{\bar 1}}\not\in \db \Lambda^{2,0}(\frg^*)$ because of the first condition in \eqref{cond-FSS-ecus-4dim}. 

Finally, the existence of an AK metric (which is also sG) and a balanced metric on $X$ follows from \eqref{cond-FSS-ecus-4dim} and Example~\ref{ex1}.
\end{proof}

\begin{remark}\label{sec6-exist}
There exist complex nilmanifolds  
satisfying the conditions of Proposition~\ref{Prop-FSS-dim4}. 
To check this assumption,
it suffices to see that there are rational solutions 
to the equation in~\eqref{cond-FSS-ecus-4dim}. For instance, if we take $x=y=\frac18$ and $A_{12}\in \Q[i]\setminus\{0\}$, the equation reduces to find  $B_{11},B_{22}\in \Q$ satisfying $B_{11}^2+B_{22}^2-6\,B_{11}B_{22}+4|A_{12}|^2=0$. This equation is a hyperbola in the $(B_{11},B_{22})$-plane for any given $A_{12}\ne 0$. A particular rational solution is $A_{12}=B_{11}=B_{22}=1$.
\end{remark}

 Next we construct a family of 5-dimensional complex nilmanifolds with 
$E_3\ne E_\infty$.

\begin{proposition}\label{FSS-5dim}
Let $X$ be a 5-dimensional complex nilmanifold defined by 
\begin{equation}\label{ecus-5dim}
\begin{cases}
d\omega^1 = d\omega^{2}=d\omega^{3}=d\omega^{4}=0,\\[4pt]
d\omega^5 = A_{12}\, \omega^{12} + A_{34}\, \omega^{34} 
	+ \omega^{1\bar{1}} + \omega^{2\bar{2}} + \omega^{3\bar{3}} 
	+ b\, \omega^{4\bar{4}},
\end{cases}
\end{equation}
with $b\in\R$ and $A_{12},A_{34}\in \C\backslash\{0\}$ satisfying $b\, |A_{12}|^2-|A_{34}|^2\neq0$. Then, the FSS of $X$ does not degenerate at the third page ($d_3\ne 0$). 
\end{proposition}

\begin{proof}
Using again~\cite[Proposition 2.1]{LUV-Fro},  
the non-degeneration of the FSS of $X$ can be  detected on its underlying pair $(\frg,J)$. 
We now apply Lemma~\ref{lema-0-n-2} with $n=5$ and $r=n-2=3$; more concretely, we next see that  $\omega^{\bar1\bar3\bar5}\in {\mathcal X}_3^{0,3}(\frg,J)$, but $\omega^{\bar1\bar3\bar5}\not\in {\mathcal X}_4^{0,3}(\frg,J)$. 

Using the equations~\eqref{ecus-5dim}, 
one can easily check that $\omega^{\bar1\bar3\bar5}$ is $\bar{\partial}$-closed. 
Moreover, the following relations are satisfied: 
\renewcommand{\arraystretch}{1.8}
$$\begin{array}{lllllll}
\partial\omega^{\bar1\bar3\bar5} 
	\!\!\!&\!\!+\!\!&\!\!\! 
	\bar{\partial}\left( -\frac{1}{\bar{A}_{12}}\omega^{2\bar3\bar5}+\frac{b}{\bar{A}_{34}}\,\omega^{4\bar1\bar5} \right) 
	\!\!\!&\!\!=\!\!&\!\!\! 0 & & 
	\\[5pt]
	\!\!\!&\!\! \!\!&\!\!\! \partial\left( -\frac{1}{\bar{A}_{12}}\omega^{2\bar3\bar5}+\frac{b}{\bar{A}_{34}}\,\omega^{4\bar1\bar5} \right) 
	\!\!\!&\!\!+\!\!&\!\!\!
	\bar{\partial}\left( \frac{1}{\bar{A}_{12}}\omega^{25\bar3}-\frac{b}{\bar{A}_{34}}\omega^{45\bar1}\right)
	\!\!\!&\!\!=\!\!&\!\!\! 0.
\end{array}$$
Therefore, one indeed has that $\omega^{\bar1\bar3\bar5}\in {\mathcal X}_3^{0,3}(\frg,J)$. 
To see that $\omega^{\bar1\bar3\bar5}\not\in {\mathcal X}_4^{0,3}(\frg,J)$, 
we need to prove that there is no $\tau\in \Lambda^{3,0}(\frg^*)$ such that $\partial\gamma +\bar{\partial}\tau=0$, where 
$$\gamma=\frac{1}{\bar{A}_{12}}\omega^{25\bar3}-\frac{b}{\bar{A}_{34}}\omega^{45\bar1}.$$
Let us observe that any non-$\bar{\partial}$-closed $\tau\in \Lambda^{3,0}(\frg^*)$ can be written as
$$\tau=\big(
	t_{12}\,\omega^{12} +
	t_{13}\,\omega^{13} +
	t_{14}\,\omega^{14} +
	t_{23}\,\omega^{23} +
	t_{24}\,\omega^{24} +
	t_{34}\,\omega^{34}
	\big)
	\wedge
	\omega^{5},
$$
with $t_{ij}\in\mathbb C$ for every $1\leq i<j\leq 4$. Then,
\begin{eqnarray*}
\partial\gamma+\bar{\partial}\tau 
	&\!\!=\!\!& 
	t_{23}\,\omega^{123\bar1} -
	t_{13}\,\omega^{123\bar2} +
	t_{12}\,\omega^{123\bar3} +
	\left(t_{24}+\frac{b\,A_{12}}{\bar{A}_{34}}\right)\,\omega^{124\bar1} \\
	&&-\,  
	t_{14}\,\omega^{124\bar2} +
	b\,t_{12}\,\omega^{124\bar4} +
	t_{34}\,\omega^{134\bar1} -
	t_{14}\,\omega^{134\bar3} +
	b\,t_{13}\,\omega^{134\bar4} \\
	&&+\, 
	t_{34}\,\omega^{234\bar2} -
	\left(t_{24}+\frac{A_{34}}{\bar{A}_{12}}\right)\,\omega^{234\bar3} +
	b\,t_{23}\,\omega^{234\bar4}.
\end{eqnarray*}
Since the previous expression must be equal to zero, in particular one has
$$t_{24}+\frac{b\,A_{12}}{\bar{A}_{34}}=0,
\qquad
t_{24}+\frac{A_{34}}{\bar{A}_{12}}=0.$$
If we solve $t_{24}$ from the second equality above and then replace its value in the first one,
we get $b\,|A_{12}|^2-|A_{34}|^2=0$. However, this is not possible by the choice of
parameters made in the statement of the proposition.
\end{proof}

As a consequence of Proposition~\ref{FSS-5dim} and Proposition~\ref{existenciaAK-balanced-n-dim} for $m=4$, 
we obtain the following:

\begin{theorem}\label{FSS-AK-balanced-5dim}
Let $X$ be a 5-dimensional complex nilmanifold defined by \eqref{ecus-5dim} 
with $A_{12}, A_{34}\in \C\setminus\{0\}$ and $b\in(-1,0)$ satisfying $|A_{12}|^2+|A_{34}|^2=6(b+1)$.
Then, $X$ admits an AK metric $F'$ (which is also sG) and a balanced metric $F''$, and the Fr\"olicher spectral sequence of $X$ has $d_3\neq0$. 
\end{theorem}
\begin{proof}
Simply note that the inequality $b\, |A_{12}|^2-|A_{34}|^2\neq0$ in Proposition~\ref{FSS-5dim} can be omitted due
to the choice of $b<0$ together with $A_{12}, A_{34}\in \C\setminus\{0\}$. Some particular (rational) values of the parameters
in~\eqref{ecus-5dim} that satisfy the conditions in the statement of the theorem are $A_{12}=1$, $A_{34}=2$ and $b=-\frac{1}{6}$.
\end{proof}

\begin{remark}\label{sec6-first-examples}
The nilmanifolds constructed in Theorem~\ref{FSS-AK-balanced-5dim} seem to be the first examples of compact AK manifolds with balanced metrics and with FSS not degenerating at the third page. 
This is clear in the class of nilmanifolds by our study above (Propositions~\ref{Prop-CFG}--\ref{Prop-FSS-Kasuya-Stelzig}). 
 On the other hand, 
we note that any even-dimensional compact semisimple Lie group of rank two endowed  with 
a Samelson 
complex structure admits an AK (non-SKT) metric,  
and 
its canonical SKT metric is not AK \cite[Proposition 4.1]{FGV}. Moreover, compact semisimple Lie groups do not admit any balanced metric compatible with  a Samelson 
complex structure \cite[Corollary 5.1]{FGV}. Furthermore, Pittie \cite{Pittie-bull} proved that 
all left-invariant complex structures on semisimple groups of rank two satisfy $E_2=E_{\infty}$. It is worthy to remind that Pittie found in \cite{Pittie-bull} a particular left-invariant complex structure on {\rm SO(9)} with FSS not degenerating at the second page, but with $E_3=E_{\infty}$. 
\end{remark}

\medskip

\section*{Acknowledgments}
\noindent 
The authors are very grateful to Ettore Lo Giudice for useful comments.
This work has been partially supported by grant PID2023-148446NB-I00, funded by MICIU/AEI/10.13039/501100011033, 
and by grant E22-23R ``Algebra y Geometr\'{\i}a'' (Gobierno de Arag\'on/FEDER).

\end{document}